\documentclass[11pt]{article}
\usepackage[english]{babel}

\usepackage[letterpaper,top=2cm,bottom=2cm,left=3cm,right=3cm,marginparwidth=1.75cm]{geometry}
\usepackage[titletoc,title]{appendix}
\usepackage{amsmath,amssymb,mathrsfs,amssymb,amsthm}
\usepackage{enumerate}
\usepackage{cases}
\usepackage{graphicx}
\usepackage[colorlinks=true, allcolors=blue]{hyperref}
\usepackage{color}
\newcommand{\vep}{\varepsilon}
\newcommand{\R}{\mathbb R}
\newcommand{\CC}{\mathbb C}
\definecolor{HW}{rgb}{0,0,0}
\definecolor{HW1}{rgb}{0,0,0}
\definecolor{HW2}{rgb}{0,0,0}

\numberwithin{equation}{section}
\numberwithin{figure}{section}
\numberwithin{table}{section}
\title{Uniform Space and Time Behavior for Acoustic Resonators}
\author{Long Li and Mourad Sini}
\author{Long Li \thanks {RICAM, Austrian Academy of Sciences, A-4040, Linz, Austria (long.li@ricam.oeaw.ac.at)} \; and Mourad Sini \thanks{RICAM, Austrian Academy of Sciences, A-4040, Linz, Austria (mourad.sini@oeaw.ac.at)}}
\newtheorem{theorem}{Theorem}[section]
\newtheorem{lemma}{Lemma}[section]
\newtheorem{remark}{Remark}[section]

\begin{document}
\date{}
\maketitle
\begin{abstract}
\noindent We deal with the time-domain acoustic wave propagation in the presence of a subwavelength resonator given by a Minneart bubble. This bubble is small scaled and enjoys high contrasting mass density and bulk modulus. It is well known that, under certain regimes between these scales, such a bubble generates a single low-frequency (or subwavelength) resonance called Minnaert resonance. In this paper, we study the wave propagation governed by Minnaert resonance effects in time domain. We derive the point-approximation expansion of the wave field uniform in space and time. The dominant part is a sum of two terms.
\begin{enumerate}
 \item The first one, which we call the primary wave, is the wave field generated in the absence of the bubble. 
 \item The second one, which we call the resonant wave, is generated by the interaction between the bubble and the background. It is related to a Dirac-source, in space, that is modulated, in time, with a coefficient which is a solution of a $1$D Cauchy problem, for a second order differential equation, having as  propagation and attenuation parameters the real and the imaginary parts, respectively, of the Minnaert resonance. 
\end{enumerate}
\noindent We show that the evolution of the resonant wave remains valid for a large time of the order $\epsilon^{-1}$, where $\epsilon$ is the radius of the bubble, after which it collapses by exponentially decaying. Precisely, we confirm that such a resonant wave has a life-time inversely proportional to the imaginary part of the related subwavelength resonances, which is in our case given by the Minnaert one. 
In addition, the real part of this resonance fixes the period of the wave.


\vspace{.2in}
\end{abstract}

\section{Introduction and Statement of the Main Results}
We deal with the acoustic wave propagation in the time-domain in the presence of small scaled inhomogeneities. In particular, we are interested in resonant small scaled inhomogeneities, i.e. subwavelength resonators in short.   
To describe such inhomogeneities, let $y_0$ be any fixed point in $\R^3$ and, for any $\vep>0$, define $\Omega_\vep:= \{x: x=y_0+\vep(y-y_0), y\in \Omega\}$ and $\Gamma_\vep:=\partial \Omega_\vep$. Here, $\Omega\subset \mathbb R^3$ is an open bounded and connected domain with a $C^2$-smooth boundary $\Gamma:=\partial \Omega$. Let $\Omega_\vep \subset \R^3$ denote a micro-bubble embedded in the homogeneous background medium. The acoustic properties of the background are the mass density $\rho_0$ and the bulk modulus $k_0$ while the one of the perturbed medium are characterized by the mass density $\rho_\vep$ and the bulk modulus $k_\vep$. Consider the following two models
\begin{align}
& \frac{1}{k_0}\partial_{tt} v^f - \nabla \cdot \frac{1}{\rho_0} \nabla v^f = f \quad \textrm{in}\; \R^3 \times \R_+,\label{eq:9}\\
& v^f(x,0) = 0, \quad \partial_t v^f(x,0) = 0, \quad \textrm{for}\; x \in \R^3\label{eq:20}
\end{align}
which describes the acoustic wave propagating in the background homogeneous medium and 
\begin{align}
& \frac{1}{k_\vep}\partial_{tt} u^f - \nabla \cdot \frac{1}{\rho_\vep} \nabla u^f = f  \quad \textrm{in}\; \R^3 \times \R_+, \label{eq:12}\\
& u^f(x,0) = 0, \quad \partial_t u^f(x,0) = 0, \quad \textrm{for}\; x\in \R^3,\label{eq:21}
\end{align}
describing the wave propagation in the perturbed medium, respectively. We are interested in waves generated by  $\Omega_\epsilon$ as a subwavelength resonator. This inhomogeneity can generate resonances in one of the following situations regarding the scales of $\rho_\vep$ and $k_\vep$ restricted to $\Omega_{\epsilon}$. 
\begin{enumerate}
\item The mass density $\rho_\vep$ is moderate valued while the bulk modulus $k_\vep$ is large. Under scales of the form $k_\vep \sim \epsilon^2$, we have existence of a sequence of subwavelength resonances. These resonances are related to the eigenvalues of the volume Newtonian operator. We call them body resonances. Such a sequence of resonances was observed and used in \cite{AS-19, DGS-21,MMS-18} in the time harmonic regime. 

\item Both the mass density $\rho_\vep$ and the bulk modulus $k_\vep$ are large. Under scales of the form  $\rho_\vep \sim \epsilon^2$ and $k_\vep \sim \epsilon^2$, we have existence of the Minnaert resonance which is related to the eigenvalue $1/2$ of the surface double-layer operator (i.e. the Neumann-Poincare operator). As its eigenfunction is of the form of a single-layer potential, such a resonance is called surface resonance.  Such a resonant frequency was observed and used in \cite{AZ-18, AZ-17, DGS-21, FH, MPS} in the time harmonic regime.
\end{enumerate}

The $\vep^2$ scale of the mass density or the bulk modulus mentioned above ensures that the generated subwavelength resonance is of order $O(1)$. This is significant for applications in the field of metamaterials, as it allows for the manipulation of moderate resonance frequencies to design effective dispersive media. We refer to \cite{ACCS-20, AFZ-17, AZ-17,C-M-P-T-1} for the time harmonic case and to \cite{C-M-P-T-2,MS-241, MS-242} for the time domain case.

In the current work, we focus mainly on the second situation, namely the Minnaert bubble, for two reasons. {\color{HW} The first one is the breath of applications of Minnaert bubbles, see for instance, \cite{P_0,P_1, P_2, P_3, P_4, P_5}}. The second one is that the analysis is more involved since we have to handle both the two operators appearing in the used Lippmann-Schwinger equations. Our arguments go similarly with less efforts to the first situation.
Therefore, from now on, we assume the coefficients $\rho_\vep$ and $k_\vep$ to be globally defined as follows:
\begin{align*}
\rho_\vep(x) := 
\begin{cases}
\rho_0,   &  x \in \R^3 \backslash \Omega_\vep, \\
{\rho_1}{\vep^{2}},  & x \in \Omega_\vep,
\end{cases}
\quad 
k_\vep(x) := 
\begin{cases}
k_0,            &  x \in \R^3 \backslash \Omega_\vep, \\
k_1 \vep^{2},  & x \in \Omega_\vep,
\end{cases}
 \end{align*}
where $\rho_0, k_0, \rho_1$ and $k_1$ are all positive real numbers.

\noindent We set $c_0 := \sqrt{k_0/\rho_0}$ and $c_1:= \sqrt{k_1/\rho_1}$ to denote the wave speeds in the background homogeneous medium and inside the bubble respectively. 
Let
\begin{align} \label{eq:45}
\omega_M:= \sqrt{\frac {{\color{HW}\mathrm{cap}(\Omega)} k_1}{|\Omega|\rho_0}}
\end{align}
denote the related Minnaert frequency generated by the micro-bubble,
where ${\color{HW}\mathrm{cap}(\Omega)}$, defined by 
\begin{align} \label{eq:74}
{\color{HW}\mathrm{cap}(\Omega)}:= \int_\Gamma \left(S^{-1}_01\right)(x) d\sigma(x),
\end{align}
represents the capacitance of $\Omega$. Here, $S_0^{-1}$ denotes the inverse of the single layer boundary operator with a kernel of $1/{4\pi |x-y|}$. Furthermore, in order to state our results, we introduce some function spaces.
Given a Banach Space $X$, we denote by
\begin{align*}
{\color{HW}W_\beta^{p,q}}(\R_+;X) := \bigg\{f \in \mathcal S_+'(X): \|f\|^q_{{\color{HW}W_\beta^{p,q}}(\R_+;X)}:= \sum^p_{l=0} \int_{\R}\|(1+t)^{\beta}\partial^l_t f(\cdot, t)\|^q_X dt <\infty \bigg\}.
\end{align*}
Here, $\mathcal S'_+(X)$ denotes the space of $X$-valued tempered distribution in $\R$ having support in $\R_+$, $q\in \mathbb N$ and $p,\beta \in \mathbb N_0$ with $\mathbb N_0:= \mathbb N \cup \{0\}$. For simplicity, ${\color{HW}W_\beta^{p,2}}(\R_+;X)$ is also denoted by $H_0^p(\R_+;X)$. Given $\alpha \in \R$, let the weighted space $L_{-\alpha}^2(\R^3)$ be defined by 
\begin{align*}
L_{-\alpha}^2(\R^3) :=\left\{\phi\in L^2_{\textrm{loc}}\left(\R^3\right): (1+|x|^2)^{-\frac\alpha2} \phi(x) \in L^2\left(\R^3\right)\right\}.
\end{align*}

\noindent The main result of this work is as follows.

\begin{theorem}\label{th:1}
Let $\vep, T > 0$. Suppose that $f\in {\color{HW}W_1^{3,2}}(\R_+; L_{\alpha}^2(\R^3)) \cap  H_0^{16}\left(\R_+; L_{\alpha}^2(\R^3)\right)$ with $\alpha > 3/2$. Then, we have the asymptotic expansion 
\begin{align}
&u^f (x,t) - v^f (x,t) = \notag\\
&\frac{i\omega_M\rho_0|\Omega|}{8\pi k_1|x-y_0|} \vep \int^{t-c_0^{-1}|x-y_0|}_0 \left(e^{-iz_M^{-}(\vep)(t-c^{-1}_0|x-y_0|-\tau)} - e^{-iz_M^{+}(\vep)(t-c^{-1}_0|x-y_0|-\tau)} \right) \partial_{tt} v^f(y_0,\tau) d\tau\notag \\
&+ \textrm{Res}(x,t), \label{eq:120}  
\end{align}
where $z^{\pm}_M(\vep):= \pm \omega_M - i\vep {\color{HW}\mathrm{cap}(\Omega)} \omega^2_M/({8\pi c_0}) $ and  $\textrm{Res}(x,t)$ satisfies
\begin{align} \label{eq:121}
\|\textrm{Res}(\cdot,t)\|_{L_{-\alpha}^2(\R^3)} \le C\vep^{\frac{3}2}\left(\|f\|_{H_0^{16}\left(\R_+; L_{\alpha}^2(\R^3)\right)} + \|f\|_{{\color{HW}W_1^{3,2}}(\R_+;L_{\alpha}^2(\R^3))}\right), \quad t \in \left(0, {T}{\vep^{-1}}\right],  
\end{align}
as $\vep \rightarrow 0$. Here, $\omega_M$ is a Minnaert frequency given by \eqref{eq:45} and $C$ is a positive constant independent of $\vep$ and $f$. 
\end{theorem}

Let us make the following comments on the above result.
\begin{enumerate}
\item First, let us recall that in \cite{AZ-18} it was shown that the  acoustic resonator given by the gas bubble, enjoying both small mass density and bulk modulus, generate two resonances (Minnaert resonances). There, the resonance, called also scattering resonance frequency, is defined as the frequency for which the related system of integral equations (using layer potentials) is not injective. Later on, it was shown in \cite{LS-04} that these scattering resonant frequencies are the actual poles (unique poles) of the natural Hamiltonian related to the wave operator. The values $z^{\pm}_M(\vep)= \pm \omega_M - i\vep {\color{HW}\mathrm{cap}(\Omega)} \omega^2_M/({8\pi c_0})$, defined and used in the above theorem, are the dominant parts of these two poles (see \cite[Lemma 5.1]{LS-04}). 

\item Observe that the dominant part of (\ref{eq:120}) is the solution of the following Cauchy problem:
\begin{align*}
&\frac{1}{k_0}\partial_{tt} u^f_{\textrm{dom}}(x,t) - \nabla \cdot \frac{1}{\rho_0} \nabla u^f_{\textrm{dom}}(x,t) = f(x,t) + a(t) \delta(x-y_0) \quad \textrm{for}\; (t,x) \in  \R_+ \times  \R^3,\\
& u_{\textrm{dom}}^f(x,0) = 0, \quad \partial_t u_{\textrm{dom}}^f(x,0) = 0, \quad \textrm{for}\; x\in \R^3,
\end{align*}
where the coefficient $a(t)$ solves
\begin{align}\label{Cauchy-problem-introduction}
& \partial_{tt} a(t) + \vep \frac{{\color{HW}\mathrm{cap}(\Omega)} \omega^2_M}{4\pi c_0} \partial_t a(t) + \left(\omega^2_M + \vep^2\frac{{\color{HW}\mathrm{cap}^2(\Omega)}\omega^4_M}{64\pi^2c^2_0}\right) a(t) \notag \\ 
&\qquad\qquad\qquad \qquad\qquad \qquad = \int_{\R^3} -\frac{{\vep \omega^2_M \rho_0|\Omega|}\partial_{tt}f(t-c_0^{-1}|y_0-y|)}{4\pi k_1 |y_0-y|}dy, \quad t\in \R_+,\nonumber\\
& a(0) = 0, \quad \partial_t a(0) = 0.
\end{align}
This wave is a sum of two parts. The first part, related to the source $f$, is nothing but the wave generated by the background in the absence of the bubble, that is, $v^f$. The second one, related to source, i.e. $a(\cdot) \delta(\cdot-y_0)$, is the resonant wave generated by the bubble, which is denoted by $u^f_{\textrm{reson}}$. The latter wave has a form of a point-source with time-modulation $a(\cdot)$ solving the $1D$ Cauchy problem (\ref{Cauchy-problem-introduction}). {\color{HW} Indeed, invoking the solution formula for the second-order ODE with zero initial conditions (see Lemma \ref{le:5}), and using the formula 
\begin{align*}
\partial_{tt} v^f (y_0, t) = \rho_0 \int_{\Omega}\frac{\partial_{tt}f(t-c_0^{-1}|y_0-y|)}{4\pi|y_0-y|} dy,
\end{align*}
it follows that the solution to \eqref{Cauchy-problem-introduction} admits the representation
\begin{align*}
a(t) = \frac{i\omega_M\rho_0|\Omega|}{8\pi k_1} \vep \int^{t}_0 \left(e^{-iz_M^{-}(\vep)(t-\tau)} - e^{-iz_M^{+}(\vep)(t-\tau)} \right) \partial_{tt} v^f(y_0,\tau) d\tau, \quad t\in \R_+.
\end{align*}
Thus, the convolution of the background Green's function and the time-modulated point source $a(t)\delta(x-y_0)$ directly yields the leading term on the right-hand side of \eqref{eq:120}}. Furthermore, we note that $u^f_{\textrm{reson}}$ can be rewritten as 
\begin{align}
&u^f_{\textrm{reson}}(x,t) := u_{\textrm{dom}}^f(x,t) - v^f(x,t) = \frac{-\vep\omega_M\rho_0|\Omega|}{4\pi k_1|x-y_0|} \notag \\
& \qquad \; \int^{t-c_0^{-1}|x-y_0|}_0 \sin(\omega_M(t-c_0^{-1}|x-y_0|-\tau))e^{-\frac{{\color{HW}\mathrm{cap}(\Omega)} \omega^2_M}{8\pi c_0}\vep\left(t-c_0^{-1}|x-y_0|-\tau\right)}\partial_{tt}v^f(y_0,\tau)d\tau. \notag
\end{align}
We see that the Minnaert frequency $\omega_M$, which is the dominant real part of the Minnaert scattering resonance, characterizes the period of the resonant wave tail, i.e. $u^f_{\textrm{reson}}$, while the dominant part of the imaginary part, i.e. $-\vep{{\color{HW}\mathrm{cap}(\Omega)} \omega^2_M}/(8\pi c_0)$, characterizes the life-time, see also the next comment. To the best of our knowledge, this is the first {\color{HW}{attempt}} to characterize the Minnaert resonance effects, including both its real and imaginary parts, on the space and time behavior of the wave propagation, generated by the source spanning all frequencies, in time domain.
It is worth mentioning that \cite{APL-22, BMV-21} derived {\color{HW2} a representation of the field as a discrete sum of modes oscillating at complex frequencies} outside the plasmonic resonator, for the truncated inverse Fourier transform of the wave field in the low-frequency regime. {\color{HW2} This is commonly referred to as a quasi-modes expansion in the physics literature.}
Compared to  \cite{APL-22, BMV-21},
the asymptotic expansion obtained Theorem \ref{th:1} is applicable to the original wave field uniform in the entire space and over a large time interval with a length of order $1/\vep$.

\item For a better understanding of the behavior of the resonant wave, we assume that the source $f$ is compactly supported both in time and space. In this case, we can clearly see there are three characteristic factors of the resonant wave $u^f_{\textrm{reson}}$ for fixed source point $y_0$ and receiver point $x$. 
\begin{enumerate}
\item Birth time. {\color{HW2} If $t<c^{-1}_0|x-y_0|$}, then this resonant wave is fully zero, as $v^f(y_0, \cdot)$ is causal. Under reasonable conditions on the source $f$, this wave is not vanishing, for at least immediately, after $t^*:=c^{-1}_0|x-y_0|$. We call $t^*$ the birth time of this wave.
\item Life time. As discussed above $t^{**}:= {8\pi c_0}/(\vep {\color{HW}\mathrm{cap}(\Omega)} \omega^2_M)$ is the start collapsing time of this wave. We call it the life-time.
\item The period of propagation. We observe that $\omega_M$ describes the period of propagation of this wave.
\end{enumerate}
These characteristic factors have signatures of the background where the bubble is located. Characterizing the values of these factors in the case of heterogeneous background media can have very important applications in inverse problems for imaging modalities using contrast agents. For example, the birth time is nothing but the travel time (the time needed to a wave to travel from the $x$ to $y_0$ in a medium with wave speed $c_0$). Therefore if we measure the wave field before and then after injecting such resonators, then we can recover this travel time. Using the Eikonal equation, we can recover the wave speed at the location of the resonators. But we can also recover the period $\omega_M$ and estimate the life-time $t^{**}$. These factors can give us other information on the medium. For example the mass density can be recovered from $\omega_M$. The full analysis of these properties, as related to inverse problems in imaging, will be reported in a future investigation. But, the reader can already see \cite{SS-2024}  for the use of the birth time to derive the travel time function and apply it to ultrasound imaging using bubbles as contrast agents. 

\item Assuming the source $f$ to be compactly supported both in time and space, 
we observe that, in much larger times, i.e. $t\;\epsilon \gg 1$, the dominant part of (\ref{eq:120}) will be exponentially decaying in time and hence this term will be lost in the reminder term. Therefore, the time-threshold $t\sim \epsilon^{-1}$ is the limit where  (\ref{eq:120})-(\ref{eq:121}) makes sense. In this sense, this time-threshold is optimal and it characterizes the life-time of the resonant wave $u_{\textrm{reson}}^f$ determined by the Minnaert resonance, as mentioned earlier. Related to this result, let us cite the recent and interesting work \cite{MP} on the wave dynamics for body waves, i.e. the regime where the mass density $\rho_\vep$ is moderate valued while the bulk modulus $k_\vep$ is large. They derived partial results with large time behavior of the form $\epsilon^{-r}$ with $r\in (0, 1/11)$, which is hence still away from the actual life-time. Their approach, which is different from our approach (that we describe below), is based on functional calculus, Laplace transform and resolvent estimates of the free operator.

\item The assumption $f \in {\color{HW}W_1^{3,2}}(\R_+; L_{\alpha}^2(\R^3))$ ensures the time integrability of $\partial_{tt}v^f$ at $y_0$ over the time interval $(0,T\vep^{-1})$ (see \eqref{eq:30}), giving sense to the resonant wave $u^f_{\textrm{reson}}$ at each point in time. When the source $f$ is compactly supported both in time and space, this assumption becomes unnecessary, as $v^f$, at $y_0$, is compactly supported in time. In such case, the $H_0^{16}\left(\R_+; L_{\alpha}^2(\R^3)\right)$-norm of $f$ is sufficient to control the remainder term $\textrm{Res}(x,t)$ in \eqref{eq:121}.

\end{enumerate}

\noindent It is worth mentioning that the high time regularity assumption of $f$ in Theorem \ref{th:1} is based on two main reasons.
\begin{enumerate}
\item First, it facilitates the derivation of a priori estimates of the scaled perturbed wave over a large time period, which are crucial for the proof of Theorem \ref{th:1}. To achieve this, we propose a new strategy similar to bootstrapping. Specifically, we first apply the Fourier-Laplace transform, in weighted-in-time spaces, and its inverse, in terms of the Fourier-Laplace parameter $s$ with positive real part to obtain an initial estimate of the scaled wave. However, this technique will lead to less favorable estimates for larger times. To address this issue, with the aid of the Lippmann-Schwinger equation, we then iteratively improve the order of the estimate of the scaled wave with respect to $\vep$ over the time interval of order $\vep^{-2}$, albeit at the cost of higher time regularities (see Lemma \ref{le:1}). This novel approach allows us to estimate the lifetime of the wave only by investigating the low resonance frequencies, i.e. the Minnaert resonance in this case. Therefore, it avoids investigating the high resonance frequency's behavior of the wave fields with respect to $\vep$, which remains unclear. 

\item Second, solving a system of five-order differential equations is necessary to capture the life time information of the resonant wave, encoded by the imaginary part of the Minnaert resonance, over a large time period of the order $\vep^{-1}$.
\end{enumerate}

{\color{HW} \noindent Such high time-regularity is consistent with typical ultrasound practice: sources are causal and smoothly windowed. For example, we may write $f = a(t)g(x)$ with $a$ a smooth pulse (e.g., Gaussian) gated to finite duration and $g$ a smooth spatial profile.}

\noindent For a fixed size of the bubble, i.e. fixed $\vep$, and a moderate contrast, in \eqref{eq:12}, the acoustic wave propagator determined by the bubble, has an equivalent form of $\partial^2_t - H_{\rho_\vep, k_\vep}$, where the Hamiltonian $H_{\rho_\vep,k_\vep}$ is defined by
\begin{align*}
&H_{\rho_\vep, k_\vep} \psi := {k_\vep}\nabla \cdot {\rho^{-1}_\vep} \nabla \psi \\
&\textrm{with the domain}\; D(H_{\rho_\vep,k_\vep}):= \left\{u\in H^1(\R^3): k_\vep \nabla \cdot {\rho^{-1}_\vep} \nabla u \in L^2(\R^3)\right\}.
\end{align*}
This Hamiltonian can be considered as a specific black box Hamiltonian, see \cite[section 1.3.1]{LS-04}. There is
a considerable literature on resonance expansions or energy decay properties of solutions to the Cauchy problem for the wave equation $\left(\partial^2_{tt} - \mathcal B \right) \psi = 0 $ with $\mathcal B$ being a black box Hamiltonian, see for instance \cite{CPV-99, J19, PV-99, S-01, T-Z-98, T-Z-00}, with the references
therein, and the book \cite{DM} for the related studies.
Analyzing the distribution of resonances of the black box Hamiltonian near the real axis plays an important role in the proofs of the aforementioned works. However, for the high contrast resonator case considered in this paper, the dependence of high resonance frequencies' distribution on $\vep$ is yet to be fully understood. This is why we propose a new approach that allows us to circumvent this issue while still achieving our goal--characterizing the propagation of the resonant wave generated by the Minnaert resonance, although the trade off entails the estimates over the time interval of order $\vep^{-1}$. Actually, it is very interesting to investigate how the wave behaves asymptotically for longer times, such as when $t\epsilon \gg 1$, and a different behavior may occur, possibly influenced by high resonance frequencies.
\bigskip

\noindent The remaining part is divided as follows. In Section \ref{sec:2}, we provide a detailed proof of Theorem \ref{th:1} using some a priori and pullback estimates, which are established in Section \ref{sec:3}. In Appendix \ref{sec:A}, we include few technical tools and estimates that are used in Sections \ref{sec:2} and \ref{sec:3}.

\section{Proof of Theorem \ref{th:1}}\label{sec:2}

{\color{HW}This section is devoted to proving Theorem \ref{th:1}. We begin with an outline of its proof.
\subsection{Outline of the proof of Theorem \ref{th:1}}
\begin{enumerate}
\item \textbf{Scaled Lippmann-Schwinger equation}. By applying a scaling transformation, we consider the scaled {perturbed} wave field $u_\vep^f(x,t)$, which is defined by
\begin{align} \label{eq:82}
u_\vep^f(x,t):= u^f(\vep (x-y_0) + y_0, \vep t), \quad x\in \R^3,\; t\in \R_+,
\end{align}
and proceed to establish its Lippmann-Schwinger equation, 
\begin{align}
u_\vep^f(x,t)- & v_\vep^f(x,t)  = -\left(\frac{1}{c^2_1} - \frac{1}{c^2_0}\right) \int_{\Omega}\frac{\partial^{2}_{t} u_{\vep}^f(y,t-c_0^{-1}|x-y|)}{4\pi|x-y|}dy \notag\\
& -\left(\frac{\rho_0}{\rho_1\vep^2}-1\right) \int_{\Gamma} \frac{\partial_{\nu}u_\vep^f(y,t-c_0^{-1}|x-y|)}{4\pi|x-y|}d\sigma(y) \notag\\
& - \left(\rho_0 - \rho_1\vep^2\right)\vep^2 \int_{\Omega} \frac{f_\vep(y, t-c^{-1}_0|x-y|)}{4\pi|x-y|}dy, \quad x\in \R^3\backslash\Gamma, \quad t\in \R_+. \label{eq:145}
\end{align}
Here, 
\begin{align}
&v_\vep^f(x,t):= v^f(\vep (x-y_0) + y_0, \vep t), \label{eq:87}\\
&f_\vep(x,t):= f(\vep(x-y_0) + y_0, \vep t), \quad x\in \R^3,\; t\in \R_+, \label{eq:149}
\end{align}
and $\nu$ denotes the outward normal to $\Gamma$.
\item  \textbf{Projection estimates of $\partial_\nu u^f_\vep$}.
To address the surface type integral, we split the interior normal derivative $\partial_\nu u^f_\vep$ into two parts,
\begin{align} \label{eq:146}
\partial_\nu u^f_\vep = \Lambda_\vep(t) S_0^{-1} 1 + \left[\partial_\nu u_\vep^f - \Lambda_\vep(t) S_0^{-1} 1\right]=: \mathcal P \partial_\nu u^f_\vep  + \mathcal Q \partial_\nu u^f_\vep, \quad \mathrm{in}\; \Gamma \times \R_+,
\end{align}
where
\begin{align} \label{eq:73}
\Lambda_\vep(t):= \frac{1}{\mathrm{cap}(\Omega)}\int_{\Gamma} \partial_\nu u^f_\vep(x, t) d\sigma(x). 
\end{align}
Here, $\mathcal P$ is the projection onto the eigenspace of the Neumann-Poincar\'{e} operator $K^*_0$ spanned by $S_0^{-1} 1$ (see \eqref{eq:84}). Here, $K^*_0$ is defined by
\begin{align*}
\left(K^*_{0}\phi\right)(x) := \int_{\Gamma}{{4\pi|x-y|^{-3}}\nu(x)\cdot(x-y) \phi(y)}d\sigma(y), \quad x\in \Gamma.
\end{align*}
We note that $\Lambda_\vep(t)$ is precisely the projection coefficient of $\mathcal P \partial_\nu u^f_\vep$ onto $S^{-1}_0 1$.

Applying the interior normal derivatives to both sides of \eqref{eq:145}, and after multiplying both sides by $\rho_1\vep^2/\rho_0$, we obtain
\begin{align}
& \frac{1}2 \left(1 + \frac{\rho_1\vep^2}{\rho_0} \right) \partial_{\nu} u_\vep^f(x,t)\notag \\
& = \frac{\rho_1\vep^2}{\rho_0} \partial_{\nu} v_\vep^f(x,t) - \left(\frac{1}{c^2_1} - \frac{1}{c^2_0}\right)\frac{\rho_1\vep^2}{\rho_0}  \partial_{\nu}\int_{\Omega}\frac{\partial^{2}_{t} u_{\vep}^f(y,t-c_0^{-1}|x-y|)}{4\pi|x-y|}dy + \left(1-\frac{\rho_1\vep^2}{\rho_0} \right) \notag \\
&\int_{\Gamma}\left(\frac{1}{c_0}\partial_t\partial_{\nu}u_\vep^f(y,t-c_0^{-1}|x-y|) + \frac{\partial_{\nu}u_\vep^f(y,t-c_0^{-1}|x-y|)}{|x-y|}\right)\frac{(x-y)\cdot \nu(x)}{4\pi|x-y|^2} d\sigma(y)  \notag\\
&-\left(\rho_0 - \rho_1\vep^2\right)\frac{\rho_1\vep^4}{\rho_0} \partial_\nu \int_{\Omega} \frac{f_\vep(y, t-c^{-1}_0|x-y|)}{4\pi|x-y|}dy, \quad x\in \Gamma, \quad t\in \R_+. \label{eq:150} 
\end{align}
Taking a fifth-order Taylor expansion of $\partial_\nu u^f_\vep$ and a fourth-order Taylor expansion of $\partial_t \partial_\nu u^f_\vep$ in $t$, \footnote{An explanation why we need to take $16$ order time derivatives is provided in Remark \ref{Why-we-need-higher-regularity}.} and using the decomposition \eqref{eq:146} of $\partial_\nu u^f_{\vep}$, the above surface integral can be rewritten as
\begin{align}
&\Lambda_\vep(t)\left(K^*_0 S_0^{-1} 1 \right)(x) + \left(K^*_0 \mathcal Q \partial_\nu u^f_\vep(\cdot,t)\right)(x) + \sum^5_{l=2} a_l(x) \partial^l_t\Lambda_\vep(t) + \mathrm{Rem}_1(x,t).\label{eq:151}
\end{align}
Here, 
\begin{align*}
a_l(x):=\frac{(-1)^l(1-l)}{c^l_0 l!}\int_{\Gamma} \frac{(x-y)\cdot \nu(x)}{4\pi|x-y|^2} {|x-y|^{l-1}}\left(S^{-1}_01\right)(x)d\sigma(x),
\end{align*}
and 
\begin{align*}
& \mathrm{Rem}_1(x,t):=  \int_{\Gamma} \frac{(x-y)\cdot \nu(x)}{4\pi|x-y|^2}\bigg[\sum^{5}_{l=2}\frac{(-1)^l(1-l)}{c^l_0 l!}|x-y|^{l-1} \mathcal Q\partial^{l}_{t} \partial_{\nu} u_\vep^f(y,t)\bigg]d\sigma(x) \\
& + \int_{\Gamma}\frac{(x-y)\cdot \nu(x)}{4\pi|x-y|^2}\int^{t-c^{-1}_0|x-y|}_{t}\partial^{6}_{t}\partial_{\nu} u_\vep^f(y,\tau)\bigg[\sum^5_{j=4}\frac{(t-c^{-1}_0|x-y|-\tau)^{j}}{j!|x-y|^{j-4} c_0^{5-j}}\bigg]d\tau d\sigma(y). \notag
\end{align*}
Similarly, the two volume integrals in \eqref{eq:150} satisfy
\begin{align}
&\partial_\nu \int_{\Omega}\frac{\partial^{2}_{t} u_{\vep}^f(y,t-c_0^{-1}|x-y|)}{4\pi|x-y|} - \frac{\partial^2_t u^f_\vep(y, t)}{4\pi|x-y|} dy= \partial_\nu\sum^3_{l=2} \int_\Omega\frac{(-1)^l\partial^{2+l}_{t} u_{\vep}^f(y,t)|x-y|^{l-1}}{4\pi c^l_0 l!}dy\notag\\
& + \partial_\nu \int_{\Omega} \int^{t-c^{-1}_0|x-y|}_{t}\partial^{6}_{t} u_\vep^f(y,\tau)\frac{(t-c^{-1}_0|x-y|-\tau)^{3}}{4\pi |x-y| 3!} d\tau dy=:\mathrm{Rem}_2(x,t), \label{eq:152}\\
&\mathrm{and}\; \partial_\nu \int_{\Omega} \frac{f_\vep(y, t-c^{-1}_0|x-y|)}{4\pi|x-y|}dy - \partial_\nu \int_{\Omega} \frac{f_\vep(y, t)}{4\pi|x-y|}dy = \partial_\nu \int_{\Omega} \frac{f_\vep(y, t)|x-y|}{8\pi c_0^2}dy \notag \\
&+ \partial_\nu \int_{\Omega}\int^{t-c^{-1}_0|x-y|}_{t}\partial^{3}_{t} f_\vep(y,\tau)\frac{(t-c^{-1}_0|x-y|-\tau)^{2}}{8\pi|x-y|} d\tau dy=:\mathrm{Rem}_3(x,t). \label{eq:153}
\end{align}
Furthermore, we have the following two integral identities (see \eqref{eq:132} and \eqref{eq:134}) 
\begin{align*}
&\int_{\Gamma}\partial_\nu v^f_\vep(x,t) d\sigma(x) - \rho_0\vep^2 \int_{\Gamma} \partial_\nu \int_{\Omega} \frac{f_\vep(y,t)}{4\pi|x-y|}dy d\sigma(x) = \frac{1}{c_0^2}\int_{\Omega} \partial^{2}_{t} v_\vep^f(y,t) dy,\\
&\mathrm{and}\;\frac{\rho_1}{\rho_0|\Omega|}\left[\int_{\Gamma}\partial_\nu u^f_\vep(x,t)d\sigma(x) - \left(\frac{1}{c^2_0} - \frac{1}{c^2_1}\right)\int_{\Gamma} \partial_\nu \int_{\Omega} \frac{\partial^2_t u_\vep(y,t)}{4\pi|x-y|}dy d\sigma(x)\right]\notag\\ 
&\qquad = \frac{\omega_M^2}{c^2_0}\Lambda_\vep(t) - \left(1-\frac{c^2_1}{c^2_0}\right) \frac{\rho^2_1}{\rho_0|\Omega|}\vep^4\int_\Omega f_\vep(y) dy.
\end{align*}
Therefore, taking the projection coefficient of both sides of \eqref{eq:150} onto $S_0^{-1}1$, and using the fact that $\left(1/2 - K^*_0\right) S_0^{-1} 1 = 0$ and $\mathcal P K^*_0 \mathcal Q = 0$, we conclude from \eqref{eq:151}, \eqref{eq:152} and \eqref{eq:153} that  
\begin{align}
-\left(1-\frac{\rho_1\vep^2}{\rho_0} \right)\sum^5_{l=2} {\partial^l_t\Lambda_\vep(t)}\int_{\Gamma}a_l(x)d\sigma(x)  &+ \frac{\omega_M^2\vep^2|\Omega|}{c^2_0} \Lambda_\vep(t)\notag\\
& = \frac{\rho_1 \vep^2}{\rho_0c_0^2}\int_{\Omega} \partial^{2}_{t} v_\vep^f(y,t) dy + \mathrm{Rem}(t), \label{eq:154}
\end{align}
where $\mathrm{Rem}(t)$ is defined by
\begin{align*}
\mathrm{Rem}(t)&:= \int_{\Gamma} \left(1-\frac{\rho_1\vep^2}{\rho_0} \right)\mathrm{Rem}_1(x,t)d\sigma(x) - \int_{\Gamma}\left(\frac{1}{c^2_1} - \frac{1}{c^2_0}\right)\frac{\rho_1\vep^2}{\rho_0} \mathrm{Rem}_2(x,t) d\sigma(x) \\
&+ \frac{\rho^2_1\vep^6}{\rho_0}\int_{\Gamma} \partial_\nu \int_{\Omega} \frac{f_\vep(y,t)}{4\pi|x-y|}dy d\sigma(x)  - \int_{\Gamma}\left(\rho_0 - \rho_1\vep^2\right)\frac{\rho_1\vep^4}{\rho_0}\mathrm{Rem}_3(x,t)d\sigma(x)\\
&+ \left(1-\frac{c^2_1}{c^2_0}\right) \frac{\rho^2_1 \vep^6}{\rho_0}\int_\Omega f_\vep(y) dy.
\end{align*}
With the aid of the identities (see \eqref{eq:147} and \eqref{eq:148})
\begin{align*}
\int_{\Gamma} a_2(x)d\sigma(x) = -\frac{|\Omega|}{c^2_0}\; \mathrm{and}\; \int_{\Gamma} a_3(x)d\sigma(x) = \frac{\mathrm{cap}(\Omega)|\Omega|}{4\pi c^3_0},
\end{align*}
dividing $(1-\rho_1\vep^2/\rho_0)|\Omega|/c^2_0$ on both sides of \eqref{eq:154}, 
we obtain
\begin{align} \label{eq:155}
\left(\sum^5_{l=2}\eta_{l-1}\partial^l_t \Lambda_\vep(t)\right) + \gamma_\vep \Lambda_\vep(t) = \int_{\Omega} \frac{\rho_1\vep^2 \partial^{2}_{t} v_\vep^f(y,t)}{(\rho_0-\rho_1\vep^2)|\Omega|} dy + \frac{c^2_0\mathrm{Rem}(t)}{(1-{\rho_1\vep^2/\rho_0})|\Omega|}.
\end{align}
Here, the constants $\gamma_\vep$ and $\eta_l$ ($l \in \{1,2,3,4\}$) are defined by 
\begin{align}
&\gamma_\vep:=\frac{\omega_M^2\vep^2}{1-{\rho_1\vep^2/\rho_0}} \label{eq:116}
\end{align}
and
\begin{align}
&\eta_1 := 1, \quad \eta_2:= -\frac{\mathrm{cap}({\Omega})}{{4\pi c_0}}, \label{eq:114} \\
&\eta_{l}:= -\frac{(-1)^{l}l}{|\Omega|c^{l-1}_0(l+1)!} \int_{\Gamma}\int_{\Gamma} \nu(x) \cdot {(x-y)}{|x-y|^{l-2}} 
\left(S^{-1}_01\right)(y) d\sigma(x) d\sigma(y), \quad l\in\{3,4\}, \label{eq:115}
\end{align}
respectively. On the other hand, differentiating \eqref{eq:145} gives the Lippmann-Schwinger equations for the time derivatives of $u^f_\vep$ (see \eqref{eq:1}); thus the corresponding normal derivatives satisfy identities akin to \eqref{eq:150} (see \eqref{eq:59}), which in turn gives identities for its projections (see \eqref{eq:78}). For each $j\in\{1,2,3\}$, applying the similar argument used to derive \eqref{eq:155} to $\partial^j_tu^f_\vep$ yields
\begin{align*}
&\sum^{4-j}_{l=1} \eta_l\partial^{l+j+1}_{t}\Lambda_\vep(t) + \gamma_{\vep} \partial^{j}_t\Lambda_\vep(t) = \int_{\Omega} \frac{\rho_1\vep^2 \partial^{j+2}_{t} v_\vep^f(y,t)}{(\rho_0-\rho_1\vep^2)|\Omega|} dy + \mathrm{Rem}^{(j)}(t).
\end{align*}
Utilizing the solution formula for the above system of ordinary differential equations and \eqref{eq:155} (see statement \eqref{d3} of Lemma \ref{le:7}), we provide the desired asymptotic behavior of $\Lambda_\vep(t)$. Our a priori estimates of $u^f_\vep$ and $v^f_\vep$ over the long-time interval $(0,T\vep^{-2})$ (see Lemmas \ref{le:0} and \ref{le:1}) ensure that the contributions of $\mathrm{Rem}$ and $\mathrm{Rem}^{(j)}$ are negligible compared with the leading term. 

For the estimate of $\mathcal Q \partial_\nu u_\vep^f$, applying $\mathcal Q$ to \eqref{eq:150}, and using the identity $\mathcal Q K^*_0 \mathcal P = 0$, we deduce from \eqref{eq:151} that 
\begin{align*}
&\left(1-\frac{\rho_1\vep^2}{\rho_0} \right)\left( \mathcal Q\left(\frac{1}2 - K^*_0\right) \mathcal Q \partial_\nu u_\vep^f(\cdot,t)\right)(x) = \frac{\rho_1\vep^2}{\rho_0}\mathcal Q \left((\partial_\nu v_\vep^f - \partial_\nu u_\vep^f)(\cdot,t)\right)(x) \\
&+\left(1-\frac{\rho_1\vep^2}{\rho_0} \right)\left[ \sum^5_{l=2}\mathcal Q a_l(x) \partial^l_t\Lambda_\vep(t) + \left(\mathcal Q \mathrm{Rem}_1(\cdot,t)\right)(x)\right] + \mathrm{Res}(x,t).
\end{align*}
Here, $\mathrm{Res}$ is the function obtained after applying $\mathcal Q$ to the second and fourth right-hand-side terms of \eqref{eq:150}. Then, utilizing the invertibility of $ \mathcal Q\left(\frac{1}2 - K^*_0\right) \mathcal Q $ (see \eqref{eq:77}) and a priori estimates of $u^f_\vep$ and $v^f_\vep$, we can derive the asymptotic estimate of $\mathcal Q \partial_\nu u_\vep^f$.

\item \textbf{Return to the original field via pullback estimates}. Finally, we apply pullback transform to \eqref{eq:150}. Building on the established asymptotic estimates for $\mathcal P \partial_\nu u^f_\vep$ and $\mathcal Q \partial_\nu u^f_\vep$, we make use of the uniform pullback estimates (see Lemmas \ref{le:3}--\ref{le:9}) to obtain the uniform asymptotic properties of the original wave fields in space and time from the scaled ones.   
\end{enumerate}
We conclude this subsection by briefly outlining the structure of the remainder of this section. Section 2.2 and Section 2.3 introduce notations (functional spaces and auxiliary operators) and the scaled Lippmann-Schwinger equations, respectively.  Section 2.4 state the statements of the a priori and pullback estimates used in the proof of Theorem 1.1. Finally, the proof of Theorem 1.1 is provided in Section 2.5.
}

\subsection{Functional spaces and auxiliary operators} We begin by introducing some new notations. 
Let $X$ be a Banach space. 
For $\sigma \in \R_+$ and $p \in \mathbb N_0$, we define 
\begin{align*}
H_{0,\sigma}^p&(\R_+;X):=\left\{f\in \mathcal S'_+(X):\|f\|^2_{H^p_{0,\sigma}(\R_+;X)}:= \sum^p_{l=0} \int_{\R} e^{-2\sigma t}\|\partial^l_t f(\cdot, t)\|_X^2 dt < \infty\right\}. 
\end{align*}
Recalling the definition of ${\color{HW}W_\beta^{p,q}(\R_+;X)}$ with $\beta,p\in \mathbb N_0$ and $q\in \mathbb N$, we denote the restriction of functions in ${\color{HW}W_\beta^{p,q}(\R_+;X)}$ to the subinterval $I \subset \R_+$ by ${\color{HW}W_\beta^{p,q}(I;X)}$ with the following norm
\begin{align*}
\|f\|^q_{{\color{HW}W_\beta^{p,q}}(I;X)}:= \sum^p_{l=0} \int_{I}\|(1+t)^{\beta}\partial^l_t f(\cdot, t)\|^q_X dt.
\end{align*}
For simplicity, ${\color{HW}W_0^{0,1}(I;X)}$ and ${\color{HW}W_0^{0,2}(I;X)}$ are also denoted by $L^1(I;X)$ and $L^2(I;X)$, respectively. 
Define $\mathbb L^2(\Omega):= \left(L^2(\Omega)\right)^3$ with the inner product defined by the integral of the dot product of two functions over the domain $\Omega$. 
We proceed to introduce the following integral operators
\begin{align*}
&SL_{0}: H^{\frac12}(\Gamma) \rightarrow H_{\textrm{loc}}^{2}(\R^3 \backslash \Gamma),  \;\; \left(SL_{0}\phi\right)(x) := \int_{\Gamma} \frac{1}{4\pi|x-y|}\phi(y)d\sigma(y), \;\; x \in \R^3\backslash \Gamma,\\
&S_{0}: H^{-\frac 12}(\Gamma)\rightarrow H^{\frac12}(\Gamma), \quad \left(S_{0}\phi\right)(x) := \int_{\Gamma} \frac{1}{4\pi|x-y|}\phi(y)d \sigma(y),\quad x\in \Gamma, \\
&N_{l}: L^2(\Omega) \rightarrow H_{\textrm{loc}}^2(\R^3), \quad \left(N_{l}\phi\right)(x) := \int_{\Omega}{\color{HW}{4\pi|x-y|^{l-1}}}\phi(y)dy, \quad x\in \R^3,\\
&K^*_{l}: H^{-\frac12}(\Gamma)\rightarrow H^{\frac 12}(\Gamma), \quad \left(K^*_{l}\phi\right)(x) := \int_{\Gamma}{\color{HW}{4\pi|x-y|^{l-3}}\nu(x)\cdot(x-y) \phi(y)}d\sigma(y), \quad x\in \Gamma.
\end{align*}
Here, $\nu$ denotes the outward normal to $\Gamma$ and $l\in \mathbb N_0$. Properties of these operators and spaces can be found in \cite{WM-00}. Furthermore, for $z\in \CC_+:=\{z \in \mathbb C: \textrm{Im}(z)>0\}$, we define 
\begin{align} \label{eq:90}
\left(R_z\phi\right)(x):= \int_{\R^3} \frac{e^{iz|x-y|}}{4\pi|x-y|} \phi(y) dy, \quad x\in \R^3.
\end{align}
It is known (see \cite[Theorem 18.3]{KK-12}) that for any $\sigma_1,\sigma_2 > 1/2$ with $\sigma_1 + \sigma_2 >2$,
\begin{align} \label{eq:24}
\sup_{z\in \CC_+: |z| \le 1}\|R_{z}\|_{L_{\sigma_1}^2(\R^3), H_{-\sigma_2}^2(\R^3)} < +\infty,
\end{align}
and (see \cite[Proposition 1.2]{RT}) that for any $\sigma_3> 1/2$,
\begin{align} \label{eq:29}
\|R_z\|_{L_{\sigma_3}^2(\R^3), H_{-\sigma_3}^2(\R^3)} \le C \frac{1+|z|^{2}}{|z|}, \quad z\in \mathbb C_+.
\end{align}
Here, $C$ is a positive constant independent of $z$.
We denote the Green function corresponding to the wave operator $k_0^{-1}\partial_{tt} - \rho^{-1}_0 \Delta $ by
\begin{align}\label{eq:76}
G(x,t) := \rho_0\frac{\delta_0(t-c_0^{-1}|x|)}{4\pi|x|}\quad \textrm{in}\; \R^3\times \R,
\end{align}
where $\delta_0$ is the Dirac delta distribution.
From now on, $\mathbb I $ denotes an identity operator in various spaces, $T \in \R_+$ represents any fixed time, and the constants may be different at different places.

\subsection{Scaled Lippmann-Schwinger equation}
This subsection is devoted to establishing the Lippmann-Schwinger equation of $u_\vep^f(x,t)$, which is defined by \eqref{eq:82}.
We begin by recalling the well-posedness of equations  \eqref{eq:9}--\eqref{eq:20} and \eqref{eq:12}--\eqref{eq:21}.

\begin{lemma} \label{le:2}
Let $p \in \mathbb N$ and $\sigma \in \R_+$. Assume that $f\in H_{0,\sigma}^p \left(\R_+; L^2(\R^3)\right)$. Equations \eqref{eq:9}--\eqref{eq:20} and \eqref{eq:12}--\eqref{eq:21} both have unique solutions in  $H_{0,\sigma}^{p+1} \left(\R_+; H^1(\R^3)\right)$.
\end{lemma}

\begin{proof}
The unique solvability of equations \eqref{eq:12}--\eqref{eq:21} in $H_{0,\sigma}^{p+1} \left(\R_+; H^1(\R^3)\right)$ is proved in \cite[Theorem 2.2]{SS-2024} and \cite[Lemma 2.1]{MS-23}. In a similar manner, the well-posedness of equations \eqref{eq:9}--\eqref{eq:20} can also be derived.
\end{proof}

It is easy to verify that $v^f_\vep$ and $u^f_\vep$ satisfy the following scaled equations
\begin{align}
& \frac{1}{k_0}\partial_{tt} v_\vep^f - \nabla \cdot \frac{1}{\rho_0} \nabla v_\vep^f = \vep^2 f_\vep, \quad \textrm{in} \; \R^3 \times \R_+, \label{eq:15}\\
& v_\vep^f(x,0) = 0, \quad \partial_t v_\vep^f(x,0) = 0, \quad \textrm{for}\; x\in \R^3 \label{eq:16}
\end{align}
and 
\begin{align}
& \frac{1}{\widetilde k_\vep}\partial_{tt} u_\vep^f - \nabla \cdot \frac{1}{\widetilde \rho_\vep} \nabla u_\vep^f = \vep^2 f_\vep, \quad  \textrm{in}\; \R^3 \times \R_+, \label{eq:10}\\
& u_\vep^f(x,0) = 0, \quad \partial_t u_\vep^f(x,0) = 0, \quad \textrm{for}\; x\in \R^3, \label{eq:11}
\end{align}
respectively. Here, the scaled source $f_\vep$ is given by \eqref{eq:149}. 
and the scaled mass density $\widetilde \rho_\vep$ and the scaled bulk modulus $\widetilde k_\vep$ are given by  
\begin{align*}
\widetilde \rho_\vep(x) := 
\begin{cases}
\rho_0,    &  x \in \R^3 \backslash \Omega, \\
{\rho_1}{\vep^{2}},   & x \in \Omega,
\end{cases}
\quad 
\widetilde k_\vep(x) := 
\begin{cases}
k_0,          &  x \in \R^3 \backslash \Omega, \\
k_1 \vep^{2},  & x \in \Omega.
\end{cases}
\end{align*}
Clearly, given $f\in H_{0,\sigma}^p \left(\R_+; L^2(\R^3)\right)$ with $p\in \mathbb N_0$ and $\sigma \in \R_+$, we have that for each $t \in \R_+$ and $j\in\{l\in\mathbb N_0: l \le p\}$, 
\begin{align}
&\|\partial^{j}_t f_\vep(\cdot,t)\|_{L^2(\Omega)} \le C\vep^{j -\frac 32} \|\partial^j_t f(\cdot, \vep t)\|_{L^2(\Omega_\vep)} \label{eq:118}
\end{align}
and
\begin{align}
&\left\|\partial^j_t f_\vep \right\|_{L^2\left((0,t/\vep); L^2(\Omega)\right)} \le C\vep^{j - 2} \left\|\partial^j_t f \right\|_{L^2\left((0,t); L^2(\Omega_\vep)\right)}.  \label{eq:122}
\end{align}
Here, $C$ is a positive constant independent of $t$, $\vep$ and $f$. With the aid of Lemma \ref{le:2}, we see that, given $p\in \mathbb N$ and $\sigma \in \R_+$, $v^f_\vep, u^f_\vep \in H_{0,\sigma}^{p+1} \left(\R_+; H^1(\R^3)\right)$ when $f\in H_{0,\sigma}^p \left(\R_+; L^2(\R^3)\right)$.
Furthermore, since $v^f_\vep$ and $u^f_\vep$ solve equations \eqref{eq:15}--\eqref{eq:16} and \eqref{eq:10}--\eqref{eq:11}, respectively,
we readily obtain
\begin{align} \label{eq:49}
&\frac{1}{k_0}\partial_{tt} (u_\vep^f- v_\vep^f)- \nabla \cdot \frac{1}{\rho_0} \nabla (u_\vep^f- v_\vep^f) \notag \\
& \qquad \qquad = -\left(\frac{1}{\widetilde k_\vep} - \frac{1}{k_0}\right)\partial_{tt} u_\vep^f + \nabla \cdot \left(\frac{1}{\widetilde \rho_\vep} - \frac{1}{\rho_0}\right)\nabla u_\vep^f \quad \textrm{in}\; \R^3 \times \R.
\end{align}

With the help of \eqref{eq:49}, we proceed to derive the Lippmann-Schwinger equation of $u^f_\vep$ and its time derivatives in the following lemma, which plays an important role in the proof of Theorem \ref{th:1}.

\begin{lemma} \label{le:6}
Let $\vep >0$ and $u^f_\vep$ be given by \eqref{eq:87}. Given $p \in \mathbb N$ and $\sigma\in \R_+$, suppose that $f\in H_{0,\sigma}^p\left(\R_+; L^{2}(\R^3)\right)$. For each $j\in \{l \in \mathbb N_0: l\le p-1\}$, we have
\begin{align}
\partial^j_t u_\vep^f(x,t)- & \partial^j_t v_\vep^f(x,t)  = -\left(\frac{1}{c^2_1} - \frac{1}{c^2_0}\right) \int_{\Omega}\frac{\partial^{j+2}_{t} u_{\vep}^f(y,t-c_0^{-1}|x-y|)}{4\pi|x-y|}dy \notag\\
& -\left(\frac{\rho_0}{\rho_1\vep^2}-1\right) \int_{\Gamma} \frac{\partial^j_t \partial_{\nu}u_\vep^f(y,t-c_0^{-1}|x-y|)}{4\pi|x-y|}d\sigma(y) \notag\\
& - \left(\rho_0 - \rho_1\vep^2\right)\vep^2 \int_{\Omega} \frac{\partial^{j}_t f_\vep(y, t-c^{-1}_0|x-y|)}{4\pi|x-y|}dy, \quad x\in \R^3\backslash\Gamma, \quad t\in \R_+, \label{eq:1}
\end{align}
where $v^f_\vep$ is given by \eqref{eq:82}  and  $\partial^j_t\partial_{\nu}u_\vep^f(x,t)$ satisfies
\begin{align}
& \frac{1}2 \left(\frac{\rho_0}{\rho_1\vep^2} + 1 \right) \partial^j_t\partial_{\nu} u_\vep^f(x,t)\notag \\
& = \partial^j_t\partial_{\nu} v_\vep^f(x,t) - \left(\frac{1}{c^2_1} - \frac{1}{c^2_0}\right) \partial_{\nu}\int_{\Omega}\frac{\partial^{j+2}_{t} u_{\vep}^f(y,t-c_0^{-1}|x-y|)}{4\pi|x-y|}dy + \left(\frac{\rho_0}{\rho_1\vep^2}-1\right) \notag \\
& \bigg[\sum^{q_0}_{l=0} \frac{(-1)^l}{c^l_0 l!}\big(K^*_l\partial^{j+l}_t\partial_{\nu}u_\vep^f(\cdot,t)\big)(x) + \sum^{q_1-1}_{l=0} \frac{(-1)^l}{c^{l+1}_0 l!} \big(K^*_{l+1}\partial^{j+l+1}_t\partial_{\nu}u_\vep^f(\cdot,t)\big)(x) \notag\\
& + \int_{\Gamma}\frac{(x-y)\cdot \nu(x)}{4\pi|x-y|^3}\int^{t-c^{-1}_0|x-y|}_{t}\partial^{j+q_0+1}_{t} \partial_{\nu} u_\vep^f(y,\tau)\frac{(t-c^{-1}_0|x-y|-\tau)^{q_0}}{q_0!} d\tau d\sigma(y) \notag\\
& + c_0^{-1}\int_{\Gamma}\frac{(x-y)\cdot \nu(x)}{4\pi|x-y|^2}\int^{t-c^{-1}_0|x-y|}_{t}\partial^{j+q_1+1}_{t} \partial_{\nu} u_\vep^f(y,\tau)\frac{(t-c^{-1}_0|x-y|-\tau)^{q_1-1}}{(q_1-1)!} d\tau d\sigma(y)\bigg] \notag\\
&-\left(\rho_0 - \rho_1\vep^2\right)\vep^2\partial_\nu \int_{\Omega} \frac{\partial^j_t f_\vep(y, t-c^{-1}_0|x-y|)}{4\pi|x-y|}dy, \quad x\in \Gamma, \quad t\in \R_+. \label{eq:59}
\end{align}
Here, $q_0 \in \left\{l \in \mathbb N_0: l \le p-1-j\right\}$ and $q_1\in \left\{l\in \mathbb N: l \le p-j\right\}$.

\end{lemma}

\begin{proof}
We focus solely on the proof of case when $j=0$ since equation \eqref{eq:1} and \eqref{eq:59} for $j\ge1$ can be directly derived from the case $j=0$ by taking time derivatives. 

With the aid of equation \eqref{eq:49}, it can be seen that for $x\in \R^3\backslash\Gamma$ and $t\in \R_+$,
\begin{align}
u_\vep^f(x,t) - v_\vep^f(x,t) 
& = -\rho_0 \int_{\Omega} \left(\frac{1}{k_1\vep^2}-\frac{1}{k_0} \right)\frac{\partial_{tt} u_\vep^f(y,t-c_0^{-1}|x-y|)}{4\pi|x-y|}dy \notag \\
& +\left(\frac{1}{\rho_1\vep^2} - \frac{1}{\rho_0}\right) \textrm{div} \int_{\R} \int_{\Omega}G(x-y,t-\tau)\nabla u_\vep^f(y,\tau)dyd\tau.\label{eq:51}
\end{align}
Here, the Green function $G$ is given by \eqref{eq:76}.
Furthermore, it follows from Green formulas that 
\begin{align*}
& -\nabla_x \cdot \int_{\R} \int_{\Omega}G(x-y,t-\tau)\nabla u_\vep^f(y,\tau)dyd\tau 
= \rho_0 \int_{\Gamma} \frac{\partial_{\nu(y)}u_\vep^f(y,t-c_0^{-1}|x-y|)}{4\pi|x-y|}d\sigma(y) \\
&- \frac{\rho_0\rho_1}{k_1} \int_{\Omega} \frac{\partial_{tt}u_\vep^f(y,t-c_0^{-1}|x-y|)}{4\pi|x-y|}dy
+ \rho_0\rho_1\vep^4\int_{\Omega}\frac{f_\vep(y, t-c^{-1}_0|x-y|)}{4\pi|x-y|}dy.
\end{align*}
This, together with \eqref{eq:51} directly yields the Lippmann-Schwinger equation \eqref{eq:1} for the case when $j=0$. In addition, by the jump relations of double layer potential (see, for example, \cite{AK-09}), $\partial_{\nu}u_\vep^f$ satisfies that for $x\in \Gamma$ and $ t\in \R_+$,
\begin{align}
&\frac{1}2 \left(\frac{\rho_0}{\rho_1\vep^2} + 1 \right)\partial_{\nu} u_\vep^f(x,t) = \partial_{\nu} v_\vep^f(x,t) - \left(\frac{1}{c^2_1} - \frac{1}{c^2_0}\right) \partial_{\nu}\int_{\Omega}\frac{\partial_{tt} u_{\vep}^f(y,t-c_0^{-1}|x-y|)}{4\pi|x-y|}dy \notag \\
& \qquad \qquad \qquad + \left(\frac{\rho_0}{\rho_1\vep^2}-1\right) h^f_\vep(x,t) -\left(\rho_0 - \rho_1\vep^2\right)\vep^2\partial_\nu \int_{\Omega}\frac{f_\vep(y, t-c^{-1}_0|x-y|)}{4\pi|x-y|}dy, \label{eq:81}
\end{align}
where $h^f_\vep$ is defined by
\begin{align*}
&h^f_\vep(x,t) :=\int_{\Gamma} \frac{\partial_{\nu}u_\vep^f(y,t-c_0^{-1}|x-y|)}{4\pi|x-y|^2}\frac{(x-y)\cdot \nu(x)}{|x-y|}d\sigma(y) \notag\\
&\qquad \qquad \qquad +c_0^{-1}\frac{\partial_t\partial_{\nu} u_\vep^f(y,t-c_0^{-1}|x-y|)}{4\pi|x-y|}\frac{(x-y)\cdot \nu(x)}{|x-y|}d\sigma(y), \quad x\in \Gamma, \quad t\in \R_+.
\end{align*}
Furthermore, utilizing Taylor expansions with respect to $t-$variable, we find
\begin{align}
&h_\vep^f(x,t) = \sum^{q_0}_{l=0} \frac{(-1)^l}{c^l_0 l!}\big(K^*_l\partial^l_t\partial_{\nu}u_\vep^f(\cdot,t)\big)(x) + \sum^{q_1-1}_{l=0} \frac{(-1)^l}{c^{l+1}_0 l!} \big(K^*_{l+1}\partial^{l+1}_t\partial_{\nu}u_\vep^f(\cdot,t)\big)(x) \notag\\
& + \int_{\Gamma}\frac{(x-y)\cdot \nu(x)}{4\pi|x-y|^3}\int^{t-c^{-1}_0|x-y|}_{t}\partial^{q_0+1}_{t} \partial_{\nu} u_\vep^f(y,\tau)\frac{(t-c^{-1}_0|x-y|-\tau)^{q_0}}{q_0!} d\tau d\sigma(y) \notag\\
& + c^{-1}_0 \int_{\Gamma}\frac{(x-y)\cdot \nu(x)}{4\pi|x-y|^2}\int^{t-c^{-1}_0|x-y|}_{t}\partial^{q_1+1}_{t}\partial_{\nu} u_\vep^f(y,\tau)\frac{(t-c^{-1}_0|x-y|-\tau)^{q_1-1}}{(q_1-1)!} d\tau d\sigma(y). \label{eq:40}
\end{align}
Therefore, it can be deduced from \eqref{eq:81} and \eqref{eq:40} that $\partial_{\nu}u_\vep^f(x,t)$ satisfies \eqref{eq:59}. The proof of this lemma is thus completed.
\end{proof}
 
{\color{HW}
In the sequel, given $\psi \in L^2(\Gamma)$, we define 
\begin{align} 
\mathcal P \psi := \frac{1}{\mathrm{cap}(\Omega)} \int_{\Gamma} \psi(x) d\sigma(x) S_0^{-1} 1, \; \mathrm{and} \;\mathcal Q\psi:= \mathbb I \psi -\mathcal P \psi, \label{eq:84}
\end{align}
where the constant $\mathrm{cap}(\Omega)$ is specified in \eqref{eq:74}. It is well-known that (see, for example, \cite{AZ-18})
\begin{align} \label{eq:83}
\left(\frac{1}2 - K^*_0\right) S_0^{-1} 1= 0, \quad \left(\frac{1}2 - K_0\right) 1 = 0.
\end{align}
Here, $K_0$ is the adjoint of $K^*_0$. It is given by
\begin{align*}
K_{0}: H^{-\frac12}(\Gamma)\rightarrow H^{\frac 12}(\Gamma), \quad \left(K_{0}\psi\right)(x) := \int_{\Gamma}{{4\pi|y-x|^{-3}}\nu(y)\cdot(y-x) \psi(y)}d\sigma(y), \quad x\in \Gamma.
\end{align*}
It can be deduced from \eqref{eq:83} that 
\begin{align*}
\mathcal P K^*_0 \psi = \frac{S^{-1}_0 1}{\mathrm{cap}(\Omega)}\int_{\Gamma} \left(K^*_0 \psi\right)(x) d\sigma(x)= \frac{S^{-1}_0 1}{\mathrm{cap}(\Omega)} \int_{\Gamma} (K_0 1)(x)\psi(x)d\sigma(x)  = \frac{1}{2} \mathcal P \psi = K^*_0 \mathcal P \psi.
\end{align*}
Form this, we easily obtain 
\begin{align}\label{eq:144}
\mathcal A \left(\frac 12 - K^*_0\right) = \left(\frac 12 - K^*_0\right) \mathcal A, \quad \mathcal A \in \{\mathcal P, \mathcal Q\}.
\end{align}
Therefore, $1/2 - K^*_0$ is a linear bounded mapping from $L^2_0(\Gamma)$ to $L^2_0(\Gamma)$, where $L^2_0(\Gamma):=\{\phi\in L^{2}(\Gamma): \int_{\Gamma} \phi(x) d\sigma(x) =0\}$. Since $K^*_0$ is a compact mapping from $L^2(\Gamma)$ to $ L^2(\Gamma)$ and $1/2 - K^*_0$ is injective on $L^2_0(\Gamma)$, we have
\begin{align}
\left(\frac{1}2 - K^*_0\right)^{-1} \in \mathcal L\left(L^2_0(\Gamma), L^2_0(\Gamma)\right). \label{eq:77}
\end{align}
Here, $\mathcal L\left(L^2_0(\Gamma), L^2_0(\Gamma)\right)$ denotes the space of all linear bounded mappings from $L^2_0(\Gamma)$ to $L^2_0(\Gamma)$.
}

{\color{HW}
Under the same assumption of Lemma \ref{le:6}, utilizing Taylor expansions with respect to $t-$variable, we have 
\begin{align} \label{eq:156} 
&\int_{\Omega}\frac{\partial^{j+2}_{t} u_{\vep}^f(y,t-c_0^{-1}|x-y|)}{4\pi|x-y|}dy = \sum^{q_3}_{l=0}\frac{(-1)^l}{c^l_0 l!}\left(N_l{\partial^{j+2+l}_{t} u_{\vep}^f(\cdot,t)}\right)(x)\notag\\ 
& +  \int_{\Omega} \int^{t-c^{-1}_0|x-y|}_{t}\partial^{j+3+ q_3}_{t} u_\vep^f(y,\tau)\frac{(t-c^{-1}_0|x-y|-\tau)^{q_3}}{4\pi |x-y|q_3!} d\tau dy,
\end{align}
and
\begin{align*}
&\int_{\Omega} \frac{\partial^j_t f_\vep(y, t-c^{-1}_0|x-y|)}{4\pi|x-y|}dy = \left(N_0\partial^j_t f_\vep(\cdot,t) - \frac{1}{c_0} N_1\partial^{j+1}_t f_\vep(\cdot,t) + \frac{1}{2 c_0^2} N_2\partial^{j+2}_t f_\vep (\cdot,t)\right)(x) \notag \\
&+ \int_{\Omega}\int^{t-c^{-1}_0|x-y|}_{t}\partial^{j+3}_{t} f_\vep(y,\tau)\frac{(t-c^{-1}_0|x-y|-\tau)^{2}}{8\pi|x-y|} d\tau dy.
\end{align*}
Here, $q_3 \in \{l \in \mathbb N_0: l\le p-2\}$ and $j\in \{l \in \mathbb N_0: l\le p-3\}$.
Therefore, with the aid of \eqref{eq:84}, \eqref{eq:83} and \eqref{eq:144}, multiplying $\rho_1\vep^2/\rho_0$ and  subtracting $(1 - {\rho_1\vep^2}/{\rho_0})\mathcal A K^*_0$ from both sides of \eqref{eq:59}|}, we arrive at 
\begin{align}
& \left(1 - \frac{\rho_1\vep^2}{\rho_0}\right)\left(\frac 12 - K^*_0\right)(\mathcal A\partial^j_t\partial_{\nu} u_\vep^f(\cdot,t))(x) = \frac{\rho_1\vep^2}{\rho_0}\left(\mathcal A\partial^j_t\partial_{\nu} v_\vep^f(\cdot,t) - \mathcal A\partial^j_t\partial_{\nu} u_\vep^f(\cdot,t)\right)(x)\notag\\
& + \left(\frac{1}{c^2_0} - \frac{1}{c^2_1}\right)\frac{\rho_1\vep^2}{\rho_0} \mathcal A\partial_{\nu}\bigg[\sum^{\max(0,q_0-2)}_{l=0}\frac{(-1)^l}{c^l_0 l!}\left(N_l{\partial^{j+2+l}_{t} u_{\vep}^f(\cdot,t)}\right)(x)\notag\\ 
& +  \int_{\Omega} \int^{t-c^{-1}_0|x-y|}_{t}\partial^{j+3+\max(0,q_0-2)}_{t} u_\vep^f(y,\tau)\frac{(t-c^{-1}_0|x-y|-\tau)^{\max(0,q_0-2)}}{4\pi |x-y|\max(0,q_0-2)!} d\tau dy \bigg] \notag\\
& + \left(1- \frac{\rho_1\vep^2}{\rho_0}\right)\bigg[\sum^{q_0}_{l=1} \frac{(-1)^l}{c^l_0 l!}\left(\mathcal A K^*_l\partial^{j+l}_t\partial_{\nu}u_\vep^f(\cdot,t)\right)(x) + \sum^{q_1-1}_{l=0} \frac{(-1)^l}{c^{l+1}_0 l!} \big(\mathcal A K^*_{l+1}\partial^{j+l+1}_t\partial_{\nu}u_\vep^f(\cdot,t)\big)(x)\notag\\
& + \mathcal A \int_{\Gamma}\frac{(x-y)\cdot \nu(x)}{4\pi|x-y|^3}\int^{t-c^{-1}_0|x-y|}_{t}\partial^{j+q_0+1}_{t} \partial_{\nu} u_\vep^f(y,\tau)\frac{(t-c^{-1}_0|x-y|-\tau)^{q_0}}{q_0!} d\tau d\sigma(y) \notag\\
& + \mathcal A \int_{\Gamma}\frac{(x-y)\cdot \nu(x)}{4\pi c_0 |x-y|^2}\int^{t-c^{-1}_0|x-y|}_{t}\partial^{j+q_1+1}_{t} \partial_{\nu} u_\vep^f(y,\tau)\frac{(t-c^{-1}_0|x-y|-\tau)^{q_1-1}}{(q_1-1)!} d\tau d\sigma(y)\bigg] \notag\\
&-\left(\rho_0 -\rho_1\vep^2\right)\frac{\rho_1\vep^4}{\rho_0}\mathcal A\partial_\nu \bigg[\left(N_0\partial^j_t f_\vep(\cdot,t) + \frac{1}{2 c_0^2} N_2\partial^{j+2}_t f_\vep (\cdot,t)\right)(x) \notag \\
&+ \int_{\Omega}\int^{t-c^{-1}_0|x-y|}_{t}\partial^{j+3}_{t} f_\vep(y,\tau)\frac{(t-c^{-1}_0|x-y|-\tau)^{2}}{8\pi|x-y|} d\tau dy\bigg] \quad x\in \Gamma, \;\; t\in \R_+, \;\;\mathcal A \in \{\mathcal P, \mathcal Q\}.\label{eq:78}
\end{align}
{\color{HW} Here, we set $q_3 = \max(0,q_0-2)$ in \eqref{eq:156}, which eliminates $q_3$ as a separate variable and simplifies the system. 

We conclude this subsection with the introduction of several useful integral identities.

\begin{lemma} \label{le:8}
Given $p\in \mathbb N$, assume that $f\in H_{0,\sigma}^{p}\left(\R_+; L^2(\R^3)\right)$ for any $\sigma>0$. For each $\vep >0$, let $v^f_\vep$ and $u^f_\vep$ be given by \eqref{eq:82} and \eqref{eq:87}, respectively.  We have that for $t\in \R_+$ and $x\in \Gamma$,
\begin{align}
& \left(\frac 12 - K^*_0\right)(\mathcal P\partial^j_t\partial_{\nu} u_\vep^f(\cdot,t))(x) = 0, \label{eq:89}\\
&\left(\mathcal P \left[\partial^{j}_t \partial_\nu v^f_\vep(\cdot,t) - \rho_0\vep^2  \partial_\nu N_0 \partial^{j}_t f_\vep(\cdot,t)\right]\right)(x) = \frac{1}{c_0^2 {\color{HW}\mathrm{cap}(\Omega)}} \int_{\Omega} \partial^{j+2}_{t} v_\vep^f(y,t) dy \left(S_0^{-1}\right)(x), \label{eq:132}\\
&\frac{c^{2}_0 {\color{HW}\mathrm{cap}(\Omega)}}{|\Omega|
} \frac{1-q}{c^{q}_0 q!}(-1)^{q-1}\left(\mathcal P K^*_q \mathcal P \partial^{j+q}_t \partial_\nu u^f_\vep(\cdot,t)\right)(x) = \eta_{q-1}\partial^{j+q}_t \Lambda_\vep(t) \left(S_0^{-1}\right)(x)\label{eq:133}
\end{align}
and
\begin{align}
&\frac{{\color{HW}\mathrm{cap}(\Omega)}\rho_1}{\rho_0|\Omega|}\left[\left(\mathcal P \partial^{j}_t \partial_\nu u^f_\vep(\cdot,t)\right)(x) - \left(\frac{1}{c^2_0} - \frac{1}{c^2_1}\right)\mathcal P\partial_{\nu}\left(N_0{\partial^{j+2}_{t} u_{\vep}^f(\cdot,t)}\right)(x)\right]\notag\\ 
& = \left[\frac{\omega_M^2}{c^2_0}\partial^j_t \Lambda_\vep(t) - \left(1-\frac{c^2_1}{c^2_0}\right) \frac{\rho^2_1}{\rho_0|\Omega|}\vep^4\int_\Omega \partial^{j}_t f_\vep(y) dy\right] \left(S_0^{-1}\right)(x). \label{eq:134}
\end{align}
Here, $j\in\{l\in \mathbb N_0: l \le p-1\}$ and $q \in \{l \in \mathbb N: 2\le l\le p+1-j\}$, and $\eta_j$ ($j \in \{1,2,3,4\}$) is specified in \eqref{eq:114}-\eqref{eq:115}.

\end{lemma}}
\begin{proof}
By Lemma \ref{le:2}, we readily obtain that $v^f_\vep, \; u^f_\vep \in H_{0,\sigma}^{p+2}\left(\R_+; L^2(\R^3)\right)$ for any $\sigma > 0$.

First, we prove \eqref{eq:89}. This identity directly follows from \eqref{eq:84} and \eqref{eq:83}.

Second, we prove \eqref{eq:132}. An integration by parts gives 
\begin{align*}
& \frac{1}{{\color{HW}\mathrm{cap}(\Omega)}} \int_{\Gamma} \partial^j_t\partial_\nu v^f_\vep(x,t) d\sigma(x) - \frac{\rho_0\vep^2}{{\color{HW}\mathrm{cap}(\Omega)}}\int_{\Gamma}\int_{\Omega}\partial_{\nu (x)}\frac{\partial^{j}_{t} f_{\vep}(y,t)}{4\pi|x-y|}dyd\sigma(x)\\
&= \frac{1}{{\color{HW}\mathrm{cap}(\Omega)}} \int_\Omega \partial^j_t \Delta v_\vep^f(y,t) dy + \frac{\rho_0 \vep^2}{{\color{HW}\mathrm{cap}(\Omega)}}\int_{\Omega} \partial^j_t f_\vep(y,t)dy.
\end{align*}
This, together with \eqref{eq:15} and \eqref{eq:84} gives \eqref{eq:132}.

Third, we prove \eqref{eq:133}. It is known (see, for instance, \cite[Lemma 2.6]{LS-04}) that 
\begin{align} \label{eq:147}
&\frac{1}{8\pi}\int_{\Gamma} \int_{\Gamma} \frac{\nu(x)\cdot (x-y)}{|x-y|} \left(S^{-1}_01\right)(y) d\sigma(y)d\sigma(x) = |\Omega|
\end{align}
and 
\begin{align}
\int_{\Gamma}\int_{\Gamma} \nu(x) \cdot (x-y) \left(S^{-1}_01\right)(y) d\sigma(y) d\sigma(x) =3{{\color{HW}\mathrm{cap}(\Omega)}|\Omega|}. \label{eq:148}
\end{align}
Combining this with \eqref{eq:84} and \eqref{eq:73} gives \eqref{eq:133}.

Fourth, we prove \eqref{eq:134}.
By integrating by parts, we find
\begin{align*}
\int_{\Gamma}\int_{\Omega}\partial_{\nu(x)}\frac{\partial^{j+2}_{t} u_{\vep}^f(y,t)}{4\pi|x-y|}dyd\sigma(x)  = - \int_{\Omega} \partial^{j+2}_{t} u_{\vep}^f(y,t)dy.
\end{align*}
In conjunction with \eqref{eq:10} gives 
\begin{align*}
\frac{1}{c^2_1}\int_{\Omega} \partial^{j+2}_{t} u_{\vep}^f(y,t)dy = \partial^{j}_t\int_{\Omega}\Delta u^f_\vep(y) dy + \rho_1\vep^4 \int_\Omega \partial^{j}_t f_\vep(y) dy.
\end{align*}
Thus, we arrive at
\begin{align} \label{eq:72} 
-\frac{1}{c^2_1}\int_{\Gamma}\int_{\Omega}\partial_{\nu(x)}\frac{\partial^{j+2}_{t} u_{\vep}^f(y,t)}{4\pi|x-y|}dy d\sigma(x) = \partial^j_t \int_{\Gamma}\partial_\nu u^f_\vep(y) d\sigma(y) + \rho_1\vep^4 \int_\Omega \partial^{j}_t f_\vep(y) dy.
\end{align}
Utilizing \eqref{eq:45}, \eqref{eq:73}, \eqref{eq:84} and \eqref{eq:72}, we readily obtain \eqref{eq:134}. The proof of this lemma is thus completed.
\end{proof}

\subsection{A priori and pullback estimates}
{\color{HW} In this subsection, we collect a priori estimates (Lemmas \ref{le:0} and \ref{le:1}) and pullback estimates (Lemmas \ref{le:3} and \ref{le:9}) that are crucial for the proof of Theorem~\ref{th:1}. Proofs appear in Section \ref{sec:3}: Lemmas \ref{le:0}-\ref{le:1} in Sections \ref{sec:3.1} and \ref{sec:3.2}; Lemmas \ref{le:3} and \ref{le:9} in Section \ref{sec:3.3}.}

\begin{lemma} \label{le:0}
Let $\vep>0$ be sufficiently small such that $\Omega_\vep \subset B_1(y_0)$. The following arguments hold true.
\begin{enumerate}[(a)]
\item \label{b1} Given $p \in \{l \in \mathbb N: l > 1\}$ and $\alpha > 1$, assume that $f\in H_{0,\sigma}^p\left(\R_+; L_{\alpha}^{2}(\R^3)\right)$ for any $\sigma\in \R_+$. 
For $j \in\{l\in \mathbb N_0 : l \le p-1\}$, we have
\begin{align} 
&\left\|\partial^j_t v^f_\vep \right\|^2_{L^2\left((0,T/\vep^2); L^2(\Omega)\right)}  \le C\vep^{2j-1} \left\|\partial^{j}_t f\right\|^2_{H_{0,\vep}^{1}\left(\R_+; L_{\alpha}^2(\R^3)\right)}. \label{eq:92} 
\end{align}
Furthermore, for $t\in (0, T\vep^{-2}]$ and $j\in\{l\in \mathbb N_0 : l \le p-2\}$, we have
\begin{align}
& \big\|\partial^j_t v^f_\vep(\cdot, t)\|_{L^2(\Omega)} \le C\vep^{j} \left\|\partial^{j}_t f\right\|_{H_{0,\vep}^{2}\left(\R_+; L_{\alpha}^2(\R^3)\right)}, \label{eq:61}\\
& \big\|\partial^j_t\partial_\nu v^f_\vep(\cdot, t)\|_{L^2(\Gamma)}  \le C \vep^{j+\frac 12}\left\|\partial^{j}_t f\right\|_{H_{0,\vep}^{2}\left(\R_+; L_{\alpha}^2(\R^3)\right)}. \label{eq:93}
\end{align}
\item \label{b2}
Given $p \in \mathbb N$ and $\alpha > 3/2$, assume that $f\in {\color{HW}W_{1}^{p,2}}\left(\R_+; L_{\alpha}^{2}(\R^3)\right)$. For $j \in\{l\in \mathbb N_0 : l \le p-1\}$, we have
\begin{align}
\int_0^{T/\vep^2} \left|\partial^j_t v^f_\vep(x, \tau)- \vep^j \partial^j_t v^f(y_0,\vep\tau)\right|d\tau \le C\vep^{j-\frac 12}\|\partial^j_t f\|_{{\color{HW}W_1^{1,2}}\left(\R_+; L_{\alpha}^{2}(\R^3)\right)},\label{eq:137}
 \end{align}
uniformly for all $x\in \Omega$, where 
\begin{align}\label{eq:30}
\big\|\partial^j_t v^f(y_0,\cdot) \big\|_{L^1\left((0,T/\vep); \R\right)} \le C\|\partial^{j}_tf\|_{{\color{HW}W_1^{1,2}}\left(\R_+; L_{\alpha}^{2}(\R^3)\right)}.
\end{align}
Furthermore, for $j \in\{l\in \mathbb N_0 : l \le p-1\}$, we have
\begin{align}
&\left\|\partial^j_t \partial_\nu v^f_\vep \right\|_{L^1\left((0,T/\vep^2); L^2(\Gamma)\right)} \le C\vep^{j-\frac 12} \left\|\partial^{j}_t f\right\|_{{\color{HW}W_1^{1,2}}\left(\R_+; L_{\alpha}^{2}(\R^3)\right)}. \label{eq:136}
\end{align}
\end{enumerate}
Here, $C$ is a positive constant independent of $\vep$ and $f$.
\end{lemma}

\begin{lemma} \label{le:1}
Let $\vep>0$ be sufficiently small such that $\Omega_\vep \subset B_1(y_0)$. Given $p \in \{l \in \mathbb N: l > 12\}$ and $\alpha > 1$, assume that $f\in H_0^{p}\left(\R_+; L_{\alpha}^2(\R^3)\right)$. Then $u^f_\vep$ satisfies the following estimates.

\begin{enumerate}[(a)] 
\item \label{f1} For $t\in (0, T\vep^{-2}]$ and $j\in\{l\in \mathbb N: l \le p-3\}$, we have
\begin{align}
& \left\|\partial^j_t u^f_\vep \right\|^2_{L^2\left((0,T/\vep^2); L^2(\Omega)\right)}  \le C \vep^{2j-{6}} \left\|\partial^{j-1}_t f\right\|^2_{H_{0,\vep}^{1}\left(\R_+; L_{\alpha}^2(\R^3)\right)}, \label{eq:3} \\
& \left\|\partial^j_t \partial_\nu u^f_\vep \right\|^2_{L^2\left((0,T/\vep^2); L^2(\Gamma)\right)} \le C\vep^{2j-2}\left\|\partial^{j}_t f\right\|^2_{{H_{0,\vep}^{3}}\left(\R_+; L_{\alpha}^2(\R^3)\right)}. \label{eq:31}
\end{align}

\item \label{f2} For $t\in (0, T\vep^{-2}]$ and $j\in\{l\in \mathbb N : l\le p-8\}$, we have 
\begin{align}
&\left\|\partial^{j}_{t}\partial_\nu u^f_\vep(\cdot,t)\right\|_{L^2(\Gamma)} \le C \vep^{j + \frac{1}2} \|f\|_{H_0^{{j + 8}}\left(\R_+; L_{\alpha}^2(\R^3)\right)}, \label{eq:75}\\
&\left\|\partial^{j}_{t} u^f_\vep(\cdot,t)\right\|_{L^2(\Omega)} \le C \vep^{j - \frac{1}2}\|f\|_{H_0^{{j + 8}}\left(\R_+; L_{\alpha}^2(\R^3)\right)}. \label{eq:79} 
\end{align}

\item \label{f3} For $t\in (0, T\vep^{-2}]$ and $j \in\{l \in \mathbb N : l \le p-12\}$, we have 
\begin{align}
&\left\|\partial^{j}_{t}\partial_\nu u^f_\vep(\cdot,t)\right\|_{L^2(\Gamma)} \le C \vep^{j + \frac 3 2}  \|f\|_{H_0^{{j + 12}}\left(\R_+; L_{\alpha}^2(\R^3)\right)}, \label{eq:91}\\
&\left\|\partial^{j}_{t} u^f_\vep(\cdot,t)\right\|_{L^2(\Omega)} \le C \vep^{j - \frac 1 2}\|f\|_{H_0^{{j + 12}}\left(\R_+; L_{\alpha}^2(\R^3)\right)}. \label{eq:95}
\end{align}

\end{enumerate}
Here, $C$ is a positive constant independent of $\vep$ and $f$.
\end{lemma}

Before stating pullback estimates, we introduce the map
\begin{align}\label{eq:112}
\Phi_\vep(y):= y_0+ \vep(y-y_0)\quad \textrm{for each}\; \vep>0.
\end{align}
Here, $y_0$ is any fixed {point} in $\R^3$.
Furthermore, given any complex valued function $\phi$ and an operator $\mathcal A$ mapping complex valued functions from one function space to another, we define 
\begin{align*}
\left(\left(\Phi_\vep \circ \mathcal A\right)\phi\right)(y): = (\mathcal A \phi) (\Phi_\vep(y)).
\end{align*}

Now we are ready to introduce two types of pullback estimates.

\begin{lemma} \label{le:3}
Let $\sigma_1,\sigma_2 > 1/2$ with $\sigma_1 + \sigma_2 > 2$. Given $\vep > 0$, the following arguments hold true.
\begin{enumerate}[(a)]
\item \label{h0}
For $\phi_1\in H^{-1/2}(\Gamma)$, we have
\begin{align*}
\left(\left(\Phi_{\vep^{-1}} \circ SL_{0}\right) \phi_1\right)(y) = \vep\frac{1}{4\pi|y-y_0|}\int_\Gamma\phi_1(x)d\sigma(x) + \textrm{Res}_1(y)
\end{align*}
with $\textrm{Res}_1(y)$ satisfying
\begin{align*} 
&\|\textrm{Res}_1\|_{L_{-\sigma_2}^2(\R^3)} \le C\vep^{\frac 32}\|R_{0}\|_{L^2_{\sigma_1}(\R^3),H^2_{-\sigma_2}(\R^3)
}\|\phi_1\|_{H^{-\frac12}(\Gamma)}.
\end{align*}

\item For $\phi_2\in L^2(\Omega)$, we have
\begin{align*}
&\left(\left(\Phi_{\vep^{-1}} \circ N_{0}\right) \phi_2\right)(y) =  \vep\frac{1}{4\pi|y-y_0|}\int_\Omega\phi_2(x)dx + \textrm{Res}_2(y)
\end{align*}
with $\textrm{Res}_2(y)$ satisfying
\begin{align*}
\|\textrm{Res}_2\|_{L_{-\sigma_2}^2(\R^3)}\le C\vep^{\frac 32}\|R_{0}\|_{L^2_{\sigma_1}(\R^3),H^2_{-\sigma_2}(\R^3)
} \|\phi_2\|_{L^2(\Omega)}. 
\end{align*}
\end{enumerate}
Here, $C$ is a positive constant independent of $\vep$, $\phi_1$ and $\phi_2$. 
\end{lemma}

\begin{lemma} \label{le:9}
Let $\alpha > 1$ and $\vep > 0$. Assume $t \in \R_+$. For $g\in H_0^{2}\left(\R_+; L^{2}(\Gamma)\right)$, we have
\begin{align}
&\left\|\int_{\Gamma}\frac{g\left(y, \vep^{-1}t - c_0^{-1}\left|\Phi_{\vep^{-1}}(\cdot)-y\right|\right)}{4\pi|\Phi_{\vep^{-1}}(\cdot) -y|}d\sigma(y)\right\|_{L_{-\alpha}^2(\R^3)} \le  C \vep \sup_{\tau \in (0, t/\vep)}\|g{(\cdot, \tau)}\|_{L^2(\Gamma)} \notag \\
& + C\vep\left\|\partial_t g \right\|^{\frac12}_{L^2\left((0,t/\vep); L^2(\Gamma)\right)}\left\|\partial_{tt} g\right\|_{L^2\left((0,t/\vep); L^2(\Gamma)\right)}^{\frac{1}2}. \label{eq:41}
\end{align}
Furthermore, for $h\in H_0^{2}\left(\R_+; L^2(\Omega)\right)$, we have
\begin{align}
&\left\|\int_{\Omega}\frac{h\left(y, \vep^{-1}t - c_0^{-1}\left|\Phi_{\vep^{-1}}(\cdot)-y\right|\right)}{4\pi|\Phi_{\vep^{-1}}(\cdot) -y|}dy\right\|_{L_{-\alpha}^2(\R^3)} \le  C \vep \sup_{\tau \in (0, t/\vep)}\|h {(\cdot, \tau)}\|_{L^2(\Omega)} \notag \\
& + C\vep \left\|\partial_t h \right\|^{\frac12}_{L^2\left((0,t/\vep); L^2(\Omega)\right)}\left\|\partial_{tt} h\right\|_{L^2\left((0,t/\vep); L^2(\Omega)\right)}^{\frac{1}2}. \label{eq:44}
\end{align}
Here, $C$ is a positive constant independent of $\vep$, $t$, $g$ and $h$.
\end{lemma}


\subsection{Asymptotic expansion of the wave fields}

In this subsection, we derive the asymptotic expansion of the wave field, which completes the proof of Theorem \ref{th:1}. 

\begin{proof}[Proof of Theorem \ref{th:1}]
It follows from \eqref{eq:1} that 
\begin{align} \label{eq:23}
& u_\vep^f(x,t)- v_\vep^f(x,t) = -\left(\frac{1}{c^2_1} - \frac{1}{c^2_0}\right) \int_{\Omega}\frac{\partial^{2}_{t} u_{\vep}^f(y,t-c_0^{-1}|x-y|)}{4\pi|x-y|}dy\notag \\
& -\left(\frac{\rho_0}{\rho_1\vep^2}-1\right) \int_{\R}\int_{\Gamma}\frac{\delta(t-c^{-1}_0|x-y|-\tau)}{4\pi|x-y|}\left[\Lambda_\vep(\tau)S^{-1}_0(y) + \mathcal Q \partial_{\nu} u_\vep^f(y,\tau) \right]d\sigma(y) d\tau \notag\\
& - \left(\rho_0 - \rho_1\vep^2\right)\vep^2 \int_{\Omega} \frac{f_\vep(y, t-c^{-1}_0|x-y|)}{4\pi|x-y|}dy =: \sum_{l = 1}^{4} r_l(x,t), \quad x\in \R^3\backslash\Gamma, \quad t\in \R_+.
\end{align}
Here, $\Lambda_\vep(t)$ is given by \eqref{eq:73}.
The rest part of the proof is divided into two parts. The first one involves calculating $\Lambda_\vep(t)$ and $\mathcal Q \partial_\nu u_\vep^f(x,t)$, and the second one derives the asymptotic expansion stated in the theorem by Lemma \ref{le:9}.

\textbf{Part I}: In this part, our first aim is to estimate the projection coefficient $\Lambda_\vep(t)$. {\color{HW} By setting $(j,q_0, q_1, \mathcal A) = (0,5,5,\mathcal P)$, $(j,q_0, q_1, \mathcal A) = (1,4,4,\mathcal P)$,  $(j,q_0, q_1, \mathcal A) = (2,3,3,\mathcal P)$ and $(j,q_0, q_1, \mathcal A) = (3,2,2,\mathcal P)$ in \eqref{eq:78}, with the aid of Lemma \ref{le:8}, proceeding as in the derivation of \eqref{eq:155}, we obtain
\begin{align}
&\sum^{4-j_1}_{l=1} \eta_l\partial^{l+j_1+1}_{t}\Lambda_\vep(t) + \gamma_{\vep} \partial^{j_1}_t\Lambda_\vep(t) = \int_{\Omega} \frac{\rho_1\vep^2 \partial^{j_1+2}_{t} v_\vep^f(y,t)}{(\rho_0-\rho_1\vep^2)|\Omega|} dy + \sum_{l=1}^{4}\textrm{Err}^{(l)}_{j_1}(t), \quad t\in \R_+, \label{eq:80}
\end{align}
where for $t\in \R_+$
$\textrm{Err}^{(p)}_{j_1}(t)$ $(p\in\{1,2,3\})$ is defined by 
\begin{align} 
&\textrm{Err}^{(1)}_{j_1}(t) := \frac{c^2_0}{|\Omega|} \sum^{5}_{l=j_1+1} \frac{(-1)^{l -1}(l-1)}{c^{l}_0 l!} \left\langle 1, K^*_l \mathcal Q\partial^{l}_t \partial_\nu u^f_\vep(\cdot,t)\right\rangle_{L^2(\Gamma)}, \label{eq:163}\\
&\textrm{Err}^{(2)}_{j_1}(t) := \frac{c^2_0}{|\Omega|}(c^{-2}_1 - c_0^{-2})\frac{\rho_1\vep^2}{\rho_1\vep^2-\rho_0}\sum^{3-j_1}_{l=2} \frac{(-1)^{l}l}{c^{l}_0 l!} \left\langle 1, \partial_\nu N_l \partial^{l+j_1 + 2}_t u^f_\vep(\cdot,t)\right\rangle_{L^2(\Gamma)}, \label{eq:164}\\
&\textrm{Err}^{(3)}_{j_1}(t) := 
\frac{c^2_0}{|\Omega|} \frac{\rho_1\vep^4}{\rho_1\vep^2-\rho_0}  \bigg\langle 1, \frac{\rho_0-\rho_1\vep^2}{2c^2_0}\partial_\nu N_2 \partial^{j_1+2}_{t}f_\vep(\cdot,t)- \rho_1\vep^2 \partial_\nu N_0\partial^{j_1}_tf_\vep(\cdot,t)\bigg\rangle_{L^2(\Gamma)} \notag \\
&\qquad\qquad\;\quad + \left(1-\frac{c^2_1}{c^2_0}\right) \frac{\rho^2_1 \vep^6}{\rho_0}\int_\Omega \partial^{j_1}_tf_\vep(y) dy, \notag
\end{align}
and $\textrm{Err}^{(4)}_{j_1}(t)$ is the constant $c^2_0/((1-\rho_1\vep^2/\rho_0)|\Omega|)$ times the projection coefficient of the sum of the four integral Taylor remainder terms in \eqref{eq:78} onto $S_0^{-1}1$.
Here, $\gamma_\vep$ is specified by \eqref{eq:116}, and $\eta_p$ $(p\in \{1,2,3,4\})$ is given by \eqref{eq:114}--\eqref{eq:115}. 

For the estimate of $\textrm{Err}^{(4)}_{j_1}(t)$ $(j_1\in \{0,1,2,3\})$, using \eqref{eq:122}, \eqref{eq:3}, \eqref{eq:31} and \eqref{eq:a4}, we obtain,
\begin{align*}
&\left|\textrm{Err}^{(4)}_{j_1}(t)\right| \le  C \vep^{\frac {11} 2}\|f\|_{H_0^{10}\left(\R_+; L_{\alpha}^2(\R^3)\right)}, \quad t \in \left(0, {T}{\vep^{-2}}\right].
\end{align*}

For the estimate of $\textrm{Err}^{(3)}_{j_1}(t)$ $(j_1\in \{0,1,2,3\})$, a straightforward calculation gives
\begin{align}
&\left\|\partial^{j_1}_tf_\vep\right\|_{L^1((0,T/\vep^2);L^2(\Omega))}\le C\vep^{j_1-\frac 52}\|f\|_{{\color{HW}W_1^{j_1,2}}(\R_+; L_{\alpha}^2(\R^3))}. \label{eq:157}
\end{align}
Therefore, we obtain 
\begin{align*}
\int^{T/\vep^2}_0 \left|\textrm{Err}^{(3)}_{j_1}(\tau)\right| d\tau  \le C\vep^{\frac 72}\|f\|_{W_1^{3,2}(\R_+; L_{\alpha}^2(\R^3))}.
\end{align*}

For the estimate of $\textrm{Err}^{(2)}_{j_1}(t)$ $(j_1\in \{0,1,2,3\})$, it follows from \eqref{eq:95} and \eqref{eq:79} that 
\begin{align}
&\vep^2 \sup_{\tau \in (0, T/\vep^2)}\left\|\left(\partial^{4}_t u_\vep^f\right)(\cdot,\tau) \right\|_{L^2(\Omega)} \le C \vep^{\frac {11} 2}\|f\|_{H_0^{16}\left(\R_+; L_{\alpha}^2(\R^3)\right)}, \label{eq:162}\\
&\vep^2 \sup_{\tau \in (0, T/\vep^2)}\left\|\left(\partial^{5}_t u_\vep^f\right)(\cdot,\tau) \right\|_{L^2(\Omega)} \le C \vep^{\frac{11}2}\|f\|_{H_0^{13}\left(\R_+; L_{\alpha}^2(\R^3)\right)}.
\end{align}
Therefore, we have 
\begin{align*}
&\left|\textrm{Err}^{(2)}_{j_1}(t)\right| \le  C \vep^{\frac {11} 2}\|f\|_{H_0^{16}\left(\R_+; L_{\alpha}^2(\R^3)\right)}, \quad t \in \left(0, {T}{\vep^{-2}}\right].
\end{align*}

For the estimate of $\textrm{Err}^{(1)}_{j_1}(t)$ $(j_1\in \{0,1,2,3\})$, it can be deduced from \eqref{eq:91} and \eqref{eq:75} that
\begin{align}
&\sup_{\tau \in (0, T/\vep^2)}\left\|\left(\mathcal Q\partial^{4}_t\partial_{\nu}u_\vep^f\right)(\cdot,\tau) \right\|_{L^2(\Gamma)} \le C \vep^{\frac {11} 2}\|f\|_{H_0^{16}\left(\R_+; L_{\alpha}^2(\R^3)\right)}, \label{eq:158}\\
&\sup_{\tau \in (0, T/\vep^2)}\left\|\left(\mathcal Q\partial^{5}_t\partial_{\nu}u_\vep^f\right)(\cdot,\tau) \right\|_{L^2(\Gamma)} \le C \vep^{\frac{11}2}\|f\|_{H_0^{13}\left(\R_+; L_{\alpha}^2(\R^3)\right)}.\label{eq:159}
\end{align}
Setting $(j,q_0, q_1, \mathcal A) = (2,3,3,\mathcal Q)$ and $(j,q_0, q_1, \mathcal A) = (3,2,2,\mathcal Q)$ in \eqref{eq:78}, invoking the invertibility of $1/2-K^*_0$ in $L^2_0(\Gamma)$ (see \eqref{eq:77}), and applying \eqref{eq:118},  \eqref{eq:93}, together with \eqref{eq:122}, \eqref{eq:3}, \eqref{eq:31} and \eqref{eq:a4} (which control the integral Taylor remainder terms), we have
\begin{align} 
&\sup_{\tau \in (0, T/\vep^2)}\left\|\left(\mathcal Q\partial^{2}_t\partial_{\nu}u_\vep^f\right)(\cdot,\tau) - \textrm{rem}(\cdot,\tau)\right\|_{L^2(\Gamma)} \notag \\
&\le C \vep^{\frac {11} 2}\|f\|_{H_0^{11}\left(\R_+; L_{\alpha}^2(\R^3)\right)}  + C \sum^5_{l=4} \sup_{\tau \in (0, T/\vep^2)}\left\|\left(\mathcal Q\partial^{l}_t\partial_{\nu}u_\vep^f\right)(\cdot,\tau) \right\|_{L^2(\Gamma)},  \label{eq:123}\\
&\sup_{\tau \in (0, T/\vep^2)}\left\|\left(\mathcal Q\partial^{3}_t\partial_{\nu}u_\vep^f\right)(\cdot,\tau) \right\|_{L^2(\Gamma)} \notag \\
&\le C \vep^{\frac {11} 2}\|f\|_{H_0^{11}\left(\R_+; L_{\alpha}^2(\R^3)\right)} +  C\sup_{\tau \in (0, T/\vep^2)}\left\|\left(\mathcal Q\partial^{5}_t\partial_{\nu}u_\vep^f\right)(\cdot,\tau) \right\|_{L^2(\Gamma)}.\label{eq:160}
\end{align}
Here, $\textrm{rem}(x,t)$ is specified by
\begin{align}\label{eq:125}
\textrm{rem}(x,t):=\frac{\rho_1\vep^2}{\rho_0-\rho_1\vep^2}\left(\frac{1}2-K^*_0\right)^{-1}\mathcal Q \left[ \partial^2_{t}\partial_\nu v^f_\vep(x,t) - \rho_0 \vep^2  \partial_\nu N_0 \partial^2_tf_\vep(x,t)\right], \quad x\in \Gamma, \;t\in \R_+.
\end{align}
It can be deduced from \eqref{eq:136}, \eqref{eq:157}, \eqref{eq:123} and \eqref{eq:125} that 
\begin{align*}
\int^{T/\vep^2}_0 \left\|\left(\mathcal Q\partial^{2}_t\partial_{\nu}u_\vep^f\right)(\cdot,\tau)\right\|_{L^2(\Gamma)} d\tau  \le C\vep^{\frac 7 2}\left(\|f\|_{H_0^{16}\left(\R_+; L_{\alpha}^2(\R^3)\right)} + \|f\|_{W_1^{3,2}(\R_+; L_{\alpha}^2(\R^3))}\right).
\end{align*}
Combining this with \eqref{eq:158}, \eqref{eq:159} and \eqref{eq:160} gives
\begin{align*}
\int^{T/\vep^2}_0 \left|\textrm{Err}^{(1)}_{j_1}(\tau)\right|  d\tau \le C\vep^{\frac 7 2}\left(\|f\|_{H_0^{16}\left(\R_+; L_{\alpha}^2(\R^3)\right)} + \|f\|_{W_1^{3,2}(\R_+; L_{\alpha}^2(\R^3))}\right).
\end{align*}
Building upon the estimates of $\textrm{Err}^{(p)}_{j_1}(t)$ ($p \in \{1,2,3,4\}$), we have that for  $j_1\in \{0,1,2,3\}$,
\begin{align} \label{eq:161}
 \int^{T/\vep^2}_0 \sum^4_{p = 1}\left|\textrm{Err}^{(p)}_{j_1}(\tau) \right|d\tau  \le C\vep^{\frac 7 2}\left(\|f\|_{H_0^{16}\left(\R_+; L_{\alpha}^2(\R^3)\right)} + \|f\|_{W_1^{3,2}(\R_+; L_{\alpha}^2(\R^3))}\right).
\end{align}
}

Therefore, applying statement \eqref{d3} of Lemma \ref{le:7} to \eqref{eq:80}, and using \eqref{eq:109}, \eqref{eq:137} and \eqref{eq:161}, we obtain
\begin{align}
\Lambda_\vep(t) &=\int^{t}_0 e^{-\frac{{\color{HW}\mathrm{cap}(\Omega)}} {8\pi c_0}\omega^2_M\vep^2(t-\tau)} \frac{e^{i\omega_M \vep(t-\tau)} - e^{-i\omega_M\vep(t-\tau)}} {2i \omega_M\vep} \frac{\rho_1\vep^2}{\rho_0 |\Omega|} \int_\Omega\partial^{2}_{t}v_\vep^f(y,\tau) dy d\tau   +  \Lambda_{\vep,\textrm{Res}}(t), \label{eq:117}
\end{align}
where $\Lambda_{\vep,\textrm{Res}}(t)$ satisfies
\begin{align}\label{eq:119}
\left|\Lambda_{\vep,\textrm{Res}}(t)\right| \le C\vep^{\frac 5 2} \left(\|f\|_{H_0^{16}\left(\R_+; L_{\alpha}^2(\R^3)\right)} + \|f\|_{W_1^{3,2}(\R_+; L_{\alpha}^2(\R^3))}\right), \quad t \in \left(0, {T}{\vep^{-2}}\right].
\end{align}

Second, we estimate $\mathcal Q \partial_\nu u^f_\vep(x,t)$.  With the aid of \eqref{eq:75}, we find 
\begin{align} \label{eq:18}
\sup_{\tau \in (0, T/\vep^2)}\left\|\partial^{2}_{t}\partial_\nu  u^f_\vep(\cdot,\tau)\right\|_{L^2(\Gamma)} \le C\vep^{\frac 5 2}\|f\|_{H_0^{10}\left(\R_+; L_{\alpha}^2(\R^3)\right)}.
\end{align}
From this, setting $(j,q_0, q_1,\mathcal A) = (0, 2, 2,\mathcal Q)$ in \eqref{eq:78}, and utilizing \eqref{eq:118}, \eqref{eq:122}, \eqref{eq:93}, \eqref{eq:a4} and Lemma \ref{le:1}, we have
\begin{align} \label{eq:19}
\sup_{\tau \in (0, T/\vep^2)}\left\|\mathcal Q \partial_{\nu} u_\vep^f(\cdot,\tau)\right\|_{L^2(\Gamma)} \le C\vep^{\frac 5 2}\|f\|_{H_0^{10}\left(\R_+; L_{\alpha}^2(\R^3)\right)}. 
\end{align}

\textbf{Part II}: In this part, we derive the asymptotic expansion \eqref{eq:120} with the remainder term satisfying \eqref{eq:121}. To achieve this goal, we consider the following five steps. We note that 
\begin{align} \label{eq:113} 
&\qquad u^f(x,t) = u^f_\vep(\Phi_{\vep^{-1}}(x), \vep^{-1}t)\; \textrm{and}\; v^f(x,t) = v^f_\vep(\Phi_{\vep^{-1}}(x), \vep^{-1}t)\quad  \textrm{for}\; (x,t)\in \R^3 \times \R_+. 
\end{align}
Here, the map $\Phi_{\vep^{-1}}$ is given by \eqref{eq:112}.

\textbf{Step 1:} In this step, we estimate $r_{1}$. It follows from \eqref{eq:95} that 
\begin{align*}
\sup_{\tau \in (0, T/\vep^2)}\left\|\partial^{2}_{t} u^f_\vep(\cdot,\tau)\right\|_{L^2(\Omega)} \le C\vep^{\frac 1 2}\|f\|_{H_0^{14}\left(\R_+; L_{\alpha}^2(\R^3)\right)}
\end{align*}
and
\begin{align*}
\vep \left\|\partial^2_t  u^f_\vep \right\|^{\frac12}_{L^2\left((0,T/\vep^2); L^2(\Omega)\right)}\left\|\partial^3_{t}  u^f_\vep\right\|_{L^2\left((0,T/\vep^2); L^2(\Omega)\right)}^{\frac{1}2} \le C\vep^{\frac 3 2}\|f\|_{H_0^{15}\left(\R_+; L_{\alpha}^2(\R^3)\right)}.
\end{align*}
Therefore, drawing upon Lemma \ref{le:9}, we arrive at 
\begin{align} \label{eq:102}
\sup_{\tau \in (0, T/\vep)}\left\|r_{1}(\phi_{\vep^{-1}}(\cdot), \vep^{-1} \tau)\right\|_{L^2_{-\alpha}(\R^3)} \le C\vep^{\frac 3 2} \|f\|_{H_0^{16}\left(\R_+; L_{\alpha}^2(\R^3)\right)}.
\end{align}

\textbf{Step 2}: In this step, we estimate $r_{2}$. It is easy to verify that 
\begin{align*}
\left||\vep\Phi_{\vep^{-1}}(x)-y| - |x-y_0|\right| \le \vep|y-y_0|, \quad \mathrm{for}\; x\in \R^3\;\mathrm{and}\; y\in \Gamma.
\end{align*}
In conjunction with \eqref{eq:117}, \eqref{eq:119} and statement \eqref{b2} of Lemma \ref{le:0}, we have
\begin{align*}
\Lambda_\vep(\vep^{-1}t-c_0^{-1}|\Phi_{\vep^{-1}}(x)-y|) =  \zeta_\vep \Lambda_{\vep,y_0}(x,t) + \Lambda_{\vep, \textrm{Res}}(x,y,t),
\end{align*}
where $\zeta_\vep := \rho_1 |\Omega|\omega_M\vep^2/\left(2i{\color{HW}\mathrm{cap}(\Omega)} k_1\right)$, $\Lambda_{\vep, y_0}$ is defined by 
\begin{align*}
\Lambda_{\vep, y_0}(x,t):=\int^{t - \frac{|x-y_0|}{c_0}}_0 
\left(e^{-iz_M^{-}(\vep)(t-c^{-1}_0|x - y_0|-\tau)} - e^{-iz_M^{+}(\vep)(t -c^{-1}_0|x-y_0|-\tau)} \right) \partial^2_tv^f(y_0,\tau)d\tau,
\end{align*}
and $\Lambda_{\vep,\textrm{Res}}$ satisfies
\begin{align*}
&\sup_{(x, y, t) \in \R^3 \times \Gamma \times \left(0, {T}{\vep}^{-1}\right]}\left|\Lambda_{\vep, \textrm{Res}}(x, y, t)\right| \le C{\vep^{\frac 52}}
\left(\|f\|_{H_0^{16}\left(\R_+; L_{\alpha}^2(\R^3)\right)} + \left\|f\right\|_{{\color{HW}W_1^{3,2}}\left(\R_+; L_{\alpha}^2(\R^3)\right)}\right).
\end{align*}
Therefore, with the aid of Lemma \ref{le:3}, we obtain
\begin{align} 
& r_{2}(\phi_{\vep^{-1}}(x), \vep^{-1}t) = \frac{i\omega_M\rho_0|\Omega|}{8\pi k_1|x-y_0|}\vep \Lambda_{\vep,y_0}(x,t) + r_{2,\textrm{Res}}(x,t),\label{eq:103}
\end{align}
where $r_{2,\textrm{Res}}(x,t)$ satisfies
\begin{align} \label{eq:104}
\sup_{\tau \in (0, T/\vep)}\|r_{2,\textrm{Res}}(\cdot, \tau)\|_{L^2_{-\alpha}(\R^3)} \le C\vep^{\frac 3 2} \left(\|f\|_{H_0^{16}\left(\R_+; L_{\alpha}^2(\R^3)\right)} + \left\|f\right\|_{{\color{HW2}W_1^{3,2}}\left(\R_+; L_{\alpha}^2(\R^3)\right)}\right).
\end{align}

\textbf{Step 3}: In this step, we estimate $r_{3}$.
Setting $(j, q_0, q_1, \mathcal A) = (1, 1, 1, \mathcal Q)$ in \eqref{eq:78} and proceeding as in the derivation of \eqref{eq:19}, we have
\begin{align*}
\sup_{\tau \in (0, T/\vep^2)}\left\|\mathcal Q \partial_t\partial_{\nu} u_\vep^f(\cdot,\tau)\right\|_{L^2(\Gamma)} \le C\vep^{\frac 5 2}\|f\|_{H_0^{10}\left(\R_+; L_{\alpha}^2(\R^3)\right)}.
\end{align*}
This, together with \eqref{eq:18}, \eqref{eq:19} and Lemma \ref{le:9} yields
\begin{align} \label{eq:105}
&\sup_{\tau \in (0, T/\vep)}\left\|r_{3}(\phi_{\vep^{-1}}(\cdot), \vep^{-1}\tau)\right\|_{L^2_{-\alpha}(\R^3)} \le C\vep^{\frac 3 2} \|f\|_{H_0^{10}\left(\R_+; L_{\alpha}^2(\R^3)\right)}.
\end{align}

\textbf{Step 4:} In this step, we estimate $r_{4}$. From \eqref{eq:118}, \eqref{eq:122} and Lemma \ref{le:9}, it can be deduced that 
\begin{align} \label{eq:106}
\sup_{\tau \in (0, T/\vep)} \left\|r_{4}(\phi_{\vep^{-1}}(\cdot), \vep^{-1} \tau)\right\|_{L^2_{-\alpha}(\R^3)} \le C\vep^{\frac 3 2} \|f\|_{H_0^{2}\left(\R_+; L_{\alpha}^2(\R^3)\right)}.
\end{align}

\textbf{Step 5:} With the aid of \eqref{eq:23} and \eqref{eq:113}, it can be deduced from \eqref{eq:102}, \eqref{eq:103}, \eqref{eq:104}, \eqref{eq:105} and \eqref{eq:106} that the asymptotic expansion \eqref{eq:120} with the remainder term satisfying \eqref{eq:121} holds.

\end{proof}

{\color{HW}
\begin{remark}\label{Why-we-need-higher-regularity}
Our argument relies on solving a system of fifth-order differential equations, and the key step is to control the right -hand-side remainder, ensuring that its interval over $(0,T/\vep^2)$ is $O(\vep^{7/2})$ (see \eqref{eq:161}). Estimating the term associated with $\mathcal Q\partial_t^4 \partial_\nu u^f_\vep$ in $\textrm{Err}_{j_1}^{(1)}$ (see \eqref{eq:163}) and the term $\vep^2 \partial^4_t u^f_\vep$ in $\textrm{Err}_{j_1}^{(2)}$ (see \eqref{eq:164}), with $j_1 \in \{0,1,2,3\}$, requires control up to 16 time derivatives (see \eqref{eq:158} and \eqref{eq:162}). Specifically, by statement \eqref{f3} of Lemma \ref{le:1}, we have the a priori bounds 
\begin{align*}
\sup_{\tau\in(0,T/\vep^2)}\|\mathcal Q \partial_t^j \partial_\nu u^f_\vep(\cdot,\tau)\|_{L^2(\Gamma)} = O(\vep^{j+3/2}), \; \mathrm{and}\; \vep^2 \sup_{\tau\in(0,T/\vep^2)}\|\partial_t^{j+2} u^f_\vep(\cdot,\tau)\|_{L^2(\Omega)} = O(\vep^{j+3/2}),
\end{align*}
provided $j+12$ time derivatives are bounded. Hence,  integrating over $(0,T/\vep^2)$ is $O(\vep^{j-1/2})$. Choosing $j=4$ (the minimal choice) requires control of order $16$. Accordingly, we assume $H^{16}$ regularity of the source. 
\end{remark}}

\section{Proofs of Lemmas \ref{le:0}, \ref{le:1}, \ref{le:3} and \ref{le:9}} \label{sec:3}

{\color{HW}This section is devoted to proving Lemmas \ref{le:0}-\ref{le:9}. Proofs of Lemmas \ref{le:0}--\ref{le:1} appear in Sections \ref{sec:3.1}-\ref{sec:3.2}, while the proof of Lemmas \ref{le:3}-\ref{le:9} appear in Section \ref{sec:3.3}}.

\subsection{Proof of Lemma \ref{le:0}} \label{sec:3.1}
Before proving Lemma \ref{le:0}, we introduce the Fourier-Laplace transform.

For any $\phi \in H_{0,\sigma_0}^p (\R_+; L^2(\R^3))$ with $\sigma_0 > 0$, its Fourier-Laplace transform  $\hat \phi$ is defined by
\begin{align}\label{eq:7}
\left(\mathcal F \phi(x, \cdot) \right)(s) := \hat \phi (x,s)  = \int^{+\infty}_{0} e^{-st} \phi(x,t) dt, \quad s = \sigma + i\xi, \; \sigma > \sigma_0, \; \xi \in \R.
\end{align}
It is well known that 
\begin{align}\label{eq:8}
\phi(x,t) = \frac{e^{\sigma t}}{2\pi}\int^{+\infty}_{-\infty} e^{i\xi t} \hat \phi(x,\sigma + i\xi) d\xi = \frac{1}{2\pi i}\int_{\sigma - i\infty}^{\sigma + i\infty} e^{st} \hat \phi(x,s) ds.
\end{align}
Furthermore, it can be deduced from \eqref{eq:7} and \eqref{eq:8} that
\begin{align}
\partial^j_t \phi(x,t) = \frac{e^{\sigma t}}{2\pi}\int^{+\infty}_{-\infty} e^{i\xi t} (\sigma + i \xi)^j \hat \phi (x,\sigma + i\xi) d\xi, \quad j= 1,2,\ldots,p. \label{eq:17}
\end{align}

Now we are ready to give the proof of Lemma \ref{le:0}.

\begin{proof}[Proof of Lemma \ref{le:0}]
Recall that $v_\vep^f(x,t):= v^f(y_0 + \vep (x-y_0), \vep t)$ for $x\in \R^3$ and $t \in \R_+$. A straightforward calculation gives 
\begin{align}
&\left\|\partial^j_t v^f_\vep(\cdot, t)\right\|_{L^2(\Omega)} = \vep^{j-\frac 3 2} \big\|\partial^j_t v^f(\cdot, \vep t)\big\|_{L^2(\Omega_\vep)}, \label{eq:96}\\
&\left\|\partial^j_t \partial_{x_p} v^f_\vep(\cdot, t)\right\|_{L^2(\Omega)} = \vep^{j-\frac 1 2} \big\|\partial^j_t\partial_{x_p} v^f(\cdot, \vep t)\big\|_{L^2(\Omega_\vep)},\label{eq:97}\\
&\left\|\partial^j_t \partial^2_{x_qx_p} v^f_\vep(\cdot, t)\right\|_{L^2(\Omega)} = \vep^{j+\frac 1 2} \big\|\partial^j_t \partial^2_{x_qx_p} v^f(\cdot, \vep t)\big\|_{L^2(\Omega_\vep)},\;\; p,q\in\{1,2,3\}, \;\; t \in \R_+. \label{eq:46}
\end{align}
Therefore, it suffices to estimate $v^f$ and its derivative.  The rest of the poof consists of two parts. The first part involves proving statement \eqref{b1} and the second part addresses statement \eqref{b2}.
 
\textbf{Part I}: First, we prove \eqref{eq:92}. We observe from \eqref{eq:96} that 
\begin{align}
\left\|\partial^j_t v^f_\vep \right\|^2_{L^2\left((0,T/\vep^2); L^2(\Omega)\right)} 
&\le {C\vep^{2j-4}}\left\|\partial^j_t v^f\right\|^2_{L^2\left((0,T/\vep); L^2(\Omega_\vep)\right)}. \label{eq:94}
\end{align}
Since $v^f$ solves equations \eqref{eq:9}--\eqref{eq:20}, we easily obtain 
\begin{align*}
\partial^j_tv^f(x,t) = \int_{\R^3}\int_\R G(x-y,t-\tau)\partial^j_t f(y,\tau)d\tau dy \quad \textrm{for}\; (t,x) \in \R_+\times\R^3.
\end{align*}
Thus, the Fourier-Laplace transform of $\partial^j_t v^f$ satisfies
\begin{align*}
\left(\mathcal F \partial^j_t v^f\right)(x,s) = -\rho_0\left(R_{is/c_0} \mathcal F \partial^j_t f(\cdot, s)\right)(x).
\end{align*}
Here, the operator $R_s$ is defined by \eqref{eq:90}. This, together with \eqref{eq:17} yields
\begin{align} \label{eq:99}
\partial^j_t v^f(x,t) =  -\frac{e^{\vep t}}{2\pi}\int_{\R} e^{i\xi t} (\vep + i \xi)^j \rho_0\left(R_{(i\vep - \xi)/c_0} \mathcal F f(\cdot, \vep + i \xi)\right)(x) d\xi.
\end{align}
Using \eqref{eq:24}, \eqref{eq:29}, \eqref{eq:99} and Plancherel theorem, we obtain
\begin{align}
&\int_{\R} e^{-2\vep\tau} \left\|\partial^j_t v^f(\cdot, \tau)\right\|^2_{H^2(B_1(y_0))} d\tau = \int_{\R}|\vep + i\xi|^{2j} \left\|\left(\mathcal F v^f\right)(\cdot, \vep + i\xi)\right\|^2_{H^2(B_1(y_0))}d\xi \notag\\
&\qquad\qquad \qquad \qquad \qquad \qquad \qquad\; \le C\int_{\R}\left(1+\left|\vep + i\xi\right|^{2}\right)\left|\vep + i\xi\right|^{2j}\left\|\hat f(\cdot, \vep + i\xi)\right\|^2_{L_{\alpha}^2(\R^3)}d\xi \notag \\
&\qquad\qquad\qquad\qquad \qquad \qquad \qquad\; \le C\left\|\partial^{j}_t f\right\|^2_{H_{0,\vep}^{1}\left(\R_+; L_{\alpha}^2(\R^3)\right)}. \label{eq:100}
\end{align}
Clearly,
\begin{align} \label{eq:129}
\left\|\partial^j_t v^f \right\|^2_{L^2\left((0,T/\vep); {H^2(B_1(y_0))}\right)} \le C\int_{\R} e^{-2\vep\tau} \left\|\partial^j_t v^f(\cdot, \tau)\right\|^2_{H^2(B_1(y_0))} d\tau.
\end{align}
Therefore, by using inequality \eqref{eq:129} and the subsequent Sobolev inequality {\color{HW}(see, e.g., \cite[Section 5.6.3]{E10})}
\begin{align} \label{eq:33}
\sup_{x\in D} |\phi(x)| + \sup_{x,y\in D}\frac{|\phi(x)-\phi(y)|}{|x-y|^{1/2}} \le C_D \|\phi\|_{H^{2}(D)}\quad \textrm{for any compact set} \; D \subset \R^3, 
\end{align}
we have 
\begin{align*}
\left\|\partial^j_t v^f\right\|^2_{L^2\left((0,T/\vep); L^2(\Omega_\vep)\right)} \le C \vep^{3}\int^{{T}/{\vep}}_0 e^{-2\vep\tau} \big\|\partial^j_t v^f(\cdot,\tau)\big\|^2_{H^2(B_1(y_0))}d\tau \le C \vep^3 \left\|\partial^{j}_t f\right\|^2_{H_{0,\vep}^{1}\left(\R_+; L_{\alpha}^2(\R^3)\right)}. 
\end{align*}
In conjunction with \eqref{eq:94}, we obtain \eqref{eq:92}.

Second, we prove \eqref{eq:61} and \eqref{eq:93}. With the aid of \eqref{eq:24}, \eqref{eq:29}, \eqref{eq:99} and Cauchy-Schwartz inequality, we have
\begin{align}
\big\|\partial^j_t v^f(\cdot,t)\big\|_{H^2(B_1(y_0))} & \le C \int_\R\left|\vep + i\xi\right|^{j}\left\|R_{(i\vep - \xi)/c_0} \hat f(\cdot, \vep + i\xi)\right\|_{H^2(B_1(y_0))} d\xi \notag\\
&\le C \int_\R \left(1 + \left|\vep + i\xi\right|\right) \left|\vep + i\xi\right|^{j}\left\|\hat f(\cdot, \vep + i\xi)\right\|_{L_{\alpha}^2(\R^3)} d\xi \notag \\
& \le C \left(\int_\R \left(1 + \left|\vep + i\xi\right|\right)^4\left|\vep + i\xi\right|^{2j}\left\|\hat f(\cdot, \vep + i\xi)\right\|^2_{L_{\alpha}^2(\R^3)} d\xi\right)^{\frac 12} \notag \\
& \le C \left\|\partial^{j}_t f\right\|_{H_{0,\vep}^{2}\left(\R_+; L_{\alpha}^2(\R^3)\right)}, \quad t\in \left(0, {T}{\vep^{-1}}\right]. \label{eq:69}
\end{align}
Combining \eqref{eq:96}, \eqref{eq:33} and \eqref{eq:69} gives \eqref{eq:61}. Furthermore, employing \eqref{eq:97}, \eqref{eq:46}, \eqref{eq:69} and the fact that $L^6(\Omega) \subset H^1(\Omega)$, we obtain
{\color{HW}
\begin{align}
&\big\|\partial^j_t\partial_{x_p} v^f_\vep(\cdot, t)\big\|_{H^1(\Omega)} \le  C \vep^{j+\frac 12}\left\|\partial^{j}_t f\right\|_{H_{0,\vep}^{2}\left(\R_+; L_{\alpha}^2(\R^3)\right)}, \quad p\in \{1,2,3\}, \; t\in (0, T\vep^{-2}).\notag
\end{align}
In conjunction with the classical trace theorem $H^1(\Omega) \xrightarrow{\ \mathrm{trace}\ } H^{1/2}(\Gamma)$ (see, e.g., \cite[Theorem 3.37]{WM-00}), we obtain \eqref{eq:93}.}

\textbf{Part II}: First, we prove \eqref{eq:137} and \eqref{eq:30}.
It is easy to verify that 
\begin{align}\label{eq:135}
& \int_{\R}\left\|\frac{d\left[(\vep + i \xi)^j\mathcal F f(\cdot, \vep + i \xi)\right]}{d\xi}\right\|^2_{L_{\alpha}^2(\R^3)}d\xi \le \left\|\partial^{j}_t f\right\|^2_{{\color{HW}W_1^{0,2}}\left(\R_+; L_{\alpha}^2(\R^3)\right)}.
\end{align}
Furthermore, it is well-established that (see \cite[Theorem 16.1]{KK-12})
\begin{align*}
\left\|\frac{d R_z}{dz}\right\|_{\mathcal L\left(L^2_{\alpha}(\R^3),H^2_{-\alpha}(\R^3)\right)} \le \widetilde C|z|, \quad z\in \mathbb C_+.
\end{align*}
Here, $\widetilde C$ is a positive constant independent of $z$.
From this, with the aid of Plancherel theorem, \eqref{eq:24}, \eqref{eq:29}, \eqref{eq:99} and \eqref{eq:135}, we have
\begin{align} \label{eq:138}
\left\|\partial^j_t v^f\right\|^2_{W_1^{0,2}\left((0,T/\vep); H^2(B_1(y_0))\right)} &\le \int_{\R} \rho_0\left\|\frac{d\left[(\vep + i \xi)^j R_{(i\vep - \xi)/c_0} \mathcal F f(\cdot, \vep + i \xi)\right]}{d\xi}\right\|^2_{H^2(B_1(y_0))}d\xi \notag \\
&\le C\left\|\partial^{j}_t f\right\|^2_{{\color{HW}W_1^{1,2}}\left(\R_+; L_{\alpha}^2(\R^3)\right)}.
\end{align}
We note that by Cauchy-Schwartz inequality,
\begin{align}\label{eq:142}
\left\|\partial^j_t v^f\right\|_{L^1\left((0,T/\vep); H^2(B_1(y_0))\right)} \le C\left\|\partial^j_t v^f\right\|_{{\color{HW}W_1^{0,2}}\left((0,T/\vep); H^2(B_1(y_0))\right)}.
\end{align}
Moreover, a straightforward calculation gives that for each $x \in \Omega$,
\begin{align*}
\int_0^{T/\vep^2}  \left|\partial^{j}_t v^f_\vep(x, \tau) - \vep^j \partial^j_t v^f(y_0, \vep\tau) \right|  d\tau = \vep^{j-1}\int_0^{T/\vep} \left|\partial^j_t v^f(\vep(x-y_0)+y_0, \tau)-\partial^j_t v^f(y_0,\tau)\right|d\tau,
\end{align*}
whence \eqref{eq:137} and \eqref{eq:30} follow from \eqref{eq:33}, \eqref{eq:138} and \eqref{eq:142}. 

Second, we prove \eqref{eq:136}. From \eqref{eq:97} and \eqref{eq:46}, we find
\begin{align*}
&\left\|\partial^j_t \partial_{x_p} v_\vep^f\right\|^2_{L^1\left((0,T/\vep^2); L^2(\Omega)\right)} \le {C\vep^{2j-3}}\left\|\partial^j_t \partial_{x_p}v^f\right\|^2_{W_1^{0,2}\left((0,T/\vep); L^2(\Omega_\vep)\right)},\\
&\left\|\partial^j_t \partial^2_{x_qx_p} v_\vep^f\right\|^2_{L^1\left((0,T/\vep^2); L^2(\Omega)\right)} \le C\vep^{2j-1}\left\|\partial^j_t  \partial^2_{x_qx_p} v^f\right\|^2_{{\color{HW}W_1^{0,2}}\left((0,T/\vep); L^2(\Omega_\vep)\right)}.
\end{align*}
Therefore, we can apply \eqref{eq:138} and Sobolev embedding $L^6(\Omega) \subset H^1(\Omega)$ to obtain
{
\begin{align*}
\left\|\partial^j_t \partial_{x_{p}}v_\vep^f\right\|^2_{L^1\left((0,T/\vep^2); H^1(\Omega)\right)} \le C \vep^{2j-1} \left\|\partial^{j}_t f\right\|^2_{{\color{HW}W^{1,2}_{1}}\left(\R_+; L_{\alpha}^2(\R^3)\right)}, \quad p\in\{1,2,3\},
\end{align*}
whence \eqref{eq:136} follows.}
The proof of this lemma is thus completed.
\end{proof} 

\subsection{Proof of Lemma \ref{le:1}} \label{sec:3.2}
 
To prove Lemma \ref{le:1}, we first establish statement \eqref{f1} of Lemma \ref{le:1} and then proceed to prove the subsequent statements sequentially, each based on the previous ones.

{\color{HW} Before proving statement \eqref{f1} of Lemma \ref{le:1}, we introduce several new functions spaces. Let $\mathbb H_0(\textrm{div}0,\Omega)$ be the $\mathbb L^2-$closure of the compactly supported divergence-free smooth functions in $\Omega$ and define
\begin{align*}
\mathbb H_0(\textrm{curl}0,\Omega) := \{\nabla \phi, \phi \in H^1_0(\Omega)\}, \quad  \nabla \mathbb H_{\textrm{{arm}}}(\Omega):= \{\nabla \psi, \psi \in H^1(D)\; \textrm{and}\; \Delta \psi = 0 \}.
\end{align*}
It is known that $\mathbb L^2(\Omega)$ has the following orthogonal decomposition (see e.g., \cite{V-94})
\begin{align*}
\mathbb L^2(\Omega) = \mathbb H_0(\textrm{div}0,\Omega)\; \oplus\; \mathbb H_0(\textrm{curl}0,\Omega)\; \oplus  \; \nabla \mathbb H_{\textrm{{arm}}}(\Omega).
\end{align*}
Define
\begin{align} \label{eq:47}
(\mathcal M \psi) (x):=\int_\Omega \nabla_y \left(\frac{1}{ 4\pi|x-y|}\right)\cdot \psi(y)dy, \quad x \in \Omega.
\end{align}
We note that $\nabla\mathcal M$ acts as the zero operator on space $\mathbb H_0(\textrm{div}0,\Omega)$ and as the identity operator on the space $\mathbb H_0(\textrm{curl}0,\Omega)$. In addition, the eigenfunctions $\left\{e_n^{(3)}\right\}_{n = 1}^{\infty} of$ $\nabla \mathcal M$ with the corresponding eigenvalue $\left\{\lambda_n^{(3)}\right\}_{n = 1}^{\infty}$ serve as a complete orthonormal basis in $\nabla \mathbb H_{\textrm{arm}}(\Omega)$, see \cite{V-94} for instance. Let $\left\{e_n^{(1)}\right\}_{n = 1}^{\infty}$ and $\left\{e_n^{(2)}\right\}_{n = 1}^{\infty}$ be the orthonormal basis of the space  $\mathbb H_0(\textrm{div}0,\Omega)$ and $\mathbb H_0(\textrm{curl}0,\Omega)$, respectively. Thus, for $\phi \in H^1(\Omega)$, we have 
\begin{align} 
&\nabla \phi = \sum^\infty_{n=1}\left[\left\langle \nabla \phi, e_n^{(1)}\right\rangle_{\mathbb L^2(\Omega)} e_n^{(1)} + \left\langle \nabla \phi, e_n^{(2)}\right\rangle_{\mathbb L^2(\Omega)} e_n^{(2)} + \left\langle \nabla \phi, e_n^{(3)}\right\rangle_{\mathbb L^2(\Omega)} e_n^{(3)}\right], \label{eq:42}\\
&\|\nabla \phi \|^2_{\mathbb L^2(\Omega)} = \sum^\infty_{n=1}\left[\left\langle \nabla \phi, e_n^{(1)}\right\rangle_{\mathbb L^2(\Omega)}^2 + \left\langle \nabla \phi, e_n^{(2)}\right\rangle_{\mathbb L^2(\Omega)}^2 + \left\langle \nabla \phi, e_n^{(3)}\right\rangle_{\mathbb L^2(\Omega)}^2\right].\label{eq:43}
\end{align}
The spectral properties of the operator $\mathcal M$ is central to the derivation of \eqref{eq:31}.
}

\begin{proof}[Proof of statement \eqref{f1} of Lemma \ref{le:1}]
We observe that 
\begin{align}
u^f_\vep(x,t) &= v^f_\vep(x,t)+ \left(u^f_\vep(x,t) - v^f_\vep(x,t)\right) \notag\\
&=:  v^f_\vep(x,t) + w^f_\vep(x,t) \quad \textrm{for}\;  x\in \R^3 \; \textrm{and}\; t\in \R_+.\label{eq:52}
\end{align}
The rest of this proof consists of three parts.

\textbf{Part I}: In this part, we prove \eqref{eq:3}. Define 
\begin{align}\label{eq:169}
l_{\sigma} :=\{\sigma + i \xi: \sigma \in \R_+ \; \textrm{and}\; \xi \in \R\}\;\; \mathrm{for}\; \sigma>0.
\end{align}
Let $w^f(x,t): = u^f(x,t) - v^f(x,t)$ for $x\in \R^3$ and $t\in \R_+$. Clearly,
\begin{align}
&w^f_\vep(x,t): = w^f(\vep (x-y_0) + y_0,\vep t), \quad \textrm{for}\; x\in \R^3,\;t\in \R_+.  \label{eq:14}
\end{align}
By equations \eqref{eq:9}--\eqref{eq:20} and \eqref{eq:12}--\eqref{eq:21}, we easily find 
\begin{align} \label{eq:166}
\frac{1}{k_\vep}\partial_{tt} w^f - \nabla \cdot \frac{1}{\rho_\vep} \nabla w^f = \left(\frac{1}{k_0} - \frac{1}{k_\vep}\right)\partial_{tt} v^f - \nabla \cdot \left(\frac{1}{\rho_0} - \frac{1}{\rho_\vep}\right)\nabla v^f =: h^f \quad \textrm{in}\; \R^3 \times \R.
\end{align} 
We note that from Lemma \ref{le:2}, $w^f \in H_{0,\sigma}^{p+1} \left(\R_+; H^1(\R^3)\right)$.
It can be seen that the Laplace-Fourier transform $\hat w^f$ of $w^f$ satisfies  
\begin{align}\label{eq:5}
&\frac{s^2}{k_\vep(x)}\hat w^f(x,s) - \nabla \cdot \frac{1}{\rho_\vep(x)} \nabla \hat w^f(x,s) = \hat h^f(x,s), \quad x\in \R^3, \; s \in l_{\sigma}.
\end{align}
Multiplying $s^l \overline{s^{l+1} \hat w^f}$ with $l\in \{0,1,\ldots,p\}$ on both sides of \eqref{eq:5}, we arrive at 
\begin{align}\label{eq:2}
s\int_{\R^3}\frac{|s|^{2l+2}}{k_\vep(x)}|\hat w^f(x,s)|^2 &+ \bar s \frac{|s|^{2l}}{\rho_\vep(x)}\left|\nabla \hat w^f(x,s)\right|^2 dx = \int_{\R^3} \bar s|s|^{2l}\hat h^f(x,s)\overline{\hat w^f(x,s)}dx, \quad s \in l_{\sigma}.
\end{align}
Since for $s \in l_\sigma$,
\begin{align} 
&\left \langle \bar s |s|^{2l}\hat h^f(\cdot, s), \hat w^f(\cdot, s) \right\rangle_{H^{-1}(\R^3), H^1(\R^3)}^2 \notag\\ 
&\le \vep^{-4}\left[\|s^{l+2} \hat v^f(\cdot,s)\|^2_{L^2(\Omega_\vep)}\|s^{l+1} \hat w^f(\cdot,s)\|^2_{L^2(\Omega_\vep)} + \| s^{l+1}\nabla \hat v^f(\cdot,s)\|^2_{\mathbb L^2(\Omega_\vep)}\|s^l\nabla \hat w^f(\cdot,s)\|^2_{\mathbb L^2(\Omega_\vep)} \right],\notag
\end{align}
with the aid of the identity \eqref{eq:2} and the fact that $|\textrm{Re}(s)| > \sigma$ for $s\in l_\sigma$, we have
\begin{align}
&\|s^{l+1}\hat w^f(\cdot,s)\|_{L^2(\Omega_\vep)} + \|s^{l} \nabla \hat w^f(\cdot,s)\|_{\mathbb L^2(\Omega_\vep)} \le \frac{C}{\sigma}\left[\|s^{l+2} \hat v^f(\cdot,s)\|_{L^2(\Omega_\vep)} + \| s^{l+1}\nabla \hat v^f(\cdot,s)\|_{\mathbb L^2(\Omega_\vep)}\right]. \label{eq:13}
\end{align}
Here, $\hat v^f$ is the Laplace-Fourier transform of $v^f$.
By utilizing \eqref{eq:17} and \eqref{eq:13} with $\sigma = \vep$, we arrive at 
\begin{align}
\int^{T/\vep}_0 e^{-2\vep \tau} \|\partial^{j}_t w^f(\cdot,&\tau)\|^2_{L^2(\Omega_\vep)}  d\tau \le \int^{+\infty}_0 e^{-2\vep \tau} \int_{\Omega_\vep} |\partial^{j}_{t} w^f(x,\tau)|^2 dx d\tau \notag\\\
& = \frac{1}{4\pi^2} \int_{\Omega_\vep} \int_{\R} \left|(\vep + i \xi)^{j} \hat w^f(x,\vep + i\xi)\right|^2 d\xi dx \notag\\
& \le C\vep^{-2}\int_{\R} e^{-2\vep \tau} \left[\|\partial^{j+1}_t v^f(\cdot,\tau)\|_{L^2(\Omega_\vep)}^2 + \|\partial^{j}_t v^f(\cdot,\tau)\|_{H^1(\Omega_\vep)}^2 \right]d\tau. \label{eq:170}
\end{align}
Furthermore, using  \eqref{eq:9} and \eqref{eq:100}, we have 
\begin{align}
\int_{\R} e^{-2\vep \tau} \|\partial^{j+1}_t v^f(\cdot,\tau)\|_{L^2(\Omega_\vep)}^2 d\tau \le \left\|\partial^{j-1}_t f\right\|^2_{H_{0,\vep}^{1}\left(\R_+; L_{\alpha}^2(\R^3)\right)}. \label{eq:171}
\end{align}
Using the inequality $\left\|R_z\right\|_{\mathcal L\left(L^2_{\alpha}(\R^3),H^1_{-\alpha}(\R^3)\right)} \le \widetilde C$, uniformly in $z \in \CC_+$ (see \cite[Theorem 16.1]{KK-12}), and proceeding as in the derivation of \eqref{eq:100}, we obtain
\begin{align} \label{eq:172}
\int_{\R} e^{-2\vep \tau} \|\partial^{j}_t  v^f(\cdot,\tau)\|_{H^1(\Omega_\vep)}^2 d\tau \le \left\|\partial^{j}_t f\right\|^2_{H_{0,\vep}^{0}\left(\R_+; L_{\alpha}^2(\R^3)\right)}.
\end{align}
By \eqref{eq:14}, we have 
\begin{align}
&\int^{T/\vep^{2}}_0 e^{-2\vep^{2}\tau} \left\|\partial^{j}_t w_\vep^f(\cdot,\tau)\right\|^2_{L^2(\Omega)} d\tau \le \vep^{2j-4}\int^{T/\vep}_0 e^{-2\vep\tau}\left\|\partial^{j}_t w^f(\cdot,\tau)\right\|^2_{L^2(\Omega_\vep)} d\tau. \notag
\end{align}
Combining this with \eqref{eq:170}, \eqref{eq:171} and \eqref{eq:172} yields 
\begin{align}
&\left\|\partial^j_t w^f_\vep \right\|^2_{L^2\left((0,T/\vep^2); L^2(\Omega)\right)}  \le C \vep^{2j-{6}} \left\|\partial^{j-1}_t f\right\|^2_{H_{0,\vep}^{1}\left(\R_+; L_{\alpha}^2(\R^3)\right)} \label{eq:88}.
\end{align}
Therefore, using \eqref{eq:52}, \eqref{eq:94}, \eqref{eq:172} and \eqref{eq:88}, we obtain \eqref{eq:3}.

\textbf{Part II}: In this part, we derive
{\color{HW}
\begin{align} \label{eq:165}
\left\|\partial^j_t \nabla u^f_\vep \right\|^2_{L^2\left((0,T/\vep^2); \mathbb L^2(\Omega)\right)} + \left\|\partial^j_t \Delta u^f_\vep \right\|^2_{L^2\left((0,T/\vep^2); L^2(\Omega)\right)} \le C\vep^{2j-2}\left\|\partial^{j}_t f\right\|^2_{H_{0,\vep}^{3}\left(\R_+; L_{\alpha}^2(\R^3)\right)}.
\end{align}}

First, we give the $L^2\left((0,T/\vep^2); \mathbb L^2(\Omega)\right)$-norm of $\partial^j_t\nabla u^f_\vep$. By using similar arguments as in the derivation of \eqref{eq:88}, we easily obtain
\begin{align} \label{eq:124}
&\left\|\partial^j_t \nabla w^f_\vep \right\|^2_{L^2\left((0,T/\vep^2); \mathbb L^2(\Omega)\right)}  \le C\vep^{2j-4} \left\|\partial^{j}_t f\right\|^2_{H_{0,\vep}^{1}\left(\R_+; L_{\alpha}^2(\R^3)\right)}.
\end{align}
Furthermore, similarly as in the derivation of \eqref{eq:92}, we can use \eqref{eq:97}, \eqref{eq:100} and the fact that $L^6(\Omega) \subset H^1(\Omega)$ to obtain
\begin{align}\label{eq:101}
\left\|\partial^j_t \partial_{x_p} v^f_\vep \right\|^2_{L^2\left((0,T/\vep^2); H^1(\Omega)\right)}  \le C\vep^{2j} \left\|\partial^{j}_t f\right\|^2_{H_{0,\vep}^{1}\left(\R_+; L_{\alpha}^2(\R^3)\right)}, \quad p\in\{1,2,3\}.
\end{align}
Combining \eqref{eq:52}, \eqref{eq:124} and \eqref{eq:101}  yields
\begin{align}
\left\|\partial^j_t \nabla u^f_\vep \right\|^2_{L^2\left((0,T/\vep^2); \mathbb L^2(\Omega)\right)} \le C\vep^{2j-4} \left\|\partial^{j}_t f\right\|^2_{H_{0,\vep}^{1}\left(\R_+; L_{\alpha}^2(\R^3)\right)}. \label{eq:107}
\end{align}

Second, we improve the $L^2\left((0,T/\vep^2); \mathbb L^2(\Omega)\right)$-norm of $\partial^j_t\nabla u^f_\vep$ in \eqref{eq:107}. 
By applying the space gradient and the $j-$th time derivative on both sides of equation \eqref{eq:51}, we find
\begin{align}
\partial^{j}_t \nabla w_\vep^f(x,t) & = -\rho_0 \partial^{j}_{t}\nabla \int_{\Omega}\left(\frac{1}{k_1\vep^2}-\frac{1}{k_0} \right)\frac{\partial_{tt} u_\vep^f(y,t-c_0^{-1}|x-y|)}{4\pi|x-y|}dy \notag \\
& + \left(\frac{\rho_0}{\rho_1\vep^2} - 1\right) \partial^{j}_t \nabla \textrm{div} \int_{\Omega} \frac{\nabla u^f_\vep (y,t-c^{-1}_0|x-y|)}{4\pi|x-y|}dy, \quad x\in \R^3\backslash \Gamma,\; t\in\ \R_+. \notag 
\end{align}
Thus, we have 
\begin{align} \label{eq:48}
 \partial^{j}_t \nabla u_\vep^f + \left(\frac{\rho_0}{\rho_1\vep^2} - 1 \right) \nabla \left(\mathcal M \partial^{j}_t \nabla  u^f_\vep\right) & =  \partial^{j}_t \nabla v_\vep^f +  h_j + g_j, \quad \mathrm{in}\; \Omega \times \R_+,
\end{align}
where the operator $\mathcal M$ is given by \eqref{eq:47}, and the functions $h_j$ and $g_j$ are defined by 
\begin{align} 
&h_j(x,t) := -\rho_0 \partial^{j}_t\nabla \int_{\Omega} \left(\frac{1}{k_1\vep^2}-\frac{1}{k_0} \right)\frac{\partial_{tt} u_\vep^f(y,t-c_0^{-1}|x-y|)}{4\pi|x-y|}dy, \notag \\
& g_j (x,t) := \left(\frac{\rho_0}{\rho_1\vep^2} - 1 \right) \partial^{j}_t\nabla \textrm{div}\int_\Omega \frac{\nabla u_\vep(y,t-c^{-1}_0|x-y|) - \nabla  u_\vep(y,t)}{4\pi|x-y|} dy,\quad x \in \R^3\backslash \Gamma,\; t\in \R_+.\notag 
\end{align}
By means of \eqref{eq:42}, \eqref{eq:43} and \eqref{eq:48}, we have 
\begin{align}\label{eq:56}
\left\|\partial^{j}_t \nabla u_\vep^f(\cdot,t)\right\|^2_{\mathbb L^2(\Omega)} = \sum^\infty_{n=1}\left[\left\langle\partial^{j}_t \nabla u_\vep^f(x,t), e_n^{(2)}(x)\right\rangle_{\mathbb L^2(\Omega)}^2 + \left\langle \partial^{j}_t \nabla u_\vep^f(x,t), e_n^{(3)}(x)\right\rangle_{\mathbb L^2(\Omega)}^2\right].
\end{align}
Here, $\left\langle\partial^{j}_t \nabla u_\vep^f(x,t), e_n^{(2)}(x)\right\rangle_{\mathbb L^2(\Omega)}$ and $\left\langle \partial^{j}_t \nabla u_\vep^f(x,t), e_n^{(3)}(x)\right\rangle_{\mathbb L^2(\Omega)}$  satisfy 
\begin{align}
&\left\langle \partial^{j}_t \nabla u_\vep^f(x,t), e_n^{(2)}(x)\right\rangle_{\mathbb L^2(\Omega)} = \frac{\rho_1\vep^2}{\rho_0} \left\langle\partial^{j}_t \nabla v_\vep^f(x,t) + h_j(x,t) + g_j(x,t), e_n^{(2)}(x)\right\rangle_{\mathbb L^2(\Omega)}\label{eq:57}
\end{align}
and 
\begin{align}
&\left\langle \partial^{j}_t \nabla u_\vep^f(x,t), e_n^{(3)}(x)\right\rangle_{\mathbb L^2(\Omega)} = \frac{\rho_1\vep^2 \left\langle \partial^{j}_t \nabla v_\vep^f(x,t) + h_j(x,t) + g_j(x,t), e_n^{(3)}(x)\right\rangle_{\mathbb L^2(\Omega)}}{\lambda_n^{(3)}(\rho_0-\rho_1\vep^2) + \rho_1\vep^2},\label{eq:58}
\end{align}
respectively. 
For the estimate of $L^2\left((0,T/\vep^2); \mathbb L^2(\Omega)\right)$-norm of $h_j$, it can be seen that 
\begin{align}
 h_j(x,t) &= \partial^{j}_t \nabla \rho_0 \left(\frac{1}{k_1\vep^2}-\frac{1}{k_0}\right)\int_\Omega \frac{1}{4\pi|x-y|}\partial^2_t u_\vep^f(y,t-c^{-1}_0|x-y|)dy\notag\\
&= \rho_0 \left(\frac{1}{k_1\vep^2}-\frac{1}{k_0}\right)\bigg[\int_{\Omega} -\frac{x-y}{4\pi|x-y|^3}\bigg(\partial_t^{j+2} u_\vep^f(y,t-c^{-1}_0|x-y|)dy  \notag\\
&-c^{-1}_0 \int_{\Omega} \frac{x-y}{4\pi|x-y|^2}\partial^{j+3}_{t} u_\vep^f(y,t-c_0^{-1}|x-y|) dy \bigg], \quad x \in \Omega,\; t\in \R_+. \notag 
\end{align}
The last identity directly follows from the regularity properties of the operator $N_0$.
Therefore, it can be deduced from \eqref{eq:3} and statement \eqref{h2} of Lemma \ref{le:a3} that   
\begin{align}\label{eq:53}
\left\|h_j \right\|^2_{L^2\left((0,T/\vep^2); \mathbb L^2(\Omega)\right)} \le C \vep^{2j-{6}} \left\|\partial^{j+1}_t f\right\|^2_{H_{0,\vep}^{2}\left(\R_+; L_{\alpha}^2(\R^3)\right)}.
\end{align}
For the estimate of $L^2\left((0,T/\vep^2); \mathbb L^2(\Omega)\right)$-norm of $g_j$, with the aid of the regularity properties of the operator $N_0$, we can use Taylor expansions with respect to $t-$variable to obtain
\begin{align}
g_j (x,&t) = \left(\frac{\rho_0}{\rho_1\vep^2} -1 \right)\nabla \textrm{div}\int_\Omega \frac{1}{4\pi|x-y|}\bigg[\partial^j_t\nabla u_\vep(y,t-c^{-1}_0|x-y|) - \partial^{j}_t \nabla u^f_\vep(y,t)\bigg]dy \notag\\ 
& = \left(\frac{\rho_0}{\rho_1\vep^2} - 1\right)\nabla \textrm{div}\int_\Omega \frac{1}{4\pi|x-y|}\int^{t-c^{-1}_0|x-y|}_{t}\partial^{j+2}_{t} \nabla u_\vep^f(y,\tau)(t-c^{-1}_0|x-y|-\tau) d\tau dy. \notag 
 \end{align}
From this, utilizing \eqref{eq:107} and statement \eqref{h2} of Lemma \ref{le:a3}, we have
\begin{align}
\left\|g_j \right\|^2_{L^2\left((0,T/\vep^2); \mathbb L^2(\Omega)\right)}&\le C{\vep^{-4}}\left\|\partial^{j+2}_t \nabla v^f_\vep \right\|^2_{L^2\left((0,T/\vep^2); \mathbb L^2(\Omega)\right)}\notag \\
& \le C \vep^{2j-4}\left\|\partial^{j+1}_t f\right\|^2_{H_{0,\vep}^{2}\left(\R_+; L_{\alpha}^2(\R^3)\right)}. \label{eq:54}
\end{align} 
By adopting \eqref{eq:101}, \eqref{eq:56}, \eqref{eq:57}, \eqref{eq:58}, \eqref{eq:53} and \eqref{eq:54},  we arrive at 
\begin{align}
&\left\|\partial^j_t \nabla u^f_\vep \right\|^2_{L^2\left((0,T/\vep^2); \mathbb L^2(\Omega)\right)}\le C\vep^{2j-2} \left\|\partial^{j+1}_t f\right\|^2_{H_{0,\vep}^{2}\left(\R_+; L_{\alpha}^2(\R^3)\right)}. \label{eq:35}
\end{align}

Third, we estimate $L^2\left((0,T/\vep^2); L^2(\Omega)\right)$-norm of $\partial^j_t \Delta u^f_\vep$. Drawing upon \eqref{eq:10}, we have
\begin{align}
\left\|\partial^{j}_{t}\Delta u_\vep^f(\cdot,t)\right\|_{L^2(\Omega)} \le C \left(\vep^4 \left\|\partial^{j}_{t} f_\vep(\cdot,t)\right\|_{L^2(\Omega)} + \left\|\partial_t^{j+2} u_\vep^f(\cdot,t)\right\|_{L^2(\Omega)}\right), \;\;\; t \in \R_+. \notag
\end{align}
Combining this with \eqref{eq:122} and \eqref{eq:88} yields
\begin{align}
&\left\|\partial^j_t \Delta u^f_\vep \right\|^2_{L^2\left((0,T/\vep^2); L^2(\Omega)\right)}\le C \vep^{2j-2} \left\|\partial^{j}_t f\right\|^2_{H_{0,\vep}^{2}\left(\R_+; L_{\alpha}^2(\R^3)\right)}. \notag
\end{align}
Thus, in conjunction with \eqref{eq:35} gives \eqref{eq:165}.  

{\color{HW} 
\textbf{Part III.} In this part, we derive \eqref{eq:31}.

To do so, we introduce a new fixed bounded domain $\widetilde \Omega \subset \R^3$ with $\Omega \subset \widetilde \Omega$. 
Building upon \eqref{eq:12}--\eqref{eq:21}, proceeding as in the derivation of \eqref{eq:2}, we have 
\begin{align}
s\int_{\R^3}\frac{|s|^{2j+2}}{k_\vep(x)}|\hat u^f(x,s)|^2 &+ \bar s \frac{|s|^{2j}}{\rho_\vep(x)}\left|\nabla \hat u^f(x,s)\right|^2 dx = \int_{\R^3} \bar s|s|^{2j}\hat f(x,s)\overline{\hat u^f(x,s)}dx, \quad s \in l_{\vep}. \notag
\end{align}
Here, $\hat u^f$ and $\hat f$ are the Laplace-Fourier transforms of $u^f$ and $f$, and $l_\vep$ is specified by \eqref{eq:169}.
From this, using Cauchy-Schwartz equality, we obtain
\begin{align} \label{eq:98}
&\left\|s^{j+1}\hat u^f(\cdot,s)\right\|_{L^2(\R^3\backslash \Omega_\vep)} + \left\|s^{j} \nabla \hat u^f(\cdot,s)\right\|_{\mathbb L^2(\R^3\backslash \Omega_\vep)}\le \frac{1}{\vep}\|s^j\hat f(\cdot,s)\|_{L^2(\R^3)}, \quad s\in l_\vep.
\end{align}
In conjunction with \eqref{eq:17} and \eqref{eq:98}, we have 
\begin{align*}
&\int^{T/\vep}_0 e^{-2\vep \tau} \left(\|\partial^{j+1}_t u^f(\cdot,\tau)\|^2_{L^2(\R^3\backslash \Omega_\vep)} + \|\partial^{j}_t \nabla u^f(\cdot,\tau)\|^2_{L^2(\R^3\backslash \Omega_\vep)} \right)d\tau \\
&\le  C \vep^{-2}\int_{\R} e^{-2\vep \tau} \|\partial^{j}_t f(\cdot,\tau)\|_{L^2(\R^3)}^2 d\tau.
\end{align*}
This, together with the inequalities
\begin{align*}
&\int^{T/\vep^{2}}_0 e^{-2\vep^{2}\tau} \left\|\partial^{j}_t u_\vep^f(\cdot,\tau)\right\|^2_{L^2(\widetilde \Omega \backslash \Omega)} d\tau \le \vep^{2j-4}\int^{T/\vep}_0 e^{-2\vep\tau}\left\|\partial^{j}_t u^f(\cdot,\tau)\right\|^2_{L^2(\R^3\backslash\Omega_\vep)} d\tau
\end{align*}
and
\begin{align*}
&\int^{T/\vep^{2}}_0 e^{-2\vep^{2}\tau} \left\|\partial^{j}_t \nabla u_\vep^f(\cdot,\tau)\right\|^2_{\mathbb L^2(\widetilde \Omega \backslash \Omega)} d\tau \le \vep^{2j-2}\int^{T/\vep}_0 e^{-2\vep\tau}\left\|\partial^{j}_t \nabla u^f(\cdot,\tau)\right\|^2_{\mathbb L^2(\R^3\backslash\Omega_\vep)} d\tau
\end{align*}
yields
\begin{align}
&\left\|\partial^{j+1}_t u^f_\vep \right\|^2_{L^2\left((0,T/\vep^2); L^2(\widetilde \Omega\backslash \Omega)\right)}  \le C \vep^{2j-{4}} \left\|\partial^{j}_t f\right\|^2_{L_{0,\vep}^{2}\left(\R_+; L_{\alpha}^2(\R^3)\right)}, \label{eq:168}\\
&\mathrm{and}\;\left\|\partial^j_t \nabla u^f_\vep \right\|^2_{\mathbb L^2\left((0,T/\vep^2); L^2(\widetilde \Omega\backslash \Omega)\right)}  \le C \vep^{2j-{4}} \left\|\partial^{j}_t f\right\|^2_{L_{0,\vep}^{2}\left(\R_+; L_{\alpha}^2(\R^3)\right)}.\label{eq:38}
\end{align}
Furthermore, it follows from \eqref{eq:10}, \eqref{eq:122} and \eqref{eq:168} that 
\begin{align} \label{eq:167}
\left\|\partial^j_t \Delta u^f_\vep \right\|^2_{L^2\left((0,T/\vep^2); L^2\left(\widetilde \Omega \backslash \Omega \right)\right)}\le C\vep^{2j-2}\left\|\partial^{j}_t f\right\|^2_{H_{0,\vep}^{1}\left(\R_+; L_{\alpha}^2(\R^3)\right)}.
\end{align}

Moreover, we claim that 
\begin{align}
&\|\partial_\nu \phi\|_{L^2(\Gamma)} \le \widetilde C\big[\|\nabla \phi\|_{\mathbb L^2(\Omega)} + \|\Delta \phi\|_{L^2(\Omega)} + \theta\|\Delta \phi\|_{L^2(\widetilde \Omega \backslash \Omega)} + \theta \|\nabla \phi\|_{\mathbb L^2(\widetilde \Omega\backslash \Omega)}\big] \label{eq:32}\\
&  \qquad \qquad \mathrm{for}\; \mathrm{\phi} \in \left\{\phi \in H^{1}(\widetilde\Omega):\Delta\phi \in L^2(\widetilde \Omega \backslash \Omega) \cap L^2(\Omega)\; \mathrm{and}\;  u_+ = u_-,\; \theta \partial^+_\nu u_+ = \partial^-_\nu u_-\; \mathrm{on}\; \Gamma\right\},\notag
\end{align}
where $\theta \in \R_+$ and $\widetilde C$ is a positive constant independent of $\theta$. Its proof is deferred to Appendix \ref{sec:a3}. Then, in view of the fact that $\partial^j_t u^f_\vep \in H^1(\widetilde \Omega)$ and $\rho_1 \vep^2 \partial^+_\nu \partial^j_t u^f_\vep = \partial^-_\nu \partial^j_t u^f_\vep$ on $\Gamma$, using  \eqref{eq:165}, \eqref{eq:38}, \eqref{eq:167} and \eqref{eq:32}, we obtain \eqref{eq:31}.
}
\end{proof}

Based on statement \eqref{f1} of Lemma \ref{le:1}, we are ready to prove statement \eqref{f2} of Lemma \ref{le:1}.

\begin{proof}[Proof of statement \eqref{f2} of Lemma \ref{le:1}]
 The proof of this statement consists of two parts, the first part involves the derivation of \eqref{eq:75}, and the second part focuses on the proof of \eqref{eq:79}. Throughout the proof, $\eta_1,\eta_2$ are specified in \eqref{eq:114}, and $\Lambda_\vep(t)$ and $\gamma_\vep$ are given by \eqref{eq:73} and\eqref{eq:116}, respectively.

\textbf{Part I}: By \eqref{eq:84}, in order to prove \eqref{eq:75}, it suffices to prove
\begin{align}
|\partial^j_t \Lambda_\vep (t)| \le C \vep^{j + \frac 1 2} \|f\|_{H_0^{j + 8}\left(\R_+; L_{\alpha}^2(\R^3)\right)} \label{eq:36} \end{align}
and
\begin{align}
\left\|\left(\mathcal Q\partial^{j}_t\partial_{\nu}u_\vep^f\right)(\cdot,t)\right\|_{L^2(\Gamma)} \le C \vep^{j + \frac 1 2} \|f\|_{H_0^{j + 8}\left(\R_+; L^2_{\alpha}(\R^3)\right)}, \quad t\in(0,T\vep^{-2}]. \label{eq:37}
\end{align}

First, we prove \eqref{eq:36}. Setting $(j,q_0, q_1, \mathcal A) = (j,3,3,\mathcal P)$ in \eqref{eq:78}, and using \eqref{eq:118}, \eqref{eq:122}, Lemma \ref{le:8} and statement \eqref{h1} of Lemma \ref{le:a3}, we obtain
\begin{align}
&\eta_2 \partial^{j+3}_{t}\Lambda_\vep(t) + \eta_1 \partial^{j + 2}_{t}\Lambda_\vep(t) + \gamma_{\vep} \partial^{j}_t\Lambda_\vep(t) = \int_{\Omega} \frac{\rho_1\vep^2 \partial^{j + 2}_{t} v_\vep^f(y,t)}{(\rho_0-\rho_1\vep^2)|\Omega|} dy + \sum_{l=1}^{2}\textrm{Res}^{(1)}_{l}(t) \label{eq:25}
\end{align}
for $t\in \R_+$, where $\Lambda_\vep(t)$ is given by \eqref{eq:73}, and $\textrm{Res}^{(1)}_{1}$ and $\textrm{Res}^{(1)}_{2}$ satisfy
\begin{align*}
& \quad\; \textrm{Res}^{(1)}_{1}(t) = \frac{c^2_0}{|\Omega|}\sum^{3}_{l = 2} \frac{(-1)^{l-1}(l-1)}{c^{l}_0 l!} \left\langle 1, \left(K^*_l \mathcal Q\partial^{j + l}_t\partial_{\nu}u_\vep^f\right)(\cdot,t)\right\rangle_{L^2(\Gamma)}, \\
&\left|\textrm{Res}^{(1)}_{2}(t)\right| \le  C \sum^5_{l=4}\left[\left\|\partial^{j+l}_t \partial_\nu u^f_\vep \right\|_{L^2\left((0,T/\vep^2); L^2(\Gamma)\right)} + \vep^2 \left\|\partial^{j+l}_t u^f_\vep \right\|_{L^2\left((0,T/\vep^2); L^2(\Omega)\right)}\right] \\
& \qquad \quad  + C\vep^{j + \frac 92} \sum^2_{l=0} \left\|\partial^{j+l}_t f(\cdot,t)\right\|_{L^2(\Omega_\vep)} +  C\vep^{j + \frac {11}2} \left\|\partial^{j+3}_t f\right\|_{H^1_0\left((0,T/\vep^2); L^2(\Omega_\vep)\right)}, \quad t\in (0,T\vep^{-2}).
\end{align*}
Similarly, we can use \eqref{eq:78} with $(j,q_0, q_1, \mathcal A) = (j + 1, 2, 2, \mathcal P)$ to obtain
\begin{align}
&\eta_1 \partial^{j+3}_{t}\Lambda_\vep(t) + \gamma_{\vep} \partial^{j+1}_t\Lambda_\vep(t) = \int_{\Omega} \frac{\rho_1\vep^2 \partial^{j +3}_{t} v_\vep^f(y,t)}{(\rho_0-\rho_1\vep^2)|\Omega|} dy + \sum_{l=1}^{2}\textrm{Res}^{(2)}_{l}(t),\quad t\in \R_+, \label{eq:26}
\end{align}
where $\textrm{Res}^{(2)}_{1}$ and $\textrm{Res}^{(2)}_{2}$ satisfy
\begin{align*}
&\quad \;\textrm{Res}^{(2)}_{1}(t) = - \frac{1}{2|\Omega|} \left\langle 1, \left(K^*_2 \mathcal Q\partial^{j + 3}_t\partial_{\nu}u_\vep^f\right)(\cdot,t)\right\rangle_{L^2(\Gamma)},\\
&\left|\textrm{Res}^{(2)}_{2}(t)\right| \le C \sum^5_{l=4} \left[\left\|\partial^{j+l}_t \partial_\nu u^f_\vep \right\|_{L^2\left((0,T/\vep^2);L^2(\Gamma)\right)} + \vep^2 \left\|\partial^{j+l}_t u^f_\vep \right\|_{L^2\left((0,T/\vep^2); L^2(\Omega)\right)}\right] \\
&\qquad \qquad   + C\vep^{j +\frac {11}2}\sum^3_{l=1} \left\|\partial^{j+l}_t f(\cdot,t)\right\|_{L^2(\Omega_\vep)} +  C\vep^{j + \frac {13}2} \left\|\partial^{j+4}_t f\right\|_{H_0^1\left((0,T/\vep^2); L^2(\Omega_\vep)\right)}, \quad t\in (0,T\vep^{-2}).
\end{align*}
Utilizing statement \eqref{f1} of Lemma \ref{le:1}, we have 
\begin{align}
&\int^{T/\vep^2}_0\left|\textrm{Res}^{(l)}_{2}(\tau)\right| d\tau \le C\vep^{j + \frac 1 2} \|f\|_{H_0^{j + 8}\left(\R_+; L_{\alpha}^2(\R^3)\right)}, \quad l \in \{1,2\}. \label{eq:28}
\end{align}
We proceed to estimate $L^2(\Gamma)$-norm of $\mathcal Q\partial^{j + 2}_t\partial_{\nu}u_\vep^f$ and $\mathcal Q\partial^{j + 3}_t\partial_{\nu}u_\vep^f$. Setting $(j,q_0, q_1, \mathcal A) = (j+2,1, 1,\mathcal Q)$ and $(j,q_0, q_1, \mathcal A) = (j+3, 0, 1, \mathcal Q)$ in \eqref{eq:78}, we use \eqref{eq:118}, \eqref{eq:122}, \eqref{eq:77} and statement \eqref{h1} of Lemma \ref{le:a3} to obtain that for $t\in (0,T\vep^{-2})$,
\begin{align*}
&\left\|\left(\mathcal Q\partial^{j + \kappa}_t\partial_{\nu}u_\vep^f\right)(\cdot,t) - \frac{\rho_0\vep^2}{\rho_1} \left(\mathcal Q\partial^{j + \kappa}_t\partial_{\nu}v_\vep^f\right)(\cdot,t)\right\|_{L^2(\Gamma)} \\
&\le C\vep^{j + \kappa+ \frac {5}2}\sum^{\kappa+2}_{l=\kappa} \left\|\partial^{j + l}_t f(\cdot,t)\right\|_{L^2(\Omega_\vep)} +  C\vep^{j + \kappa + \frac {7}2} \left\|\partial^{j+\kappa+3}_t f\right\|_{H^1_0\left((0,T/\vep^2); L^2(\Omega_\vep)\right)}
\\
&+ C \sum^{5}_{l=4}\left\|\partial^{j+l}_t \partial_\nu u^f_\vep \right\|_{L^2\left((0,T/\vep^2); L^2(\Gamma)\right)} + C\vep^2  \sum^{5}_{l=4}\left\|\partial^{j+l}_t u^f_\vep \right\|_{L^2\left((0,T/\vep^2); L^2(\Omega)\right)}, \quad \kappa \in \{2,3\}.
\end{align*}
Thus, it follows from inequality \eqref{eq:93} and statement \eqref{f1} of Lemma \ref{le:1} that  
\begin{align} \label{eq:34}
&\int^{T/\vep^2}_0 \left|\textrm{Res}^{(l)}_{1}(\tau)\right| d\tau \le C\vep^{j + \frac 1 2} \|f\|_{H_0^{j + 8}\left(\R_+; L_{\alpha}^2(\R^3)\right)}, \quad l \in \{1,2\}.
\end{align}
Furthermore, we note that, drawing upon \eqref{eq:92} and Cauchy-Schwartz inequality, we obtain
\begin{align}
\int^{T/\vep^{2}}_0 \big\|\partial^{j}_t v^f_\vep(\cdot, \tau)\big\|_{L^2(\Omega)}d\tau  \le C\vep^{j-\frac 32}\left\|\partial^{j}_t f\right\|_{H_{0,\vep}^{1}\left(\R_+; L_{\alpha}^2(\R^3)\right)}. \label{eq:22}
\end{align}
Building upon statement \eqref{d1} of Lemma \ref{le:7}, we can use \eqref{eq:25}, \eqref{eq:26}, \eqref{eq:28}, \eqref{eq:34}, \eqref{eq:22} and \eqref{eq:109} to obtain \eqref{eq:36}.

Second, we prove \eqref{eq:37}. With the aid of \eqref{eq:a10}, it follows from \eqref{eq:31} that 
\begin{align*}
\left\|\left(\mathcal Q\partial^{j}_t\partial_{\nu}u_\vep^f\right)(\cdot,t)\right\|_{L^2(\Gamma)} \le C \vep^{j - \frac 1 2} \|f\|_{H_0^{j + 5}\left(\R_+; L^2_{\alpha}(\R^3)\right)}, \quad t\in(0,T\vep^{-2}].
\end{align*}
Then, setting $(j,q_0, q_1, \mathcal A) = (j,1, 1,\mathcal Q)$ in \eqref{eq:78}, we can use \eqref{eq:118}, \eqref{eq:122}, \eqref{eq:77}, \eqref{eq:93}, \eqref{eq:a4} and statement \eqref{f1} of Lemma \ref{le:1} to obtain 
\begin{align}
\left\|\left(\mathcal Q\partial^{j}_t\partial_{\nu}u_\vep^f\right)(\cdot,t)\right\|_{L^2(\Gamma)} \le C \vep^{j + \frac 3 2} \|f\|_{H_0^{j + 8}\left(\R_+; L^2_{\alpha}(\R^3)\right)}, \quad t\in(0,T\vep^{-2}].\label{eq:86}
\end{align}
This directly implies \eqref{eq:37}.

\textbf{Part II}: We observe that \eqref{eq:1} can be rewritten as
\begin{align}
\partial^j_t u_\vep^f(x,t) &-  \partial^j_t v_\vep^f(x,t)  = -\left(\frac{1}{c^2_1} - \frac{1}{c^2_0}\right) \int_{\Omega}\frac{\partial^{j+2}_{t} u_{\vep}^f(y,t-c_0^{-1}|x-y|)}{4\pi|x-y|}dy \notag\\
& -\left(\frac{\rho_0}{\rho_1\vep^2}-1\right) \int_{\Gamma} \frac{1}{4\pi|x-y|}\bigg[\partial^{j}_{t} \partial_{\nu}u_\vep^f(y,t) + \int^{t-c^{-1}_0|x-y|}_{t}\partial^{j+1}_{t} \partial_\nu u_\vep^f(y,\tau) d\tau \bigg] d\sigma(y) \notag\\
& - \left(\rho_0 - \rho_1\vep^2\right)\vep^2 \int_{\Omega} \frac{\partial^{j}_{t} f_\vep(y, t-c^{-1}_0|x-y|)}{4\pi|x-y|}dy, \quad x\in \R^3\backslash\Gamma, \quad t\in \R_+.\label{eq:39}
\end{align}
Combining this with \eqref{eq:122}, \eqref{eq:61}, \eqref{eq:75}, statement \eqref{f1} of Lemma \ref{le:1} and statement \eqref{h1} of Lemma \ref{le:a3} gives \eqref{eq:79}.

\end{proof}

Building upon statements \eqref{f1} and \eqref{f2} of Lemma \ref{le:1}, we are now in a position to give the proof of statement \eqref{f3} of Lemma \ref{le:1}.

\begin{proof} [Proof of statement \eqref{f3} of Lemma \ref{le:1}]
Since \eqref{eq:95} can be easily derived by employing \eqref{eq:122}, \eqref{eq:61}, \eqref{eq:91}, \eqref{eq:39}, statement \eqref{f1} of Lemma \ref{le:1} and statement \eqref{h1} of Lemma \ref{le:a3}, we solely focus on the proof of \eqref{eq:91}. Similarly to the derivation of \eqref{eq:75}, in order to derive \eqref{eq:91}, we require the estimate of $|\partial^j_t \Lambda_\vep (t)|$ and $\left\|\left(\mathcal Q\partial^{j}_t\partial_{\nu}u_\vep^f\right)(\cdot,t)\right\|_{L^2(\Gamma)}$. With the aid of \eqref{eq:86}, it suffices to prove
\begin{align}
&\left|\partial^j_t \Lambda_\vep (t)\right| \le C \vep^{j + \frac 3 2} \|f\|_{H_0^{j + 12}\left(\R_+; L_{\alpha}^2(\R^3)\right)}. \label{eq:110}
\end{align}
We note that by statement \eqref{f2} of Lemma \ref{le:1}, we have that
\begin{align} 
\left\|\partial^{j+4}_{t}\partial_\nu u^f_\vep(\cdot,t)\right\|_{L^2(\Gamma)} \le C \vep^{j + \frac{9}2} \|f\|_{H_0^{j + 12}\left(\R_+; L_{\alpha}^2(\R^3)\right)}\label{eq:111}
\end{align}
and
\begin{align}
\left\|\partial^{j+4}_{t} u^f_\vep(\cdot,t)\right\|_{L^2(\Omega)} \le C \vep^{j + \frac{5}2} \|f\|_{H_0^{j + 12}\left(\R_+; L_{\alpha}^2(\R^3)\right)}. \label{eq:127}  
\end{align}
Setting $(j,q_0, q_1, \mathcal A) = (j,4,4,\mathcal P)$, $(j,q_0, q_1, \mathcal A) = (j+1,3,3,\mathcal P)$ and  $(j,q_0, q_1, \mathcal A) = (j+2,2,2,\mathcal P)$ in \eqref{eq:78}, and using \eqref{eq:118},  \eqref{eq:122}, \eqref{eq:111}, \eqref{eq:127}, \eqref{eq:a4}, statement \eqref{f1} of Lemma \ref{le:1} and Lemma \ref{le:8},
we have that
\begin{align} \label{eq:130}
&\sum^{3-j_1}_{l=1} \eta_l\partial^{l+j+j_1+1}_{t}\Lambda_\vep(t) + \gamma_{\vep} \partial^{j+j_1}_t\Lambda_\vep(t) = \int_{\Omega} \frac{\rho_1\vep^2 \partial^{j+j_1+2}_{t} v_\vep^f(y,t)}{(\rho_0-\rho_1\vep^2)|\Omega|} dy + \sum_{l=1}^{2}\textrm{Res}^{(j_1)}_{l}(t)
\end{align}
for $t\in \R_+$, where $\textrm{Res}^{(j_1)}_{1}$ and $\textrm{Res}^{(j_1)}_2$ satisfies 
\begin{align} 
&\textrm{Res}^{(j_1)}_{1}(t) =\frac{c^2_0}{|\Omega|} \sum^{4}_{l= j_1+1} \frac{(-1)^{l -1}(l-1)}{c^{l}_0 l!} \left\langle 1, K^*_l \mathcal Q\partial^{l+j}_t \partial_\nu u^f_\vep(\cdot,t)\right\rangle_{L^2(\Gamma)}, \quad t\in \R_+, \notag \\
&\left|\textrm{Res}^{(j_1)}_{2}(t)\right| \le  C \vep^{ j + \frac {9} 2}\|f\|_{H_0^{j+12}\left(\R_+; L_{\alpha}^2(\R^3)\right)}, \quad t \in \left(0, {T}{\vep^{-2}}\right],  \quad j_1\in \{0,1,2\}. \label{eq:131}
\end{align} 
Here, $\gamma_\vep$ is given by \eqref{eq:116}, and $\eta_1,\eta_2$,  and $\eta_3$ are given by \eqref{eq:114} and \eqref{eq:115}, respectively.
Furthermore, setting $(j,q_0, q_1, \mathcal A) = (j+2, 2, 2,\mathcal Q)$ and $(j,q_0, q_1, \mathcal A) = (j+3, 1, 1, \mathcal Q)$ in \eqref{eq:78}, and employing \eqref{eq:118}, \eqref{eq:122}, \eqref{eq:77}, \eqref{eq:93}, \eqref{eq:111}, \eqref{eq:127}, \eqref{eq:a4}, and statement \eqref{f1} of Lemma \ref{le:1}, we arrive at 
\begin{align} 
    \left\|\left(\mathcal Q\partial^{j + l}_t\partial_{\nu}u_\vep^f\right)(\cdot,t)\right\|_{L^2(\Gamma)} \le C\vep^{j + \frac 92} \|f\|_{H_0^{j + 12}\left(\R_+; L_{\alpha}^2(\R^3)\right)}, \quad t \in \left(0, {T}{\vep^{-2}}\right], \quad l \in \{2,3\}. \notag
\end{align}
Combining this with \eqref{eq:111} gives 
\begin{align} \label{eq:126}
\left|\textrm{Res}^{(j_1)}_{1}(t)\right| \le  C \vep^{ j + \frac {9} 2}\|f\|_{H_0^{j+12}\left(\R_+; L_{\alpha}^2(\R^3)\right)}, \quad t \in \left(0, {T}{\vep^{-2}}\right],  \quad j_1\in \{0,1,2\}. 
\end{align}

Drawing upon statement \eqref{d2} of Lemma \ref{le:7}, it follows from \eqref{eq:22}, \eqref{eq:130}, \eqref{eq:131}, \eqref{eq:126} and \eqref{eq:109} that \eqref{eq:110} holds. 
\end{proof}

{
\subsection{Proof of Lemmas \ref{le:3} and \ref{le:9}} \label{sec:3.3}

\begin{proof}[Proof of Lemma \ref{le:3}]
With the aid of \eqref{eq:24}, the statement of this lemma can be proved in the same manner as statements (c) and (d) of Lemma 2.3 in \cite{LS-04}.
\end{proof}

\begin{proof}[Proof of Lemma \ref{le:9}]
We only focus on the derivation of \eqref{eq:41} since the proof of \eqref{eq:44} can be derived in a same manner. {\color{HW}First, by reverse triangle inequality, we have }
\begin{align} \label{eq:a2}
||\vep x_1 + x_2| -|x_2|| \le \vep |x_1|, \quad \textrm{for any}\; x_1, x_2\in \R^3\; \textrm{and}\; \vep \in \R_+.
\end{align}
Let $\widetilde g_\vep(y,t):=g(y, t/\vep)$ for $(y,t) \in \R^3 \times \R_+$.
With the aid of \eqref{eq:a2}, using the identity 
\begin{align}\label{eq:128} 
&|\Phi_{\vep^{-1}}(x)-y| = |y_0 + \vep^{-1}(x-y_0)-y| = \vep^{-1}|\vep(y_0-y) + x-y_0|
\end{align}
and employing the Taylor expansions, we arrive at 
\begin{align}
&\int_\Gamma \frac{g(y, \vep^{-1}t-c^{-1}_0\left|\Phi_{\vep^{-1}}(x)-y\right|)}{4\pi\left|\Phi_{\vep^{-1}}(x)-y\right|}d\sigma(y)\notag \\ 
& = \int_{\Gamma} \frac{\widetilde g_\vep(y, t - c^{-1}_0|\vep(y_0-y) + x-y_0|)}{4\pi\left|\Phi_{\vep^{-1}}(x) -y\right|}d\sigma(y)\notag\\
& = \int_{\Gamma}\frac{\widetilde g_\vep(y, t - c_0^{-1} \left|x-y_0\right|)}{4\pi|\Phi_{\vep^{-1}}(x)-y|}d\sigma(y) - c^{-1}_0 \int_{\Gamma} r(y)\frac{\partial_t \widetilde g_\vep (y, t_{x,y,y_0})}{4\pi\left|\Phi_{\vep^{-1}}(x)-y\right|}d\sigma(y)\label{eq:143},
\end{align}
where $|t_{x,y,y_0}-(t - c^{-1}_0|x-y_0|)| = |r(y)|$ and $r(y)$ satisfies
\begin{align} \label{eq:55}
|r(y)| \le C \vep.
\end{align}
It is easy to verify that 
\begin{align}
&\int_{\R^3} \frac 1 {({1+|x|^2)}^{\frac \alpha2}} \left|\int_{\Gamma} \frac{1}{4\pi\left|\Phi_{\vep^{-1}}(x) - y\right|}\widetilde g_\vep (y,  t - c_0^{-1}\left|x-y_0\right|)d\sigma(y)\right|^2dx \notag \\
&\le \int_{\R^3}  \frac 1 {({1+|x|^2)}^{\frac \alpha2}} \left(\int_{\Gamma} \frac{1}{4\pi\left|\Phi_{\vep^{-1}}(x) - y\right|}\sup_{\tau \in (0, t)}|\widetilde g_\vep{(\cdot, \tau)}|d\sigma(y)\right)^2dx. \notag
\end{align}
Therefore, using statement \eqref{h0} of Lemma \ref{le:3}, we obtain 
\begin{align}
&\left\|\int_{\Gamma}\frac{\widetilde g_\vep\left(y, t - c_0^{-1}\left|\cdot-y_0\right|\right)}{4\pi|\Phi_{\vep^{-1}}(\cdot) -y|}d\sigma(y) \right\|_{L_{-\alpha}^2(\R^3)} \notag \\ 
&\le  \left[\vep\left\|\frac{1}{4\pi|\cdot-y_0|}\right\|_{L_{-\alpha}^2(\R^3)} + C\vep^{\frac 32} \|R_{0}\|_{L^2_{\alpha}(\R^3),H^2_{-\alpha}(\R^3)}\right]\sup_{\tau \in (0, t)}\|\widetilde g_\vep{(\cdot, \tau)}\|_{L^{2}(\Gamma)}. \label{eq:27}
\end{align}
Furthermore, utilizing \eqref{eq:55} and statement \eqref{h0} of Lemma \ref{le:3}, we have 
\begin{align}
&\int_{\R^3} \frac 1 {({1+|x|^2)}^{\frac \alpha2}} \left|\int_{\Gamma} \frac{1}{4\pi\left|\Phi_{\vep^{-1}}(x) - y\right|}r(y)\partial_t \widetilde g_\vep (y, t_{x,y,y_0}) d\sigma(y)\right|^2dx \notag\\
&\le C \vep^2 \int_{\R^3} \frac 1 {({1+|x|^2)}^{\frac \alpha2}} \left(\int_{\Gamma} \frac{1}{4\pi\left|\Phi_{\vep^{-1}}(x) -y
\right|} |\partial_t \widetilde g_\vep(y, t_{x,y,y_0})| d\sigma(y)\right)^2 dx \notag\\
&\le C\int_{\R^3} \frac 1 {({1+|x|^2)}^{\frac \alpha2}} \left( \int_{\Gamma} \frac{1}{4\pi\left|\Phi_{\vep^{-1}}(x) - y \right|}\sup_{\tau\in (0,t/\vep)}|\partial_t g(y, \tau)|d\sigma(y) \right)^2  dx. \notag\\
&\le C\vep^2 \int_{\Gamma}\sup_{\tau\in (0,t/\vep)}|\partial_t g(y, \tau)|^2dy \le C\vep^2\int_\Gamma \int^{t/\vep}_0 \left|\partial_t g(y, \tau) \partial_{tt} g (y, \tau) \right| d\tau d\sigma(y).\notag
\end{align}
From this, in conjunction with \eqref{eq:143}, \eqref{eq:27} and Cauchy-Schwartz inequality gives \eqref{eq:41}. The proof of this lemma is thus completed.
\end{proof}
}

\begin{appendices}
\renewcommand{\theequation}{\Alph{section}.\arabic{equation}}

\section{} \label{sec:A}
\subsection{Solutions of the system of ordinary differential equations}

\begin{lemma} \label{le:5}
Let $a_1, a_2\in \R_+$. Assume that $h \in H_{0,\sigma}^2 \left(\R_+\right)$ and $f \in H_{0,\sigma}^0\left(\R_+\right)$ for any $\sigma\in \R_+$, and suppose that $h$ satisfies
\begin{align}
&\partial_{tt} h  + a_1\partial_t h + a_2 h = f \quad \textrm{in}\; \R_+, \label{eq:66}\\
&h(0) = 0, \quad \partial_t h(0) = 0. \label{eq:67}
\end{align}
If $a_1^2 - 4a_2 < 0$, we have that for $t \in \R_+$,
\begin{align} \label{eq:62}
&h(t) = \frac{1}{i\sqrt{4a_2-a_1^2}}\int^{t}_0 \left[e^{{\left(-a_1 + i\sqrt{4a_2-a_1^2}\right)}(t-\tau)/2} - e^{{\left(-a_1 -i\sqrt{4a_2-a_1^2}\right)}(t-\tau)/2}\right]f(\tau) d\tau.
\end{align}

\end{lemma}

\begin{proof}
Since $h \in H_{0,\sigma}^2\left(\R_+\right)$ and $f \in H_{0,\sigma}^0\left(\R_+\right)$ for any $\sigma\in \R_+$, it can be seen that their Fourier-Laplace transforms $\hat h$ and $\hat f$ exist in the domain $\{z\in \CC: \textrm{Re}(z) > 0\}$. Recall that for any $\phi \in H_{0,\sigma}^2 \left(\R_+\right)$, its Fourier-Laplace transform is $\hat \phi(s):= \int^{+\infty}_0 e^{-st}\phi(t) dt$. 
With the aid of \eqref{eq:66} and \eqref{eq:67}, we see that $\hat h$ satisfies
\begin{align*}
s^2 \hat h +  a_1 s \hat h + a_2 \hat h = \hat f.
\end{align*}
Further, let $r_1 := \left({-a_1 + i\sqrt{4a_2-a_1^2}}\right)/{2}$ and $r_2 := \left({-a_1 - i\sqrt{4a_2-a_1^2}}\right)/{2}$,
which are two complex roots of the equation $r^2+a_1r+a_2=0$. Therefore, we have 
\begin{align} \label{eq:50}
\hat h = \frac{\hat f}{s^2+ a_1 s + a_2} = \frac{\hat f}{(s-r_1)(s-r_2)} = \frac{1}{r_1-r_2}\left(\frac{\hat f}{s-r_1} - \frac{\hat f}{s-r_2}\right).
\end{align}
By the well-known identity (see \cite{D-74})
\begin{align*}
\int^{\infty}_0 e^{-st}e^{z t} = \frac{1}{s-z} \quad \textrm{for}\; \textrm{Re}(s) > \textrm{Re}(z)\; \textrm{with any}\; z\in \CC,
\end{align*}
and convolution theorem of the Fourier-Laplace transform, we readily deduce from \eqref{eq:50} that \eqref{eq:62} holds.

\end{proof}

Based on Lemma \ref{le:5}, we have the following lemma.

\begin{lemma} \label{le:7}
Let $\gamma_\vep$ be given by \eqref{eq:116}, and $\eta_l$ for $l\in\{1,2,3,4\}$ be specified in \eqref{eq:114}--\eqref{eq:115}. Assume that $g_0, g_1, g_2, g_3 \in H_{0,\sigma}^0 \left(\R_+\right) $ for any $\sigma\in \R_+$. The following arguments hold true.

\begin{enumerate}[(a)]
    \item \label{d1}
Suppose that $h \in H_{0,\sigma}^3 \left(\R_+\right)$ for any $\sigma\in \R_+$ and that satisfies 
\begin{align} \label{eq:85}
\sum^{2-j}_{l= 1} \eta_l \partial^{l+1+j}_t h +  \gamma_{\vep}\partial^j_t h = g_j \quad \textrm{in}\; \R_+, \quad j\in \{0,1\},   
\end{align}
we have 
\begin{align}\label{eq:108}
h(t) = \frac{1}{\lambda^+_{M,1} - \lambda^-_{M,1}} \int^{t}_0 \left[e^{\lambda_{M,1}^+(t-\tau)} - e^{\lambda_{M,1}^-(t-\tau)}\right] \left[g_0(\tau)- \eta_2 g_1(\tau) \right] d\tau, \quad t\in \R_+,
\end{align}
where $\lambda^{\pm}_{M,1}$ are two roots of the equation $\lambda^2 -\gamma_{\vep}\eta_2 \lambda + \gamma_{\vep} \lambda = 0.$ 

\item \label{d2}
Let $\vep>0$ be sufficiently small such that $1-(\eta_3 - \eta^2_2)\gamma_\vep > 0$.
Suppose that $h \in H_{0,\sigma}^4 \left(\R_+\right)$ for any $\sigma\in \R_+$ and that satisfies         
\begin{align} 
\sum^{3-j}_{l= 1} \eta_l \partial^{l+1+j}_t h +  \gamma_{\vep}\partial^j_t h = g_j \quad \textrm{in}\; \R_+, \quad j\in \{0,1,2\}, \label{eq:64}
\end{align}
then we have           
\begin{align} \label{eq:68}
h(t) = \frac{1}{\lambda^+_{M,2} - \lambda^-_{M,2}} \int^{t}_0 \left[e^{\lambda_{M,2}^+(t-\tau)} - e^{\lambda_{M,2}^-(t-\tau)}\right] \frac{e_1(\tau)}{1-\eta_3\gamma_{\vep} + \eta_2^2 \gamma_{\vep}} d\tau, \quad t\in \R_+,
\end{align}
where $e_1(\tau):=g_0(\tau)-\eta_2(g_1(\tau) - \eta_2g_2(\tau))-\eta_3 g_2(\tau),\; \tau \in \R_+$, and $\lambda^{\pm}_{M,2}$ are two roots of the equation 
$\lambda^2 + {-\eta_2 \gamma_{\vep}\lambda + \gamma_{\vep} \lambda}/\left({1-\eta_3\gamma_{\vep} + \eta_2^2 \gamma_{\vep}}\right)= 0$. 

\item \label{d3}
Let $\vep>0$ be sufficiently small such that $1-(\eta_3 - \eta^2_2)\gamma_\vep > 0$ and
\begin{align}
    1-\eta_3\gamma_{\vep}  + \gamma_{\vep}\frac{\eta^2_2-\eta_4 \gamma_{\vep}\eta_2 + \eta_3{\eta^2_2}\gamma_{\vep}}{1-\eta_3\gamma_{\vep} + \eta_2^2\gamma_{\vep}} >0. \notag
\end{align}
Suppose that $h \in H_{0,\sigma}^5 \left(\R_+\right)$ for any $\sigma\in \R_+$ and that satisfies 
\begin{align}
\sum^{4-j}_{l= 1} \eta_l \partial^{l+1+j}_t h +  \gamma_{\vep}\partial^j_t h = g_j \quad \textrm{in}\; \R_+, \quad j\in \{0,1,2,3\}, \label{eq:60}
\end{align}
then we have
\begin{align} \label{eq:71}
h(t) = \frac{1}{\lambda^+_{M,3} - \lambda^-_{M,3}} \int^{t}_0   \frac{e^{\lambda_{M,3}^+(t-\tau)}e_2(\tau) - e^{\lambda_{M,3}^-(t-\tau)}e_2(\tau)}{1-\eta_3\gamma_{\vep}  + \gamma_{\vep}\frac{\eta^2_2-\eta_4 \gamma_{\vep}\eta_2 + \eta_3{\eta^2_2}\gamma_{\vep}}{1-\eta_3\gamma_{\vep} + \eta_2^2\gamma_{\vep}}} d\tau, \quad t\in \R_+, 
\end{align}
where 
\begin{align*}
e_2(t):=\left[g_0-\eta_3 g_2 - (\eta_4-\eta_3\eta_2)g_3 - \frac{(\eta_2-\gamma_{\vep}(\eta_4-\eta_3\eta_2))( g_1-\eta_2(g_2 - \eta_2g_3)-\eta_3 g_3)}{1-\eta_3\gamma_{\vep} + \eta_2^2\gamma_{\vep}}\right](t)
\end{align*}
for $t\in \R_+$,
and $\lambda^{\pm}_{M,3}$ are two roots of the equation 
\begin{align*}
\lambda^2 + \frac{\gamma_{\vep} \frac{-\eta_2 + \eta_4 \gamma_{\vep} - {\eta_3}\gamma_{\vep}\eta_2}{1-\eta_3\gamma_{\vep} + \eta_2^2\gamma_{\vep}} \lambda + \gamma_{\vep}}{1-{\eta_3}\gamma_{\vep} + \gamma_{\vep}\frac{\eta^2_2-\eta_4 \gamma_{\vep}\eta_2 + \eta_3{\eta^2_2}\gamma_{\vep}}{1-\eta_3\gamma_{\vep} + \eta_2^2\gamma_{\vep}}} = 0.
\end{align*}
\end{enumerate}
\end{lemma}

\begin{proof}
\eqref{d1} Subtracting \eqref{eq:85} for $j=1$ mulitpied by $\eta_2$ from \eqref{eq:85} for $j=0$, we have 
\begin{align} \label{eq:63}
\partial_{tt} h -\eta_2\gamma_{\vep}\partial_{t} h + \gamma_{\vep} h = g_0 - \eta_2 g_1, \quad \textrm{in}\;\R_+.
\end{align}
Thus, \eqref{eq:108} can be derived by using Lemma \ref{le:5}.

\eqref{d2} Proceeding as in the derivation of \eqref{eq:63}, we can use \eqref{eq:64} with $j\in\{1,2\}$ to obtain
\begin{align*}
\partial_{ttt} h -\eta_2\gamma_{\vep}\partial_{tt} h + \gamma_{\vep} \partial_th = g_1 - \eta_2 g_2, \quad \textrm{in}\;\R_+.
\end{align*}
This, together with \eqref{eq:64} for $j \in \{0,2\}$ yields
\begin{align}
(1-\eta_3\gamma_{\vep} + \eta_2^2 \gamma_{\vep}) \partial_{tt} h- \eta_2 \gamma_{\vep} \partial_t h + \gamma_{\vep} h = g_0-\eta_2(g_1 - \eta_2g_2)-\eta_3 g_2, \quad \textrm{in}\;\R_+. \label{eq:70}
\end{align}
Therefore, one can deduce from Lemma \ref{le:5} that \eqref{eq:68} holds.

\eqref{d3} Similarly to the derivation of \eqref{eq:70}, we can utilize \eqref{eq:60} for $j \in \{1,2,3\}$ to obtain
\begin{align*}
(1-\eta_3\gamma_{\vep} + \eta_2^2 \gamma_{\vep}) \partial_{ttt} h- \eta_2 \gamma_{\vep} \partial_{tt} h + \gamma_{\vep} \partial_t h = g_1-\eta_2(g_2 - \eta_2g_3)-\eta_3 g_3, \quad \textrm{in}\;\R_+.
\end{align*}
Combining this with \eqref{eq:60} for $j \in \{0,2,3\}$ gives
\begin{align*}
\left(1-\eta_3\gamma_{\vep}  + \gamma_{\vep}\frac{\eta^2_2-\eta_4 \gamma_{\vep}\eta_2 + \eta_3{\eta^2_2}\gamma_{\vep}}{1-\eta_3\gamma_{\vep} + \eta_2^2\gamma_{\vep}}\right) \partial_{tt} h + \gamma_{\vep}\frac{-\eta_2 + \eta_4 \gamma_{\vep} - {\eta_3}\gamma_{\vep}\eta_2}{1-\eta_3\gamma_{\vep} + \eta_2^2\gamma_{\vep}} \partial_{t} h + \gamma_{\vep} h =  e_2,  \quad \textrm{in}\;\R_+.
\end{align*}
From this, using Lemma \ref{le:5} again, we conclude that \eqref{eq:71} holds. 
\end{proof}

\begin{remark}
We observe that $\lambda^{\pm}_{M,j}$ for $j\in\{1,2,3\}$ in Lemma \ref{le:7} satisfy
\begin{align} \label{eq:109}
\left|\textrm{Re}\left(\lambda_{M,j}^\pm\right) + \frac{C_\Omega}{8\pi c_0} \omega^2_M\vep^2\right| \le C\vep^4, \quad \left|\textrm{Ima}\left(\lambda_{M,j}^\pm\right) \mp \omega_M\vep\right| \le C\vep^3.
\end{align}
Here, $C$ is a positive constant independent of $\vep$.
\end{remark}

\subsection{Further auxiliary estimates}

\begin{lemma} \label{le:a3}
Let $V_1,V_2\subset\mathbb{R}^3$ be two sets, each either a bounded $C^2$ domain or a $C^2$ hypersurface. 
For $\ell=1,2$, let $\mu_\ell$ denote the canonical measure on $V_\ell$ (Lebesgue measure if $V_\ell$ is a domain and surface measure if $V_\ell$ is a hypersurface). 
Let $s:V_1\times V_2 \rightarrow \mathbb{R}$ be a kernel that may be singular along the diagonal $\{(x,y):x=y\}$. Then the following estimates hold.

\begin{enumerate}[(a)]
\item \label{h1} 
Given $g\in H_0^{p}\left(\R_+; L^2(V_1)\right)$ with $p\in \{l\in\mathbb N: l>1\}$, for any fixed $\tau' \in \R_+$, we have that 
\begin{align}
&\int_{V_2}\left|\int_{V_1} s(x,y)\int_{\tau'-c^{-1}_0|x-y|}^{\tau'}\partial^{j_1+1}_t g(y,\tau)\frac{(\tau'-\tau)^{j_1}}{{j_1}!} d\tau d\sigma_{V_1}(y) \right|^2 d\sigma_{V_2}(x) \notag\\
&\le C \|\mathcal K_{s,j_1+1}\|^2_{\mathcal L (L^2(V_1), L^2(V_2))}\left\|\partial^{j_1+1}_t g\right\|_{L^2\left((0,\tau'); L^2(V_1)\right)}\left\|\partial^{j_1 + 2}_t g\right\|_{L^2\left((0, \tau'); L^2(V_1)\right)}\label{eq:a4} 
\end{align}
and 
\begin{align}
&\int_{V_2}\left|\int_{V_1} s(x,y)\partial^{j_2}_t g(y, \tau'-c^{-1}_0|x-y|)d\sigma_{V_1}(y)  \right|^2 d\sigma_{V_2}(x) \notag\\
&\le C \|\mathcal K_{s,0}\|^2_{\mathcal L (L^2(V_1), L^2(V_2))}\left\|\partial^{j_2}_t g\right\|_{L^2\left((0,\tau'); L^2(V_1)\right)}\left\|\partial^{j_2 + 1}_t g\right\|_{L^2\left((0,\tau'); L^2(V_1)\right)},\label{eq:a1}
\end{align}
where $j_1\in \{l\in \mathbb N_0: l+2 \le p\}$, $j_2\in \{l\in \mathbb N_0: l + 1 \le p\}$ and the integral operator $\mathcal K_{s,l}$ is defined by
\begin{align}
\left(\mathcal K_{s,\sigma} \phi\right) (x):= \int_{V_1} \left|s(x,y)\right||x-y|^{\sigma} \phi(y) d\sigma_{V_1}(y),  \quad x \in V_2, \;\sigma \in \R_+. \label{eq:141}
\end{align}
Here, $C$ is a positive constant independent of $s$, $\tau'$ and $g$.

\item \label{h2}
Let $T \in \R_+$ be fixed. Given $g\in H_0^{p}\left(\R_+; L^2(V_1)\right)$ with $p\in \mathbb N$, we have
\begin{align}
\int_0^{T}\int_{V_2}\left|\int_{V_1} s(x,y)\int_{\tau'-c^{-1}_0|x-y|}^{\tau'}\partial^{j_3+1}_t g(y,\tau)\frac{(\tau'-\tau)^{j_3}}{{j_3}!} d\tau  d\sigma_{V_1}(y)\right|^2 d\sigma_{V_2}(x)d\tau' \notag \\
\le 
C \|\mathcal K_{s,j_3+\frac 12}\|^2_{\mathcal L (L^2(V_1), L^2(V_2))}\left\|\partial^{j_3+1}_t g\right\|^2_{L^2\left((0,T); L^2(V_1)\right)} \label{eq:139}
\end{align}
and
\begin{align}
\int^T_0\int_{V_2}\left|\int_{V_1} s(x,y) \partial^{j_4}_t g(y,\tau'-c^{-1}_0|x-y|) d\sigma_{V_1}(y)\right|^2 d\sigma_{V_2}(x) d\tau' \notag\\
\le C \|\mathcal K_{s,0}\|^2_{\mathcal L (L^2(V_1), L^2(V_2))}\left\|\partial^{j_4}_t g\right\|^2_{L^2\left((0,T); L^2(V_1)\right)}, \label{eq:a7}
\end{align}
where $j_3\in \{l \in \mathbb N_0: l < p\}$, $j_4\in \{l \in \mathbb N_0: l \le p\}$, and the operators $\mathcal K_{s,j_3+ 1/2}$ and $\mathcal K_{s,j_4}$ are specified in \eqref{eq:141}. Here, $C$ is a positive constant independent of $T$, $s$ and $g$.
\end{enumerate}

\end{lemma}

\begin{proof}
\eqref{h1}
First, we derive \eqref{eq:a4}.
By Cauchy-Schwartz inequality, we have 
\begin{align}
&\int_{V_2}\left|\int_{V_1} s(x,y)\int_{\tau'-c^{-1}_0|x-y|}^{\tau'}\partial^{j_1+1}_t g(y,\tau)\frac{(\tau'-\tau)^{j_1}}{j_1!} d\tau d\sigma_{V_1}(y)\right|^2  d\sigma_{V_2}(x) \notag \\
&\le \int_{V_2}\left(\int_{V_1} |s(x,y)|\left|\int_{\tau'-c^{-1}_0|x-y|}^{\tau'}\partial^{j_1+1}_t g(y,\tau)\frac{(\tau'-\tau)^{j_1}}{j_1!} d\tau \right| d\sigma_{V_1}(y)\right)^2 d\sigma_{V_2}(x) \notag \\
&\le C\int_{V_2}\left(\int_{V_1} \left|s(x,y)\right||x-y|^{j_1+1} \sup_{\tau \in (0,\tau')}\left|\partial^{j_1+1}_t g(y,\tau)\right|d\sigma_{V_1}(y)\right)^2 d\sigma_{V_2}(x) \notag \\
& \le C\|\mathcal K_{s,j_1+1}\|^2_{\mathcal L (L^2(V_1), L^2(V_2))} \int_{V_1} \sup_{\tau \in (0,\tau')}\left|\partial^{j_1+1}_t g(y,\tau)\right|^2d\sigma_{V_1}(y).
\label{eq:a5}
\end{align}
Since $g\in H_0^{p}\left(\R_+; L^2(V_1)\right)$, we find that for each $q \in \{l\in \mathbb N_0: l+1\le p\}$
\begin{align}
\left|\partial^{q}_t g(y_1, t)\right|^2  &= \left|\partial^{q}_t g(y_1,t)\right|^2 - \left|\partial^{q}_t g(y_1, 0)\right|^2  \notag\\
& = \int_0^{t} 2 \partial^{q}_t g(y_1,\tau) \partial^{q+1}_t g(y_1,\tau)d\tau, \quad t\in(0,\tau'),\quad \textrm{a.e.}\; y_1\in V_1. \label{eq:a10}
\end{align}
This, together with \eqref{eq:a5} and Cauchy-Schwartz inequality yields \eqref{eq:a4}. Similarly to the derivation of \eqref{eq:a4}, we can use \eqref{eq:a10} to get \eqref{eq:a1}.

\eqref{h2}
Using Cauchy-Schwartz inequality, we have 
\begin{align}
&\int_0^{T}\int_{V_2}\left|\int_{V_1} s(x,y)\int_{\tau'-c^{-1}_0|x-y|}^{\tau'}\partial^{j_3+1}_t g(y,\tau)\frac{(\tau'-\tau)^{j_3}}{{j_3}!} d\tau  d\sigma_{V_1}(y) \right|^2d\sigma_{V_2}(x)d\tau' \notag\\
&\le \int_0^{T}\int_{V_2} \left(\int_{V_1} \left|s(x,y) \right| \left|\int_{\tau'-c^{-1}_0|x-y|}^{\tau'}\partial^{j_3+1}_t g(y,\tau)\frac{(\tau'-\tau)^{j_3}}{{j_3}!} d\tau \right| d\sigma_{V_1}(y)\right)^2 d\sigma_{V_2}(x)d\tau' \notag \\
& \le C\int_0^{T} \int_{V_2}\left(\int_{V_1} \left|s(x,y)|x-y|^{j_3+\frac 12}\right| \left[\int_{\tau'-c^{-1}_0|x-y|}^{\tau'}\left|\partial^{j_3+1}_t g(y,\tau)\right|^2 d\tau\right]^{\frac 12} d\sigma_{V_1}(y)\right)^2 d\sigma_{V_2}(x)d\tau' \notag \\
& \le C \|\mathcal K_{s,j_3+\frac 12}\|^2_{\mathcal L (L^2(V_1), L^2(V_2))} \int_0^{T} \int_{V_1} \int_{\tau'- M_{V_1,V_2}}^{\tau'}\left|\partial^{j_3+1}_t g(y,\tau)\right|^2 d\tau d\sigma_{V_1}(y) d\tau'. \label{eq:140}
\end{align}
Here, $M_{V_1,V_2} = \max_{y_1\in V_1, y_2 \in V_2} |y_1 - y_2|$.
By the causal properties of $g$, we readily obtain
\begin{align*}
\int_0^{T}\int_{\tau'- M_{V_1,V_2}}^{\tau'}\left|\partial^{j_3+1}_t g(y_1,\tau)\right|^2 d\tau d\tau' &=
\int^T_0 \int_{-M_{V_1,V_2}}^{0}\left|\partial^{j_3+1}_t g(y_1,\tau + \tau')\right|^2d\tau d\tau'\\
&= \int_{-M_{V_1,V_2}}^{0} \int^T_0 \left|\partial^{j_3+1}_t g(y_1,\tau + \tau')\right|^2d\tau' d\tau\\
&\le M_{V_1,V_2} \int^T_0 \left|\partial^{j_3+1}_t g(y_1,\tau)\right|^2d\tau, \; \textrm{a.e.}\; y_1\in V_1.
\end{align*}
This, together with \eqref{eq:140} yields \eqref{eq:139}. Moreover, by using similar arguments as employed in the derivation of \eqref{eq:139}, we readily obtain that \eqref{eq:a7} holds.

The proof of this lemma is thus completed.
\end{proof}

\subsection{Proof of inequality (\ref{eq:32})}\label{sec:a3}

{\color{HW}\begin{proof}
    Since $\Gamma$ is $C^2$-smooth, to validate \eqref{eq:32}, it suffices -after the flatten transformation- to prove the following inequality
\begin{align}
  &\|\partial_\nu \phi\|_{\mathbb L^2(\partial{B_{1/2}^{-}})} \le \widetilde C\big[\|\nabla \phi\|_{\mathbb L^2(B^-_1)} + \|g\|_{L^2(B^-_1)} + \theta\|g\|_{L^2(B_1^+)} + \theta \|\nabla \phi\|_{\mathbb L^2(B^+_1)}\big] \label{eq:65}\\
&  \qquad \qquad \mathrm{for}\; \mathrm{\phi} \in \left\{\phi \in H^{1}(B_1): \nabla \cdot\widetilde \theta  G \nabla u = \widetilde \theta g\; \mathrm{in}\; B_1, \;  g\in L^2(B_1)\right\},\notag  
\end{align}
where $G$ is uniformly elliptic and $C^1$-smooth and $\widetilde \theta = \theta$ in $B^+_1$ and $\widetilde \theta =1 $ in $B^-_1$. Here, given $r\in \R_+$, $B_r:=\{x\in \R^3: |x| < r\}$ and $B^\pm_r:=\{x\in \R_\pm^3: |x| < r\}$.

Let $h \in \R$ be sufficiently small. We define the $j-$th ($j\in\{1,2\}$) difference quotients
\begin{align*}
&\left(D^j_h \phi\right)(x):= \left(\phi(x + h e_j) - \phi(x)\right)/h, \quad  x \in \widetilde B_h:=\{x\in B_1: x + re_j \in B_1\; \mathrm{with}\; |r|\in(0,h]\}.
\end{align*}
Here, $e_1 = (1,0,0)^T$ and $e_2 = (0,1,0)^T$.
It is clear that for $\psi \in H^{1}_0(\widetilde B_h)$,
\begin{align*}
\int_{\widetilde B_{h}}  \left(\widetilde\theta G \nabla D^j_h\phi\right)(x) \cdot \nabla \psi(x)dx & = -\int_{\widetilde B_{h}} \widetilde \theta(x)  g(x) \left(D_{-h}^{j}\psi\right)(x)dx \\
& + \int_{\widetilde B_h}\widetilde\theta(x) D^j_{h}G(x) \nabla \phi(x+he_j) \cdot \nabla \psi(x)dx.
\end{align*}
We choose $\chi \in C_c^{\infty}(\widetilde B_{h})$ such that $\chi = 1$ in $B_{1/2}$ and set $\psi = \chi^2 D^j_h\phi$ in the above identity, we have 
\begin{align*}
&\int_{\widetilde B_{h}} \widetilde\theta(x) \left(G \chi \nabla D^j_h\phi\right)(x) \cdot \left(\chi \nabla D^j_h\phi\right)(x)dx =- \int_{\widetilde B_{h}} \widetilde \theta(x) g(x) \left(D_{-h}^{j}\chi^2 D^j_h\phi \right)(x) dx + \\
&\int_{\widetilde B_{h}} \widetilde\theta(x) D^j_{h}G(x)\nabla \phi(x+h e_j) \cdot \left(\nabla \chi^2 D^j_h \phi\right)(x)dx - \int_{\widetilde B_{h}} \left(\widetilde\theta G \nabla D^j_h\phi\right)(x) \cdot \left(D^j_h\phi\nabla \chi^2 \right)(x)dx.
\end{align*}
By the Cauchy-Schwartz and Young inequalities, together with the uniform elliptic properties of the matrix $G$, we have 
\begin{align*}
\left\|\nabla D^j_h \phi \right\|_{\mathbb L^2(B^{-}_{1/2})} \le \widetilde C \left(\|g\|_{L^2(B^-_{1})} + \|\nabla \phi \|_{\mathbb L^2(B^{-}_{1})} + \theta\|g\|_{L^2(B^+_{1})} + \theta \|\nabla \phi\|_{\mathbb L^2 (B^+_{1})}\right).
\end{align*}
Therefore, with the aid of the properties of difference quotients (see e.g., \cite[Theorem 3, Page 277]{E10})) and the fact that $\nabla\cdot G  \nabla \phi = g$ in $B_1^{-}$, we obtain that for $ j,k\in\{1,2,3\}$,
\begin{align*}
\left\| \partial_{x_jx_k}\phi \right\|_{\mathbb L^2(B^{-}_{1/2})} \le \widetilde C \left(\|g\|_{L^2(B^-_{1})} + \|\nabla \phi \|_{\mathbb L^2(B^{-}_{1})} + \theta\|g\|_{L^2(B^+_{1})} + \theta \|\nabla \phi\|_{\mathbb L^2 (B^+_{1})}\right),
\end{align*}
whence \eqref{eq:65} follows from the fact $H^1(\Omega) \xrightarrow{\ \mathrm{trace}\ } H^{1/2}(\Gamma) \hookrightarrow L^2(\Gamma)$.
\end{proof}}

\end{appendices}

\section*{Acknowledgment}  

This work is supported by the Austrian Science Fund (FWF) grant P: 36942.


\begin{thebibliography}{99}

\bibitem{ACCS-20} H. Ammari, D. P. Challa, A. P. Choudhury and M. Sini, The equivalent media generated
by bubbles of high contrasts: Volumetric metamaterials and metasurfaces, {\em Multiscale Model.
Simul. \bf{18}} (2020), 240-–293.

\bibitem{AS-19}
H. Ammari, A. Dabrowski, B. Fitzpatrick, P. Millien and M. Sini,
Subwavelength resonant dielectric nanoparticles with high refractive indices, {\em Math. Methods Appl. Sci. \bf{42}} (2019), 6567–6579.

\bibitem{AZ-18} H. Ammari, B. Fitzpatrick, D. Gontier,
              H. Lee and H. Zhang, Minnaert resonances for acoustic waves in bubbly media, {\em Ann. Inst. H. Poincar\'{e} C Anal. Non Lin\'{e}aire \bf{7}} (2018), 1975--1998.

\bibitem{AFZ-17}{H. Ammari, B. Fitzpatrick, D. Gontier, H. Lee and H. Zhang}, A mathematical and numerical framework for bubble meta-screens, {\em SIAM J. Appl. Math. \bf{77}} (2017), 1827--1850.

\bibitem{AK-09} H. Ammari, H. Kang and H. Lee, {\em Layer potential techniques in spectral analysis}, American Mathematical Society, Providence, RI, 2009.

\bibitem{APL-22} H. Ammari, P. Millien and Alice L. Vanel, Modal approximation for strictly convex plasmonic resonators in the time domain: the {M}axwell's equations, {\em J. Differential Equations. \bf{309}} (2022), 676--703.

\bibitem{AZ-17} H. Ammari and H. Zhang, Effective medium theory for acoustic waves in bubbly fluids near
Minnaert resonant frequency, {\em SIAM J. Math. Anal. \bf{49}} (2017), 3252–3276.

\bibitem {BMV-21} L. Baldassari, P. Millien and Alice L. Vanel, Modal approximation for plasmonic resonators in the time domain: the scalar case, {\em Partial Differ. Equ. Appl. \bf{2}} (2021), Paper No. 46, 40.

\bibitem{C-M-P-T-1}
R. Caflisch, M. Miksis, G. Papanicolaou and L. Ting, 
Effective equations for wave propagation in a bubbly liquid, {\em J. Fluid Mec. \bf{153}} (1985), 259--273.

\bibitem{C-M-P-T-2}
R. Caflisch, M. Miksis, G. Papanicolaou and L. Ting, 
Wave propagation in bubbly liquids at finite volume fraction, 
{\em J. Fluid Mec. \bf{160}} (1986), 1--14. 

{\color{HW}\bibitem{P_0}
D.\,C. Calvo, A.\,L. Thangawng, C.\,N. Layman,
{\em Low-frequency resonance of an oblate spheroidal cavity in a soft elastic medium,
J. Acoust. Soc. Am. \bf{132}} (2012), EL1--EL7.}

\bibitem{CPV-99} F. Cardoso, G. Popov and G. Vodev, Distribution of resonances and local energy decay in the transmission problem. {II}, {\em Math. Res. Lett. \bf{6}} (1999), {377--396}.

\bibitem{DGS-21}
A. Dabrowski, A. Ghandriche and M. Sini,
Mathematical analysis of the acoustic imaging modality using bubbles as contrast agents at nearly resonating frequencies, {\em Inverse Probl. Imaging \bf{15}} (2021), 555-597.

{\color{HW}\bibitem{P_1}
M. Devaud, T. Hocquet, J.\,C. Bacri, V. Leroy,
{\em The Minnaert bubble: an acoustic approach,
Eur. J. Phys. \bf{29}} (2008), 1263.}

\bibitem{D-74} G. Doetsch,
     {\em Introduction to the theory and application of the {L}aplace transformation}, {Springer-Verlag, New York-Heidelberg}, {1974}.

\bibitem{DM} S. Dyatlov and M. Zworski, {\em Mathematical theory of scattering resonances}, American Mathematical Society, Providence, RI, {2019}.

{\color{HW}\bibitem{E10} L. C. Evans,
{\em Partial Differential Equations} (2nd edn), American Mathematical Society, Providence, RI, 2010.}

\bibitem{FH} F. Feppon and H. Ammari, Modal decompositions and point scatterer approximations near the {M}innaert resonance frequencies, {\em Stud. Appl. Math. \bf{149}} (2022), {164--229}.

\bibitem{J19} J. Galkowski, The quantum {S}abine law for resonances in transmission
              problems, {\em Pure Appl. Anal. \bf {1}} (2019) 27--100.

{\color{HW}
\bibitem{P_2}
V. Galstyan, O.\,S. Pak, H.\,A. Stone,
{\em A note on the breathing mode of an elastic sphere in Newtonian and complex fluids,
Phys. Fluids \bf{27}} (2015), 032001.}

{\color{HW}
\bibitem{P_3}
P.\,A. Hwang, W.\,J. Teague,
{\em Low-frequency resonant scattering of bubble clouds,
J. Atmos. Oceanic Technol. \bf{17}} (2000), 847--853.
}
            
\bibitem{RT} R. Igor and T. Tao,
Effective limiting absorption principles, and applications, {\em Comm. Math. Phys. \bf{333}} (2015), {1--95}.

\bibitem{KK-12} A. Komech and E. Kopylova, {\em Dispersion decay and scattering theory}, John Wiley \& Sons, Inc., Hoboken, NJ, 2012.

{\color{HW}
\bibitem{P_4}
V. Leroy, A. Bretagne, M. Fink, A. Tourin, H. Willaime, B. Tabeling, 
{\em Design and characterization of bubble phononic crystals, 
Appl. Phys. Lett. \bf{95}} (2009), 171904.

\bibitem{P_5}
V. Leroy, T. Strybulevych, M. Lanoy, F. Lemoult, A. Tourin, J.H. Page, 
{\em Supersorption of acoustic waves with bubble metascreens, 
Phys. Rev. B \bf{91}} (2015), 020301.

}

\bibitem{LS-04} L. Li and M. Sini, {Uniform Resolvent Estimates for Subwavelength Resonators: The Minnaert Bubble Case}, {\em arXiv:2406.02192v3.}

\bibitem{MPS}
A. Mantile, A. Posilicano and M. Sini,
On the origin of Minnaert resonances,
{\em J. Math. Pures Appl.(9) \bf{165}} (2022), 106–147.

\bibitem{MP} A. Mantile and A. Posilicano, The point scatterer approximation for wave dynamics, {\em Partial Differ. Equ. Appl. \bf(5)} (2024), 26pp.

\bibitem{WM-00} W. McLean, {\em Strongly elliptic systems and boundary integral equations}, Cambridge University Press, Cambridge, 2000.

\bibitem{MMS-18}
T. Meklachi, S. Moskow and J. C. Schotland, Asymptotic analysis of resonances of small volume high contrast linear and nonlinear scatterers, {\em J. Math. Phys. \bf{59}} (2018), 083502, 20.

\bibitem{MS-23} A. Mukherjee and M. Sini, Acoustic cavitation using resonating microbubbles: analysis in the time-domain, {\em SIAM J. Math. Anal. \bf{55}} (2023), 5575--5616.
 
\bibitem{MS-241} A. Mukherjee and M. Sini, Dispersive effective model in the time-domain for acoustic waves propagating in bubbly media, {\em SIAM J. Appl. Math. \bf{85}} (2025), 2508-2542.

\bibitem{MS-242} A. Mukherjee and M. Sini, Dispersive effective metasurface model for bubbly media, {\em arXiv:2412.14895.} {To appear in JDE.}

\bibitem{V-94} {V. Y. Raevski\u i}, Some properties of potential theory operators and their
              application to the investigation of the fundamental equation
              of electro- and magnetostatics, {\em Teoret. Mat. Fiz. \bf{100}} (1994), {323--331}.
              
\bibitem{PV-99} G. Popov and G. Vodev, Distribution of the resonances and local energy decay in the
              transmission problem, {\em Asymptot. Anal. \bf{19}}, (1999), (253--265).

\bibitem{SS-2024}
S. Senapati and M. Sini,
{Minnaert Frequency and Simultaneous Reconstruction of the Density, Bulk and Source in the Time-Domain Wave Equation}, {\em  Arch. Ration. Mech. Anal. \bf {249}} (2025), Paper No. 48, 58 pp.

\bibitem{S-01} P. Stefanov, Resonance expansions and {R}ayleigh waves, {\em Math. Res. Lett. \bf{8}} (2001), {107--124}.

\bibitem{T-Z-98}
S. H. Tang and M. Zworski, From quasimodes to resonances, {\em Math. Res. Lett. \bf{5}} (1998), 261--272.

\bibitem{T-Z-00}
S. H. Tang and M. Zworski, Resonance expansions of scattered waves, {\em Comm. Pure Appl. Math. \bf{53}} (2000), {1305--1334}.


\end{thebibliography}
\end{document}